\documentclass[11pt]{article}
\usepackage{amssymb,amsmath,amsfonts,amsthm,enumerate,color}
\usepackage[toc,page,title,titletoc,header]{appendix}
\usepackage{graphicx}
\usepackage{indentfirst}
\usepackage{multicol}
\usepackage{booktabs}
\usepackage{url}
\graphicspath{{./Figures/}}

\numberwithin{equation}{section}

\newtheorem{Theorem}{Theorem}[section]
\newtheorem{Lemma}[Theorem]{Lemma}
\newtheorem{proposition}[Theorem]{Proposition}

\newtheorem{algorithm}{Algorithm}
\newtheorem{remark}[Theorem]{Remark}
\numberwithin{equation}{section}
\newtheorem{example}{Example}[section]

\def \Vh0{\stackrel{\circ}{V}_h}

\newcommand{\lc}
{\mathrel{\raise2pt\hbox{${\mathop<\limits_{\raise1pt\hbox
{\mbox{$\sim$}}}}$}}}

\newcommand{\gc}
{\mathrel{\raise2pt\hbox{${\mathop>\limits_{\raise1pt\hbox{\mbox{$\sim$}}}}$}}}

\newcommand{\ec}
{\mathrel{\raise2pt\hbox{${\mathop=\limits_{\raise1pt\hbox{\mbox{$\sim$}}}}$}}}

\def\bb{\begin{equation}} \def\ee{\end{equation}}

\def\beqn{\begin{eqnarray}}  \def\eqn{\end{eqnarray}}

\def\beqnx{\begin{eqnarray*}} \def\eqnx{\end{eqnarray*}}

\def\bn{\begin{enumerate}} \def\en{\end{enumerate}}

\def\bd{\begin{description}} \def\ed{\end{description}}

\makeatletter

\newenvironment{figurehere}
  {\def\@captype{figure}}
  {}
\makeatother

\title{Algorithm for Hamilton-Jacobi equations in density space via a generalized Hopf formula} %from optimal transport and mean field games}
\date{}
\begin{document}

\author{Yat Tin Chow\footnote{Department of Mathematics, UCLA, Los Angeles, CA 90095-1555  (ytchow@math.ucla.edu, wcli@math.ucla.edu, sjo@math.ucla.edu, wotaoyin@math.ucla.edu). Research supported by AFOSR MURI proposal number 18RT0073, ONR grant: N00014-1-0444, N00014-16-1-2119, N00014-16-215-1, NSF grant ECCS-1462398 and DOE grant DE-SC00183838.
}
\and
Wuchen Li\footnotemark[1]
\and
Stanley Osher\footnotemark[1]
\and
Wotao Yin\footnotemark[1]
}
%\providecommand{\keywords}[1]{\textbf{\textit{Keywords---}} #1}
%\begin{keywords}{Hamilton-Jacobi equation in density space; Mean field games; Optimal transport; Generalized Hopf formula.}\end{keywords}

\maketitle
\begin{abstract}
We design fast numerical methods for Hamilton-Jacobi equations in density space (HJD), which arises in optimal transport and mean field games. We overcome the curse-of-infinite-dimensionality nature of HJD by proposing a generalized Hopf formula\footnote{We drop the word ``generalized'' in what follows.} in density space. The formula transfers optimal control problems in density space, which are constrained minimizations supported on both spatial and time variables, to optimization problems over only spatial variables. 
This transformation allows us to compute  HJD efficiently via multi-level approaches and coordinate descent methods. 
\end{abstract}


\section{Introduction}
In recent years, optimal control problems in density space have started to play vital roles in physics \cite{Nelson}, fluid dynamics \cite{BB} and probability \cite{PMFG1}. Two typical examples are mean field games (MFGs) \cite{CAIN, Hopf_MFG} and optimal transportation \cite{vil2008}. For these optimal control problems,  Hamilton-Jacobi equation in density space (HJD) determines the global information of the system \cite{WHJB, WHJB1}, which describes the time evolution of the optimal value in density space. More precisely, HJD refers to the functional differential equation as follows:
Let $x\in X$, and $\rho(\cdot)\in \mathcal{P}(X)$ represent the probability density space supported on $X$.  Let $U\colon [0, \infty)\times \mathcal{P}(X)\rightarrow \mathbb{R}$ be the value function. Consider
\begin{equation*}
\begin{cases}
\partial_s U(s,\rho) + \mathcal{H}(\rho, \delta_{\rho} U) = 0 \\
U(0,\rho)= G(\rho),
\end{cases}
\end{equation*}
where $\delta_\rho$ is the $L^2$ first variation w.r.t. $\rho$ and $\mathcal{H}$ represents the total Hamiltonian function in $\mathcal{P}(X)$:
\begin{equation*}
\mathcal{H}(\rho,\delta_{\rho(x)} U):=\int_{X} H(x, \nabla_x\delta_{\rho(x)}U)\rho(x)dx +F(\rho),
\end{equation*}
with the given Hamiltonian function $H$ on $X$. Here, $F$, $G\colon \mathcal{P}(X)\rightarrow \mathbb{R}$ are given interaction potential and initial cost functional in density space, respectively. 

In applications, HJD has  been shown very effective at modeling population differential games, also known as MFGs, which study strategic dynamical interactions in large populations by extending finite players' differential games. This setting provides powerful tools for modeling macro-economics, stock markets, and wealth distribution \cite{MFGA}. In this setting, a Nash equilibrium (NE) describes a status in which no individual player in the population is willing to change his/her strategy unilaterally. A widely-studied special class of MFG is the potential game \cite{PMFG}, where all players face the same cost function or potential, and every player minimizes this potential. This amounts to solving an optimal control problem in density space. In this case, a NE refers to the characteristics of HJD, which form a PDE system consisting of continuity equation and Hamilton-Jacobi equation in $X$. These two equations represent the dynamical evolutions of the population density and the cost value, respectively. 

Despite the importance of HJD, solving it numerically is not a simple task. It is known that computing Hamilton-Jacobi equations using a grid in a dimension greater than or equal to three is difficult.  The cost increases exponentially with the dimension, which is known as the curse of dimensionality \cite{JO}. HJD is even harder to compute since it involves  an infinite-dimensional functional PDE. In this paper, expanding the ideas in \cite{Hopf_Lax_4, Hopf_Lax_2, Hopf_Lax_3, JO}, we overcome the curse of infinite dimensionality in HJD by exploiting a Hopf formula in density space. 
This approach considers a particular primal-dual formulation associated with the optimal control problem in density space. Specifically, the Hopf formula is given as 
\begin{eqnarray*}
U(t,\rho) 
&:=& \sup_{\Phi_t} \, \Bigg \{  \int_X 
 \rho_t \Phi_t  dx -   \int_0^t \left(  F(\rho_s) - \int_X 
 \rho_s\delta_{\rho_s} F(\rho_s)dx \right) ds -  G^*(  \Phi_0)   :  \notag \\
& &
\hspace{2.5cm} \begin{matrix} 
&\partial_s \rho_s = \delta_{\Phi_s} \mathcal{H} (\rho_s, \Phi_s),\quad \partial_s \Phi_s = - \delta_{\rho_s} \mathcal{H} (\rho_s, \Phi_s)\,\\
&   \rho(x,t) = \rho_t(x),\quad \Phi(x,t) = \Phi_t(x) \\
\end{matrix}
 \Bigg\},
\end{eqnarray*}
where $\Phi_0(x)=\Phi(0,x)$ and 
$$G^*(\Phi_0):=\sup_{\rho_0\in\mathcal{P}(X)}~\int_X\rho_0\Phi_0dx -G(\rho_0).$$
%and  $\delta$ is the $L^2$ first variation. 
We further discretize the above variational problem following the same discretization as in optimal transport on graphs \cite{Li2, Li-SE, Li1, Li4, LiFisher}. We then apply a multi-level block stochastic gradient descent method to optimize the discretized problem.

In the literature of numerical methods for potential MFGs are seminal works of Achdou, Camilli, and Dolcetta \cite{MFG1, MFG2, MFG}. Their approaches utilize the primal-dual structure of the optimal control formulation,  simplifying it by a Legendre transform and applying Newton's method to the resulting saddle point system. Different from their approaches, we focus on solving the dual problem, in which the optimal control problem is an optimization problem over the terminal adjoint state $\Phi(x): = \Phi(x,t)$, satisfying the MFG system. Since this is a functional of a single variable, many optimization techniques for high-dimensional problems can be applied, for example, coordinate gradient descent methods. Also, numerical methods for special cases of potential games were introduced in \cite{NN}. They transform the optimal control problem into a regularized linear program. Unlike these methods, our methods can be applied to general Lagrangians for optimal control problems in density space. Yet another well-known line of research focuses on  stationary MFG systems \cite{BC1, SMFG}, for which proximal splitting methods have been used. They are different from our focus on time-dependent MFGs. 

{ The Hopf maximization principle gives us an optimal balance between the indirect method (Pontryagin's maximum principle), e.g. the well-known MFG system (49)-(51) below in \cite{Hopf_MFG}, and the direct method (optimization over the spaces of curves), e.g. the primal-dual formulation in \cite{MFG1, MFG2, MFG} and Hopf formula (57)-(59) in \cite{Hopf_MFG}. This balance leads to computational efficiency.  
There are several existing formulations for solving HJD numerically: 
 (i) the original formulation in \eqref{2}-\eqref{211} or its resulting (primal) Lagrangian  formulation, (ii) the intermediate primal-dual formulation in \cite{MFG1, MFG2, MFG}, (iii) the dual formulation (Hopf formula) \eqref{lala_hopf} in (57)-(59) in \cite{Hopf_MFG},  (iv) the resulting KKT optimality condition \eqref{1} (the MFG system (49)-(51) in \cite{Hopf_MFG}), and (v) the proposed Hopf formulation \eqref{Smart} in this paper. Under suitable conditions, the {five} formulations are equivalent, but their effects on computation are different. Formulations (i), (ii), and (iii)  involve a large number of variables, which lead to high complexities on problems with high dimensions and long time intervals;  (iv) is a forward-backward system that needs different numerical methods.} Our approach (v) is a balance between the indirect and direct methods and reduces the number of variables to a single terminal adjoint state $\Phi_t(x) := \Phi(x, t)$.  The memory requirement is thus greatly reduced in our approach. On the other hand, we keep a variable and a functional such that our algorithm produces a descending sequence  that  converges to a local minimum. 

We utilize the coordinate descent method, which avoids the difficulties coming from a nonsmooth functional.
We remark that the proposed approach can handle Hamiltonians of homogeneous degree $1$, which can be used as a mean-field level set approach for the reachability problem. 
Moreover, we choose the Hopf formulation to handle the case where the Hamiltonian $H$ is non-convex.  {We propose to check the computed limit (i.e., whether it is a global minimum) via the condition $\Phi(x,0) \in \partial G (\rho (x,0))$.} 

The rest of this paper is organized as follows. In Section 2, we briefly review potential MFGs and related HJD and formally derive the Hopf formula in density space.   
We also propose a rigorous approach on discrete grid approximations of optimal control problems and show the validity of the Hopf formula under proper assumptions. In Section 4, we design a fast multi-level random coordinate descent method for solving the discrete Hopf formula that we obtained in Section 3.  Several numerical examples are presented in Section 5 to illustrate the effectiveness of the proposed algorithm.
\section{Hopf formula in mean field games}
In this section, we briefly review potential MFGs. They are related to optimal control problems in density space, which induce Hamilton-Jacobi equations in density space. We  propose the Hopf formula in density space for subsequent numerical computation.  
\subsection{Potential mean field games}
 Consider a differential game played by one population, which contains countably infinitely many agents. Each agent selects a pure strategy from a strategy set $X$, which is a $d$-dimensional torus. The aggregated state of the population can be described by the population state $\rho(x)\in\mathcal{P}(X)=\big\{\rho(\cdot)\colon \int_{X}\rho(x)dx=1,~\rho(x)\geq 0\big\}$, where $\rho(x)$ represents the population density of players choosing strategy $x\in X$. 
The game assumes that each player�'s cost is independent of his/her identity (autonomous game). In a differential game, each agent plays the game dynamically facing the same Lagrangian $L\colon X\times TX\rightarrow \mathbb{R}$, where $TX$ represents the tangent space of $X$. The term ``mean field'' makes sense when each player's potential energy $f$ and terminal cost $g$ rely on  mean-field quantities of all other players' choices, mathematically written as $f, g\colon X\times \mathcal{P}(X)\rightarrow \mathbb{R}$.

The Nash equilibrium (NE) describes a status in which no player in population is willing to change his/her strategy unilaterally. In a MFG, it is represented as a primal-dual dynamical system:   
\beqn
\begin{cases}
&\partial_s \rho(x,s)+ \nabla_x \cdot ( \rho(x,s) \, D_p H(x, \nabla_x \Phi (x,s)))=0 \,\\
&\partial_s \Phi(x,s) + H (x, \nabla_x \Phi (x,s))  + f (x,\rho(\cdot,s))=0\,\\
&\rho(x,t) = {\rho}(x),\quad \Phi(x,0) = g(x, \rho(\cdot,0)),\\
\end{cases}
\label{1}
\eqn
where the Hamiltonian  $H$ is defined as \begin{equation*}
H(x,p) := \sup_{v\in TX}~\langle v , p  \rangle - L(x,v). \,
\label{well}
\end{equation*}
Here $H$ relates to the Lagrangian $L$ through a Legendre transform in $v$. And $\rho(s,\cdot)$ represents the population state at time $s$ satisfying the continuity equation 
while $\Phi(s,\cdot)$ governs the velocity of population according to the Hamilton-Jacobi equation.

A game is called a potential game when there exists a differentiable potential energy $F\colon \mathcal{P}(X)\rightarrow \mathbb{R}$ and terminal cost $G\colon \mathcal{P}(X)\rightarrow \mathbb{R}$ such that 
\begin{equation*}
{\delta_{\rho(x)}}F(\rho)=f(x,\rho),\quad {\delta_{\rho(x)}}G(\rho)=g(x,\rho),
\end{equation*}
where ${\delta_{\rho(x)}}$ is the $L^2$ first variation operator. The above definition represents that the incentives of all the players can be globally modeled by a functional called the potential~\cite{PMFG1}. In this case, the game is modeled as the following optimal control problem in density space:
\begin{subequations}\label{2ab}
\begin{equation}\label{2}
\inf_{\rho, v}\quad  \left \{ \int_0^t \big[\int_X L( x,  v(x,s) )   \rho(x,s) \, dx-F(\rho(\cdot,s))\big]ds  + G(  \rho(\cdot,0) ) \right \},
\end{equation}
where the infimum is taken among all vector fields $ v(x,s)$ and densities $\rho(x,s)$ subject to the continuity equation
\beqn
\begin{cases}
\frac{\partial}{\partial s}\rho(x,s) + \nabla \cdot (   \rho(x,s)   v(x,s) )=0, \quad  0\leq s \leq t  \,, \\
 \rho(x,t) =  { \rho }(x)  \,.&
\end{cases}
\label{211}
\eqn
\end{subequations}

It can be shown that, under suitable conditions of $L$, $F$, $G$, NEs are minimizers of potential games. In other words, every NE \eqref{1} satisfies the Euler-Lagrange equation (Karush-Kuhn-Tucker conditions) of the optimal control problem \eqref{2ab}. Let $\mathcal{H}(\rho,\Phi)$ denote the total Hamiltonian defined over the primal-dual pair $(\rho,\Phi)$:
\beqnx
\mathcal{H}(\rho,\Phi) := \int_X \rho(x) H(x, \nabla_x \Phi (x) ) \,dx +F(\rho(\cdot)), 
\eqnx
where $\delta_\rho$, $\delta_\Phi$ are $L_2$ first variations w.r.t. $\rho$ and $\Phi$. 
Then, NE \eqref{1} is given as  
\beqnx
\begin{cases}
&\partial_s \rho_s = \delta_{\Phi_s} \mathcal{H} (\rho_s, \Phi_s),\quad \partial_s \Phi_s= - \delta_{\rho_s} \mathcal{H} (\rho_s, \Phi_s)\,\\
&\rho_t={\rho}(x), \quad\Phi_0= \delta_{\rho_0}G(\rho_0).
\end{cases}
\eqnx

The time evolution of the minimal value in optimal control satisfies the Hamilton-Jacobi equation. In the case of density space, the optimal value function in \eqref{2} is denoted by $U\colon [0,+\infty)\times \mathcal{P}(X)\rightarrow \mathbb{R}$. As shown in \cite{WHJB,WHJB1}, $U$ satisfies the Hamilton-Jacobi equation in density space
\begin{equation*}
\begin{cases}
&\partial_s U(s,\rho(\cdot)) + \mathcal{H}(\rho(\cdot), \delta_\rho U) = 0 \\
&U(0,\rho)= G(\rho) \,.
\end{cases}
\end{equation*}
Here, HJD is a functional partial differential equation. If $U$ is solved, then its characteristics in density space, i.e. $(\rho, \Phi)$, are known. In particular, 
$\Phi(t,x)={\delta_{\rho(x)}}U(t,\rho)$. Thus, NE \eqref{1} is found. Next, we shall design a fast numerical algorithm for HJD.

\subsection{Hopf formula in density space}
Our approach is based on a primal-dual reformulation of the optimal control problem \eqref{2ab}, which we call the Hopf formula. 
\begin{proposition}[Hopf formula in density space]
Assume the duality gap between the primal problem \eqref{2ab} and its dual problem is zero, then
\begin{eqnarray}
U(t,\rho) 
&:=& \sup_{\Phi} \, \Bigg \{  \int_X 
 \rho(x) \Phi(x)  dx -   \int_0^t \left(  F(\rho(\cdot,s)) - \int_X 
 \rho(s,x) \delta_{\rho_s}F(\rho(\cdot, s)) dx \right) ds -  G^*(  \Phi(\cdot,0)   )   :  \notag \\
& &
\quad \quad \quad \begin{matrix} 
 \partial_s  \rho(x,s)  + \nabla \cdot (  \rho(x,s) D_p H(x, \nabla \Phi(x,s)  )  )= 0 \notag ~\\
 \partial_s  \Phi(x,s)  +  H(x, \nabla \Phi(x,s)  )+\delta_{\rho_s(x)}F(\rho(\cdot, s))=0\\
  \rho(x,t) = \rho(x), \quad \Phi(x,t) = \Phi(x)
\end{matrix}
 \Bigg\}  
\\ \label{lala_hopf}
\end{eqnarray}
where 
\begin{equation*}
G^*(\Phi(\cdot, 0)):=\sup_{\rho(\cdot, 0)\in \mathcal{P}(X)}\Big\{ G(\rho(\cdot,0) )-\int_X\rho(x,0)\Phi(x,0)dx\Big\}.
\end{equation*}
\end{proposition}

\begin{proof}[Formal derivation]
We first define the flux function $m(s,x):=\rho(s,x)v(s,x)$ in \eqref{2ab}. Thus problem \eqref{2ab} takes the form 
\begin{equation*}
U(t,\rho):=\inf_{\rho, v}\quad  \left \{ \int_0^t \big[\int_X L( x, \frac{m(x,s)}{\rho(x,s)} )   \rho(x,s) \, dx-F(\rho(\cdot,s))\big]ds  + G(  \rho(\cdot,0) ) \right \},
\end{equation*}
where the infimum is taken among all flux functions $m(x,s)$ and densities $\rho(x,s)$ subject to
\begin{equation*}
\begin{cases}
\frac{\partial}{\partial s}\rho(x,s) + \nabla \cdot m(x,s)=0, \quad  0\leq s \leq t  \,, \\
 \rho(x,t) =  { \rho }(x).&
\end{cases}
\end{equation*}

Next, we compute the dual of the optimal control problem \eqref{2ab}. Assume that, under suitable assumptions of $F$, $G$, $L$, the duality gap of optimal control problem \eqref{2ab} is zero. Hence we can switch ``inf'' and ``sup'' signs in our derivations. Let the Lagrange multiplier of continuity equation \eqref{211} be denoted by $\Phi(x,s)$. The optimal control problem \eqref{2ab} becomes
\begin{eqnarray*}
U(t,\rho) &=& \inf_{m(\cdot, s) , \rho(\cdot, s) ,  \rho(\cdot, t) = \rho  } \sup_{\Phi(\cdot, t)} \, \Bigg \{ \int_0^t  \int_X L\left( x, \frac{m(x,s)}{\rho(x,s)}\right) \rho(x,s) \,dx  ds  -  \int_0^t  F( \rho(\cdot,s)) ds +  G(\rho(\cdot,0) ) \notag \\
& &\hspace{3.5cm}  + \int_0^t \int_X \left( \partial_s \rho(x,s)+ \nabla \cdot m(x,s) \right) \Phi(x,s) dx  ds \Bigg\}  \notag \\
&= &\sup_{\Phi(\cdot, s)} \inf_{m(\cdot, s) , \rho(\cdot, s) ,  \rho(\cdot, t) = \rho } \, \Bigg \{ \int_0^t  \int_X L\left( x, \frac{m(x,s)}{\rho(x,s)}\right) \rho(x,s) \,dx  ds  -  \int_0^t  F( \rho(\cdot,s)) ds +  G(\rho(\cdot,0) )  \notag \\
& & \hspace{3.5cm} + \int_0^t \int_X \left( \partial_s \rho(x,s)+ \nabla \cdot m(x,s) \right) \Phi(x,s) dx  ds \Bigg\} \notag \\
&=&\sup_{\Phi(\cdot, s)} \inf_{m(\cdot, s) , \rho(\cdot, s) ,  \rho(\cdot, t) = \rho } \, \Bigg \{ \int_0^t  \int_X\Big[ L\left( x, \frac{m(x,s)}{\rho(x,s)}\right)- \frac{m(x,s)}{\rho(x,s)}\cdot\nabla \Phi(x,s)  \Big]\rho(x,s) \,dx  ds\\
 & &\hspace{3.5cm} -  \int_0^t  F( \rho(\cdot,s)) ds +  G(\rho(\cdot,0) ) + \int_0^t \int_X \partial_s \rho(x,s) \Phi(x,s) dx  ds \Bigg\} \notag \\
 &=& \sup_{\Phi(\cdot, s)} \inf_{ \rho(\cdot, s) ,  \rho(\cdot, t) = \rho } \, \Bigg \{ -  \int_0^t   \int_X  \rho(x,s)    H( x,  \nabla \Phi(x,s)  ) \,dx ds  -  \int_0^t  F( \rho(\cdot,s)) ds\\
 & & \hspace{3.5cm} +G(\rho(\cdot,0) ) + \int_0^t \int_X \partial_s \rho(x,s) \Phi(x,s)  dx ds \Bigg\},
\end{eqnarray*}
where the third equality is given by integration by parts, and the fourth equality follows by the Legendre transform in the third equality, i.e., with $v(x,s):=\frac{m(x,s)}{\rho(x,s)}$, $$ H(x,\nabla\Phi)=\sup_{v\in TX}~\nabla\Phi\cdot v-L(x,v).$$
%%In above formula, by integration by parts w.r.t $s$, we have
%%\begin{equation*}
%%\int_0^t \int_X \partial_s \rho(x,s) \Phi(x,s)  dx ds=
%%\end{equation*}
By integration by parts w.r.t. $s$ for the functional $\int_0^t \int_X \partial_s \rho(x,s) \Phi(x,s)  dx ds$, we obtain
\begin{eqnarray*}
 U(t,\rho) &= &  \sup_{\Phi(\cdot, s)} \inf_{ \rho(\cdot, s) ,  \rho(\cdot, t) = \rho } \, \Bigg \{ -  \int_0^t   \int_X  \rho(x,s)    H( x,  \nabla \Phi(x,s)  ) \,dx ds  -  \int_0^t  F( \rho(\cdot,s)) ds\\
& & \hspace{3.1cm} +G(\rho(\cdot,0) )-\int_X\rho(x,0)\Phi(x,0)dx\\
& &\hspace{3.1cm}+\int_X\rho(x,t)\Phi(x,t)dx- \int_0^t \int_X  \rho(x,s) \partial_s\Phi(x,s)  dx ds \Bigg\}.
%  \label{new}
\end{eqnarray*}
Then,
\begin{eqnarray*}
 U(t,\rho)&= & \sup_{\Phi } \sup_{\Phi(\cdot, s), \Phi(\cdot, t) = \Phi} \inf_{ \rho(\cdot, s) , \rho(\cdot, t) = \rho } \, \Bigg \{ -   \int_0^t  \int_X  \rho(x,s)    H( x,  \nabla \Phi(x,s)) \,dx ds-  \int_0^t  F( \rho(\cdot,s)) ds    \\ 
& & \hspace{3.5cm} -  G^*(  \Phi(\cdot,0)   ) + \int_X 
 \rho(x,t) \Phi(x,t)  dx   - \int_0^t \int_X 
 \rho(x,s) \partial_s \Phi(x,s)  dx ds \Bigg\}.
 \label{U}
\end{eqnarray*}

We optimize the above formula w.r.t. $\rho(x,s)$ and $\phi(x,s)$. Suppose for a fixed $\Phi = \Phi(t,\cdot)$, the saddle point problem
\begin{eqnarray*}
 &  & \sup_{\Phi(\cdot, s), \Phi(\cdot, t) = \Phi} \inf_{ \rho(\cdot, s) , \rho(\cdot, t) = \rho } \, \Bigg \{ -   \int_0^t  \int_X  \rho(x,s)    H( x,  \nabla \Phi(x,s)) \,dx ds - \int_0^t \int_X 
 \rho(x,s) \partial_s \Phi(x,s)  dx ds    \\ 
& & \hspace{4.2cm}  -  \int_0^t  F( \rho(\cdot,s)) ds -  G^*(  \Phi(\cdot,0)   ) + \int_X 
 \rho(x,t) \Phi(x,t)  dx   \Bigg\} 
\end{eqnarray*}
has a unique solution. It is simple to check that this saddle point satisfies \eqref{1}. Substituting \eqref{1} into \eqref{U}, we derive the Hopf formula \eqref{lala_hopf}.
\end{proof}
%We notice that the formula \eqref{lala_hopf} seem to {go beyond} these assumptions and give legitimate numerical results, as shown in section \ref{sec5}. 
%%Although it is not the scope of this paper to search for more precise conditions for which the above shall hold.
% it is an interesting topic, and it will be great to explore this direction.
%{\color{green} In the same spirit, in the discrete version, a conjecture is given.}

Equation \eqref{lala_hopf} can be viewed as the Hopf formula of the optimal control problem \eqref{2ab}.  This goes in line with \cite{Hopf_Lax_4, Hopf_Lax_2,Hopf_Lax_3}.  That means that \eqref{lala_hopf}  contains  an optimization problem and uses a minimal number of unknown variables. We develop fast algorithms based on this formula. 
%\noindent Concerning the above conclusion from the inf-sup inequality, the following remarks are in force:
%%\begin{remark}
%%\end{remark}

\begin{remark}
When $(-F)$, $G$ and $L$ are convex and smooth, the discrete formulation of the primal dual formulation of \eqref{2ab} has been used for numerical computation in \cite{MFG1,MFG2, MFG} along with Newton's method.  We, on the other hand, prefer sticking to the formulation \eqref{lala_hopf} since we hope to solve for non-convex $(-F)$, $G$ and $L$ with nonsmooth $H(x,p)$, while keeping a minimal number of variables. In addition, the Hopf formula \eqref{lala_hopf} can be further simplified into
\beqn
U(t,\rho) 
&=& \sup_{\Phi(\cdot, s)} \, \Bigg \{  \int_X  \rho(x) \Phi(x)  dx -   \int_0^t F^* (\Phi(\cdot,s)) ds -  G^*(  \Phi(\cdot,0)   )  \Bigg\},  \label{lala_hopf1}
\eqn
which coincides with (57)-(59) in \cite{Hopf_MFG}.  However, the formulation \eqref{lala_hopf1}, similar to the Lagrangian formulation \eqref{2ab}, has more independent variables after discretization of $\Phi(x,s)$. Hence, it is not ideal for numerical computation.
\end{remark}

\begin{remark}\label{rm1}
The Hopf formula \eqref{lala_hopf} is also related to the dual formulation of an optimal transport problem. 
When $F(\rho) = 0$, the primal equation in \eqref{lala_hopf} can be dropped. Let $\rho = \rho_1$ in
\beqnx
U(t,\rho_1) 
 =  \sup_{\Phi} \, \left \{  \int_X 
 \rho(x) \Phi(x)  dx -  G^*(  \Phi(\cdot,0)   )   :   \partial_s  \Phi(x,s)  +  H( x,  \nabla \Phi(x,s)  ) = 0, \, \Phi(x,t) = \Phi(x) \right\}. %\\
\eqnx
This is precisely the Kantorovich dual of the optimal transport problem from $\rho_0$ to $\rho_1$ when we choose $G(\rho) = \iota_{\rho_0} (\rho)$ and let $t =1$. Here, for a set $A$ and a subset $B \subset A$, the indicator function $\iota_B : A \rightarrow \{ 0,\infty \}$ is defined as
\beqnx
\iota_B (x) = 
\begin{cases}
0 & \text{ if } x\in B \\
\infty & \text{ if } x\notin B
\end{cases}\ .
\eqnx
If $B = \{x_0\}$ is a singleton, we write $\iota_{x_0} (x) := \iota_{\{x_0\}} (x)$, abusing the notation.
\end{remark}
\begin{remark}
As in remark \ref{rm1}, our Hopf formula \eqref{lala_hopf} reduces to Monge-Kantorovich duality of the optimal transport with a specific choices of $F$, $G$ and $t$.  Moreover, the simplified formula can be used to compute the proximal map of $p$-Wasserstein distance in the $L^2$ sense. Let us recall the connection between optimal transport and \eqref{2ab}. The optimal transport problem can be formulated in an optimal control problem in density space, known as the Benamou-Brenier formula \cite{vil2008}. Consider $L(x, q) = \frac{1}{2}|q|_2^p$. Then, \beqnx
 U(1,\rho_1)  
&\begin{matrix} \text{(Definition)} \\ = \end{matrix}& \inf_{ v(\cdot, s)  , \rho(\cdot, s) } \, \Bigg\{ \int_0^1 \int_X L( v(x,s)) \rho(x,s) \,dx ds  +  G(\rho (x,0) ) : \\
& & \hspace{2cm} \partial_s \rho+ \nabla \cdot ( \rho v ) = 0 , \,  \rho(1) =\rho_1 \Bigg\} \\
&\begin{matrix} \text{(Benamou-Brenier)} \\ = \end{matrix}& \inf_{\rho_0 } \, \Bigg\{ \big(W_p ( \rho_0,  \rho_1)\big)^p +  G(\rho_0 ) \Bigg\} \\
&\begin{matrix} \text{(Kantorovich duality)} \\ = \end{matrix}&  \inf_{\rho_0 }  \sup_{\Phi_1} \, \Bigg\{  \int_{Y}  \Phi(y,1)  \rho_1(y) d y -  \int_{X}  \Phi(x,0)   \rho_0(x) d x +  G(\rho_0 ) : \\
& &\hspace{1.5cm} \partial_s  \Phi(x,s)  +  H( \nabla \Phi(x,s)  ) \leq 0, \, \Phi(x,1) = \Phi_1(x)
\Bigg \}  \\
&\begin{matrix} \text{(Convexity of $G,H$)} \\ = \end{matrix}&   \sup_{\Phi_1}   \, \Bigg\{  \int_{Y}  \Phi(y,1)  \rho_1(y) d y - G^*(  \Phi(\cdot,0) ) : \\
& & \quad \quad \partial_s  \Phi(x,s)  +  H( \nabla \Phi(x,s)  ) = 0, \, \Phi(x,1) = \Phi_1(x)
\Bigg \}, 
\eqnx
where $W_p( \rho_0,  \rho_1)$ is the $L^p$-Wasserstein metric which can be defined via the Benamou-Brenier formulation as follows:
\begin{equation*}
\big(W_p( \rho_0,  \rho_1)\big)^p := \inf_{ v(\cdot, s)  , \rho(\cdot, s) } \, \Bigg\{ \int_0^1 \int_X L( v(x,s)) \rho(x,s) \,dx ds  \colon \partial_s \rho+ \nabla \cdot ( \rho v ) = 0 , \,  \rho(0)=\rho_0, \,\rho(1) =\rho_1\Bigg\}.
\end{equation*}

If one aims to consider a general optimization problem over $G$ regularized by $W_{p}^{p} $ as in\beqnx
\min_{\rho_0} \{ \beta W_{p}^{p} (\rho_0,\rho ) + G(\rho_0) \},
\eqnx
we can either apply the above formulation directly or apply a splitting method, in which we need the proximal maps of $ W_{p}^{p} $ (in $L^2$ sense) as
\beqnx
 \text{Prox}_{\beta W_{p}^{p} (\cdot,\rho ) } (\rho_1) 
&=& \mathop\mathrm{argmin}_{\rho_0} \, \Bigg \{
\beta W_{p}^{p} (\rho_0, \rho ) +
 \frac{1}{2}\| \rho_0  -\rho_1\|^2  \Bigg\}  \\
&=& \rho_1 -  \beta  \tilde{ \hat \Phi},
\eqnx
where
\beqnx
\tilde{ \hat \Phi} &:=& \mathop\mathrm{argmax}_{\tilde{\Phi}} \, \Bigg \{ \int_X \rho(x)  \tilde{\Phi}(x) dx - 
 \int_X \rho_1(x)  \Phi(\tilde{ \Phi}, 0,\cdot)dx + \frac{\beta}{2} \|  \Phi(\tilde{\Phi}, 0,\cdot)\|^2: \\
& & \hspace{1.8cm}  \partial_s \Phi + H (x, \nabla_x \Phi )  = 0 \,, \Phi(\tilde{\Phi}, t,\cdot) = \tilde{\Phi} \Bigg\}  \,.
\eqnx
\end{remark}

\section{Discretization and rigorous treatment}

In this section, we aim to give a rigorous treatment to the discrete spatial states in potential MFGs. 
Our spatial discretization follows the same  work on optimal transport on graphs as in \cite{Li4, LiFisher} and our proof follows  the ideas in \cite{Hopf_Lax_4}.

For illustrative purposes, we focus on the following special form of the Lagrangian:
$$L(x,v):=\sum_{i=1}^n L(v_i), $$
where $L :\mathbb{R}^1\rightarrow \mathbb{R}^1$ is a proper function, define the Hamiltonian $H~:~\mathbb{R}^1\rightarrow \mathbb{R}^1$ as
$$H(p) :=\sup_{v\in \mathbb{R}^1} \left\{ pv-L(v) \right\}.$$

Consider $G=(V, E)$ as a uniform toroidal graph with equal spacing $\Delta x=\frac{1}{M}$ in each dimension. Here, $V$ is a vertex set with $|V|=(M+1)^d$ nodes, and each node, $i=(i_k)_{k=1}^d\in V$, $1\leq k\leq d$, $0\leq i_k\leq n$, represents a cube with length $\Delta x$:
\[
C_i=\{
(x_1,\cdots, x_d)\in [0,1]^d\colon
|x_1-i_1\Delta x|\le \Delta x/2,\cdots,
|x_d-i_d\Delta x|\le \Delta x/2\}.
\]
Here $E$ is an edge set, where $i+\frac{e_v}{2}:=\textrm{edge}(i,i+e_v)$, and $e_v$ is a unit vector at $v$th column.

Define $$\rho_i := \int_{C_i} \rho(x) dx \in [0,1]$$ on each $i \in V$. Let the discrete flux function be $m:=(m_{i+\frac{e_v}{2}})_{i+\frac{e_v}{2}\in E}$,
where $m_{i+\frac{e_v}{2}}$ represents the discrete flux on the edge $i+\frac{e_v}{2}$, i.e.,
\[
m_{i+\frac{e_v}{2}} \approx \int_{C_{i+\frac{e_v}{2}}(x)}m_v(x)\;dx\ ,
\]
where $m(x)=(m_v(x))_{v=1}^d$ is the flux function in continuous space. 

Thus the discrete divergence operator is:
\begin{equation*}
\textrm{div}(m)|_i=\frac{1}{\Delta x}\sum_{v=1}^d (m_{i+\frac{1}{2}e_v}-m_{i-\frac{1}{2}e_v}).
\end{equation*}
The discretized cost functional forms
\begin{equation*}
  \mathcal{L}(m,\rho) :=  \sum_{i+\frac{e_v}{2}\in E} \tilde{L} \left( m_{i+\frac{1}{2}e_v} , \theta_{i+\frac{1}{2}e_v} \right) 
  \end{equation*}
where
\begin{equation*}
 \tilde{L} \left( m_{i+\frac{1}{2}e_v} , \theta_{i+\frac{1}{2}e_v} \right)  := 
  \begin{cases}
 L\left(\frac{m_{i+\frac{1}{2}e_v}}{ \theta_{i+\frac{1}{2}e_v}}\right) \theta_{i+\frac{1}{2}e_v} & \textrm{if $ \theta_{i+\frac{1}{2}e_v}>0$}\ ;\\
 0 & \textrm{if $ \theta_{i+\frac{1}{2}e_v}=0$ and  $m_{i+\frac{e_v}{2}}=0$\ ;}\\
 +\infty & \textrm{Otherwise}\ .
 \end{cases}
  \end{equation*}
and $ \theta_{i+\frac{1}{2}e_v}:=\frac{1}{2}(\rho_i+\rho_{i+e_v})$ is the discrete probability on the edge $i+\frac{e_v}{2}\in E$.

We further introduce a time discretization. The time interval $[0,1]$ is divided into $N$ intervals with endpoints $t_n=n \Delta t$, $\Delta t=\frac{1}{N}$, $l=0,1,\cdots, N$. Combining the above spatial discretization and a forward finite difference scheme on the time variable, we arrive at the following discrete optimal control problem:
\begin{subequations}\label{finite}
\begin{equation}
{\tilde U } (t,\rho):= \inf_{m, \rho}~\Big\{\sum_{n=1}^N \Delta t \, \mathcal{L}(m^n, \rho^n)-\sum_{n=1}^N \Delta t\,  F(\rho^n)+G(\rho^0)\Big\}
\end{equation}
where the minimizer is taken among $\{\rho\}_i^n$, $\{m\}^n_{i+\frac{e_v}{2}}$, such that for $n = 0,...,N-1$
\begin{equation}
\begin{cases}
\rho^{n+1}_i - \rho^n_i +  \Delta t\cdot\textrm{div}(m^{n+1})|_i  = 0 \,, \\
 \rho^{N}_i =   \rho_i  \ .&
\end{cases}
\end{equation}
\end{subequations}

We next derive the discrete Hopf formula for minimization \eqref{finite}. Denote $ \rho^{N}_i :=  \rho_i $ and $$ (\mathbf{m}, \mathbf{\rho},\mathbf{\Phi}) := 
\left( \{ m^n_{i+\frac{1}{2}e_v}  \}_{ n=0}^{N-1} , \{ \rho^n_i \}_{ n=0}^{N-1}  ,   \{ \Phi^n_i \}_{ n=1}^{N}  \right)
 \in \mathbb{R}^{|E| N}  \times  [0,1] ^{|V| N}  \times \mathbb{R} ^{|V| N}  \,.
 $$ 
 Hence by an application of Lagrange multiplier at \eqref{finite}, then we have
\beqnx
{\tilde U } (t,\{\rho_i\}) =\inf_{\{ m^n_{i+\frac{1}{2}e_v}  \}_{ n=0}^{N-1}  \in \mathbb{R}^{|E| N} \, , \, \{ \rho^n_i \}_{ n=0}^{N-1}   \in  [0,1] ^{|V| N}  }  \sup_{ \{ \Phi^n_i \}_{ n=1}^{N}   \in \mathbb{R} ^{|V| N}  } \, \mathcal{F} (\mathbf{m}, \mathbf{\rho},\mathbf{\Phi})
\eqnx
where
\beqn
\mathcal{F} (\mathbf{m}, \mathbf{\rho},\mathbf{\Phi}) &:= & 
 \sum_{n=0}^{N-1} \sum_{i+\frac{e_v}{2}\in E} \Delta t \,  { \mathcal{L} } \left( m_{i+\frac{1}{2}e_v}^{n+1} , \theta_{i+\frac{1}{2}e_v}^{n+1} \right)   - \sum_{n=0}^{N-1} \Delta t \,  F( \{\rho\}_i^{n+1} ) +  G( \{\rho\}_i^0 )  \\
& & + \sum_{n=0}^{N-1} \sum_{i \in V} \Phi_i^{n+1} \left(  \rho^{n+1}_i - \rho^n_i +  \Delta t\cdot\textrm{div}(m^{n+1})|_i  \right)  \ . 
\label{orz}
\eqn
For a rigorous treatment, we assume:\begin{enumerate}
\item[(A1)] 
The Lagrangian $L :\mathbb{R} \rightarrow \mathbb{R}$ is a proper, lower semi-continuous,  convex functional.

\item[(A2)] 
The Lagrangian $L :\mathbb{R} \rightarrow \mathbb{R}$ has the following properties:
\begin{itemize}
\item
for any fixed $x\neq 0$, $\lim_{y \rightarrow 0^+} L\left(\frac{x}{y}\right) y = \infty$;
\item
for any fixed $y$, the function $(x,y) \rightarrow L\left(\frac{x}{y}\right) y$ is equi-coercive (under parameter $y$) w.r.t $x$ in the following sense: for all $N >0$, there exists $K$ (independent of $y$)
\beqnx
\left |  L\left(\frac{x}{y}\right) y \right| \geq K 
\eqnx
whenever $|x| \geq N$. 

\end{itemize}

\item[(A3)] 
The functional $F :[0,1] ^{|V|} \rightarrow \mathbb{R}$ is a proper, upper semi-continuous, concave functional.

\item[(A4)] 
The functional $G :[0,1] ^{|V|} \rightarrow \mathbb{R}$ is proper, lower semi-continuous, and convex in $\{ \rho_i \}_{i=1}^{|V|} $.

\item[(A5)]
$ H:  \mathbb{R}^1\rightarrow \mathbb{R}^1 $ is in $C^2$, and $F :[0,1] ^{|V|} \rightarrow \mathbb{R}$ and $G :[0,1] ^{|V|} \rightarrow \mathbb{R}$ are in $C^2((0,1) ^{|V|})$.

\item[(A6)] 
Denote the Legendre transform of the function $ F:  [0,1]^{|V|} \rightarrow \mathbb{R}$ by $F^*$. Suppose $F^*$ is coercive, i.e.
\beqnx
| F^*(x) | \rightarrow \infty
\eqnx
as $x \rightarrow \infty$.

\item[(A7)] 
The derivative of the function $ F:  [0,1]^{|V|} \rightarrow \mathbb{R}^1 $ satisfies, for any $\bar \rho \in \{0,1\}^{|V|}$
\beqnx
| \partial_\rho F( \{ \rho _i \})  |^2_2 \rightarrow \infty
\eqnx
whenever $\{ \rho _i \}\rightarrow \bar \rho $. 

\end{enumerate}
Under the above assumptions, we introduce the discrete Hopf formula by the following theorem
\begin{Theorem}\label{thm}
If (A1)-(A7) holds, then the value function $ {\tilde U } (t,\{\rho_i\})$ in \eqref{finite} equals 
\begin{eqnarray}
 {\tilde U } (t,\{\rho_i\})=&  \sup_{\{ \Phi_i \}  \in \mathbb{R} ^{|V| }    }    \Bigg \{ 
 \sum_{i\in V} \Phi_i^{N}  \rho_i   - \sum_{n=1}^N \Delta t \left( F( \{\rho\}_i^n ) -  \sum_i [\nabla_\rho F ( \{\rho\}_i^n ) ]_i \rho^n_i \right)     -  G^*( \{\Phi\}_i^0 )   :  \notag \\
&\hspace{1.5cm}   \begin{matrix}
  \rho^{n}_i - \rho^{n-1}_i +  \Delta t \sum_{v=1}^d D_p H \left(\frac{1}{\Delta x} (\Phi_i^n-\Phi_{i+e_v}^n) \right)    \theta^n_{i+\frac{1}{2}e_v}  = 0  \\
    \Phi^{n+1}_i - \Phi^{n}_i +  \frac{ \Delta t }{4} \sum_{v=1}^d  H  \left(\frac{1}{\Delta x}(\Phi_i^n - \Phi_{i+e_v}^n) \right) + \Delta t[\nabla_\rho F ( \{\rho\}_i^n ) ]_i  = 0 \\
     \rho^{N}_i = \rho_i, \, \Phi_i^N   = \Phi_i 
  \end{matrix}  \Bigg\} 
\label{Smart}
\end{eqnarray}
\end{Theorem}

\begin{remark} 
We remark that if $( \{ \rho^n_i\} , \{ \Phi^n_i \} )$ are computed according to the constraints given in \eqref{Smart} for all $n=0,..,N-1$, then for each $n$, the numerical Hamiltonian  
$$ \mathbb{H}(\rho^{n}, \Phi^n)  =  \sum_{i+\frac{e_v}{2}\in E}   H  \left(\frac{1}{\Delta x}(\Phi_i^n - \Phi_{i+e_v}^n) \right)   \theta^n_{i+\frac{1}{2}e_v}+  F( \{\rho\}_i^n )  $$
is conserved, where we write $\rho^{n} := \{ \rho^{n}_i\}$ and $\Phi^n := \{ \Phi^n_i\}$.
\end{remark}
\begin{remark} 
If $F(\rho)=0$, \eqref{Smart} is an unstable scheme for initial value Hamilton-Jacobi equations. 
In computations, we handle it using a monotone scheme; see %by the Lax-Friedrichs viscosity in
section \ref{section4} (Remark \ref{rm10}) below.   
\end{remark}

\begin{remark}
We note in numerical examples in Section \ref{sec5} that our formula appears to be valid beyond the assumptions (A1)-(A7), e.g., in the case when $H$ is a nonsmooth, nonconvex Hamiltonian.   
The continuous analog of \eqref{Smart} is discussed and proposed in Section 2.2.
The minimal assumptions of validity for \eqref{Smart} to hold may be an interesting direction to explore, and some possibilities are discussed in, e.g., \cite{Hopf_Lax_4, Hopf_MFG, proofimportant}.
\end{remark}

We prove Theorem \ref{thm} by showing the following three lemmas. %We first have the following lemma:
\begin{Lemma}\label{lem1}
Assume (A1). Then, the functional $\mathcal{L} (m,\rho)$ is convex.
\end{Lemma}
\begin{proof}
We shall show that $\mathcal{L}$ is convex. Since 

\begin{equation*}
  \mathcal{L}(m,\rho) :=  \sum_{i+\frac{e_v}{2}\in E} \tilde{L} \left( m_{i+\frac{1}{2}e_v} , \theta_{i+\frac{1}{2}e_v} \right) 
  \end{equation*}
We only need to show that $ \tilde{L}$ is convex.  
%\begin{equation*}
% \tilde{L} \left( m_{i+\frac{1}{2}e_v} , \theta_{i+\frac{1}{2}e_v} \right)  := 
%  \begin{cases}
% L\left(\frac{m_{i+\frac{1}{2}e_v}}{ \theta_{i+\frac{1}{2}e_v}}\right) \theta_{i+\frac{1}{2}e_v} & \textrm{if $ \theta_{i+\frac{1}{2}e_v}>0$}\ ;\\
% 0 & \textrm{if $ \theta_{i+\frac{1}{2}e_v}=0$ and  $m_{i+\frac{e_v}{2}}=0$\ ;}\\
% +\infty & \textrm{otherwise}\ .
% \end{cases}
%  \end{equation*}
In other words, for $y>0$,  $L(\frac{x}{y})y$ is convex for $(x,y)$. In fact, that is true since 
$$ \textrm{Hess}\left(L\left(\frac{x}{y}\right)y\right)= L''\left(\frac{x}{y}\right) \begin{pmatrix} 
\frac{1}{y} & -\frac{x}{y^2} \\
-\frac{x}{y^2} & \frac{x^2}{y^3}
\end{pmatrix}\succeq 0\ .
$$
\end{proof}

We now proceed as in \cite{Hopf_Lax_4} and obtain the following lemma. This lemma is similar to the primal-dual formulation in \cite{MFG1,MFG2, MFG}, but we provide it here for the sake of completeness.

\begin{Lemma} \label{hahalolXD}
Write 
$$ (\mathbf{m}, \mathbf{\rho},\mathbf{\Phi}) := 
\left( \{ m^n_{i+\frac{1}{2}e_v}  \}_{ n=0}^{N-1} , \{ \rho^n_i \}_{ n=0}^{N-1}  ,   \{ \Phi^n_i \}_{ n=1}^{N}  \right)
 \in \mathbb{R}^{|E| N}  \times  [0,1] ^{|V| N}  \times \mathbb{R} ^{|V| N}  \,.
 $$
and
$$ ( \tilde{ \mathbf{\rho} },\mathbf{\Phi}) := 
\left( \{ \rho^n_i \}_{ n=0}^{N-1}  ,   \{ \Phi^n_i \}_{ n=1}^{N}  \right)
 \in  [0,1] ^{|V| N}  \times \mathbb{R} ^{|V| (N-1)}  \,.
 $$
and let 
$ \mathcal{F} (\mathbf{m}, \mathbf{\rho},\mathbf{\Phi}) $ given in \eqref{orz}, and
\beqnx
\tilde{\mathcal{F}} ( \tilde{ \mathbf{\rho} },\mathbf{\Phi})  &:=&  
- \sum_{n=1}^N \sum_{i+\frac{e_v}{2}\in E}   H  \left(\frac{1}{\Delta x}(\Phi_i^n - \Phi_{i+e_v}^n) \right)   \theta^n_{i+\frac{1}{2}e_v}
 \Delta t     - \sum_{n=1}^N \Delta t F( \{\rho\}_i^n )  -  G^*( \{\Phi\}_i^0 ) \\
& & \hspace{3cm}  + \sum_{n=1}^{N-1} \sum_{i \in V}\left(  \Phi_i^{n} -  \Phi_i^{n+1} \right)    \rho^{n}_i + \sum_{i \in V} \Phi_i^{N}  \rho_i   
\eqnx
If (A1), (A2), (A3), (A4) are satisfied, then 
%we have
%then we have
\begin{eqnarray}
\begin{split}
&\inf_{\{ m^n_{i+\frac{1}{2}e_v}  \}_{ n=0}^{N-1}  \in \mathbb{R}^{|E| N} \, , \, \{ \rho^n_i \}_{ n=0}^{N-1}   \in  [0,1] ^{|V| N}  }  \sup_{ \{ \Phi^n_i \}_{ n=1}^{N}   \in \mathbb{R} ^{|V| N}  } \, \mathcal{F} (\mathbf{m}, \mathbf{\rho},\mathbf{\Phi})\\
 =&  \sup_{ \{ \Phi^n_i \}_{ n=1}^{N}   \in \mathbb{R} ^{|V| N}  }   \inf_{\{ \rho^n_i \}_{ n=1}^{N-1}   \in  [0,1] ^{|V| (N-1)}  }\tilde{\mathcal{F}} ( \tilde{ \mathbf{\rho} },\mathbf{\Phi})   \,. 
\end{split}
\label{wawawawa} 
\end{eqnarray}
\end{Lemma}

\begin{proof}
In fact, from the equi-coercivity (A2), we find that there exists a closed and bounded interval $C \subset \mathbb{R}$ s.t.
\beqnx
\inf_{\{ m^n_{i+\frac{1}{2}e_v}  \}_{ n=0}^{N-1}  \in \mathbb{R}^{|E| N} \, , \, \{ \rho^n_i \}_{ n=0}^{N-1}   \in  [0,1] ^{|V| N}  }  \sup_{ \{ \Phi^n_i \}_{ n=1}^{N}   \in \mathbb{R} ^{|V| N}  } \, \mathcal{F} (\mathbf{m}, \mathbf{\rho},\mathbf{\Phi}) \\
=
\inf_{\{ m^n_{i+\frac{1}{2}e_v}  \}_{ n=0}^{N-1}  \in C^{|E| N} \, , \, \{ \rho^n_i \}_{ n=0}^{N-1}   \in  [0,1] ^{|V| N}  }  \sup_{ \{ \Phi^n_i \}_{ n=1}^{N}   \in \mathbb{R} ^{|V| N}  } \, \mathcal{F} (\mathbf{m}, \mathbf{\rho},\mathbf{\Phi})
\eqnx
Now that $ \mathcal{F} (\mathbf{m}, \mathbf{\rho},\mathbf{\Phi}) $ is lower-semicontinuous and quasi-convex w.r.t. $(\mathbf{m}, \mathbf{\rho} ) $ (from (A1), (A3) and (A4)) and upper-semicontinuous and quasi-concave w.r.t. $ \mathbf{\Phi} $ (from linearity), as well as $ C^{|E| N} \times [0,1] ^{|V| N} $, we have by an application of Sion's minimax theorem \cite{sion1,sion2} that
 \beqnx
\inf_{\{ m^n_{i+\frac{1}{2}e_v}  \}_{ n=0}^{N-1}  \in C^{|E| N} \, , \, \{ \rho^n_i \}_{ n=0}^{N-1}   \in  [0,1] ^{|V| N}  }  \sup_{ \{ \Phi^n_i \}_{ n=1}^{N}   \in \mathbb{R} ^{|V| N}  } \, \mathcal{F} (\mathbf{m}, \mathbf{\rho},\mathbf{\Phi})\\
=
\sup_{ \{ \Phi^n_i \}_{ n=1}^{N}   \in \mathbb{R} ^{|V| N}  }  \inf_{\{ m^n_{i+\frac{1}{2}e_v}  \}_{ n=0}^{N-1}  \in C^{|E| N} \, , \, \{ \rho^n_i \}_{ n=0}^{N-1}   \in  [0,1] ^{|V| N}  }  \, \mathcal{F} (\mathbf{m}, \mathbf{\rho},\mathbf{\Phi}) \\
=
\sup_{ \{ \Phi^n_i \}_{ n=1}^{N}   \in \mathbb{R} ^{|V| N}  }  \inf_{\{ m^n_{i+\frac{1}{2}e_v}  \}_{ n=0}^{N-1}  \in \mathbb{R}^{|E| N} \, , \, \{ \rho^n_i \}_{ n=0}^{N-1}   \in  [0,1] ^{|V| N}  }  \, \mathcal{F} (\mathbf{m}, \mathbf{\rho},\mathbf{\Phi})
\eqnx
where the last equality is again obtained by equi-coercivity in (A2).

Now let us fix $ ( \tilde{ \mathbf{\rho} },\mathbf{\Phi}) = (  \{ \rho^n_i \}_{ n=1}^{N-1} ,  \{ \Phi^n_i \}_{ n=1}^{N}  )  $, and consider the optimization
\beqnx
 \inf_{\{ m^n_{i+\frac{1}{2}e_v}  \}_{ n=0}^{N-1}  \in \mathbb{R}^{|E| N} \, , \, \{ \rho^0_i \}  \in  [0,1] ^{|V|}  }  \, \mathcal{F} (\mathbf{m}, \mathbf{\rho},\mathbf{\Phi}) \,.
\eqnx

We next derive its duality formula. Following the discrete integration by part, then
\beqnx
& & \inf_{\{ m^n_{i+\frac{1}{2}e_v}  \}_{ n=0}^{N-1}  \in \mathbb{R}^{|E| N} \, , \, \{ \rho^0_i \}  \in  [0,1] ^{|V|}  }  \, \mathcal{F} (\mathbf{m}, \mathbf{\rho},\mathbf{\Phi}) \\
&=&
\inf_{\{ m^n_{i+\frac{1}{2}e_v}  \}_{ n=0}^{N-1}  \in \mathbb{R}^{|E| N} \, , \, \{ \rho^0_i \}  \in  [0,1] ^{|V|}  }  \,
\Bigg\{ 
 \sum_{n=0}^{N-1} \sum_{i+\frac{e_v}{2}\in E} \Delta t \,  { \mathcal{L} } \left( m_{i+\frac{1}{2}e_v}^{n+1} , \theta_{i+\frac{1}{2}e_v}^{n+1} \right)   - \sum_{n=0}^{N-1} \Delta t \,  F( \{\rho\}_i^{n+1} ) +  G( \{\rho\}_i^0 )  \\
& &  \quad \quad  \quad \quad \quad  \quad \quad \quad  \quad \quad \quad  \quad \quad  + \sum_{n=0}^{N-1} \sum_{i \in V} \Phi_i^{n+1} \left(  \rho^{n+1}_i - \rho^n_i +  \Delta t\cdot\textrm{div}(m^{n+1})|_i  \right)  \Bigg \} \\
&=&
\inf_{ \{ \rho^0_i \}  \in  [0,1] ^{|V|}  } 
 \,
\Bigg\{ 
 \sum_{n=0}^{N-1} \sum_{i+\frac{e_v}{2}\in E} \Delta t \,
 \inf_{ m_{i+\frac{1}{2}e_v} } 
\Bigg\{  L    \left(\frac{ m^{n+1}_{i+\frac{1}{2}e_v} }{ \theta^{n+1}_{i+\frac{1}{2}e_v}} \right)  \theta^{n+1}_{i+\frac{1}{2}e_v}   + \frac{1}{\Delta x}(\Phi_i^{n+1} - \Phi_{i+e_v }^{n+1} )   m^{n+1}_{i+\frac{1}{2}e_v}  \Bigg\} \\
& &  \quad \quad  \quad \quad \quad  - \sum_{n=0}^{N-1} \Delta t \,  F( \{\rho\}_i^{n+1} ) +  G( \{\rho\}_i^0 )   + \sum_{n=0}^{N-1} \sum_{i \in V} \Phi_i^{n+1} \left(  \rho^{n+1}_i - \rho^n_i  \right)  
 \Bigg \} \\
 \eqnx
where the last equality is from the spatial integration by parts for $\sum_{n=0}^{N-1} \sum_{i \in V} \Phi_i^{n+1} \textrm{div}(m^{n+1})|_i $. From the Legendre transform 
$$H(p)=\sup_{v\in\mathbb{R}^1}~pv-L(v)$$ with $p=\frac{1}{\Delta x}(\Phi_{i+e_v}^n - \Phi_i^n)$ and $v=\frac{ m^{n+1}_{i+\frac{1}{2}e_v} }{ \theta^{n+1}_{i+\frac{1}{2}e_v}}$, 
we have 
 \beqnx
 & & \inf_{\{ m^n_{i+\frac{1}{2}e_v}  \}_{ n=0}^{N-1}  \in \mathbb{R}^{|E| N} \, , \, \{ \rho^0_i \}  \in  [0,1] ^{|V|}  }  \, \mathcal{F} (\mathbf{m}, \mathbf{\rho},\mathbf{\Phi}) \\
&=&
\inf_{ \{ \rho^0_i \}  \in  [0,1] ^{|V|}  } 
\Bigg\{ 
 - \sum_{n=1}^{N} \sum_{i+\frac{e_v}{2}\in E} H \left(\frac{1}{\Delta x}(\Phi_i^n - \Phi_{i+e_v}^n) \right)   \theta^{n}_{i+\frac{1}{2}e_v}  \Delta t   - \sum_{n=1}^{N} \,  F( \{\rho\}_i^{n+1} )  \Delta t   \\
& &  \quad \quad  \quad \quad \quad +  G( \{\rho\}_i^0 )   + \sum_{n=0}^{N-1} \sum_{i \in V} \Phi_i^{n+1} \left(  \rho^{n+1}_i - \rho^n_i  \right)  \Bigg \} \\
&=&
\inf_{ \{ \rho^0_i \}  \in  [0,1] ^{|V|}  } 
\Bigg\{ 
 - \sum_{n=1}^{N} \sum_{i+\frac{e_v}{2}\in E} H \left(\frac{1}{\Delta x}(\Phi_i^n - \Phi_{i+e_v}^n) \right)   \theta^{n}_{i+\frac{1}{2}e_v}  \Delta t   - \sum_{n=1}^{N} \,  F( \{\rho\}_i^{n+1} )  \Delta t  +  G( \{\rho\}_i^0 )  \\
& &  \quad \quad  \quad \quad \quad   + \sum_{n=1}^{N-1} \sum_{i \in V}\left(  \Phi_i^{n} -  \Phi_i^{n+1} \right)    \rho^{n}_i + \sum_{i \in V} \Phi_i^{N}  \rho_i   - \sum_{i \in V} \Phi_i^{0}  \rho^0_i   \Bigg \} \\
&=& \tilde{\mathcal{F}} ( \tilde{ \mathbf{\rho} },\mathbf{\Phi}),
\eqnx
where the last line follows from the definition of Legendre transform for $\{\rho^0\}_i^V$.
\end{proof}

%%%\noindent {Remark 6:} 
\begin{remark}
We remark that \cite{MFG1,MFG2, MFG} utilized a similar version of the above lemma and computed the saddle point uses Newton's method.  However, since we aim to reduce the number of dimensions in our numerical scheme and also aim to handle nonsmooth cases, we do not stop at this formulation.
\end{remark}
%%Next, we would like to obtain the following lemma:

\begin{Lemma}  \label{hahalolXDsayonara}

If (A1), (A3),(A4), (A5), (A6),(A7) are satisfied, then 
\begin{eqnarray*}
& & \sup_{ \{ \Phi^n_i \}_{ n=1}^{N}   \in \mathbb{R} ^{|V| N}  }   \inf_{\{ \rho^n_i \}_{ n=1}^{N-1}   \in  [0,1] ^{|V| (N-1)}  }\tilde{\mathcal{F}} ( \tilde{ \mathbf{\rho} },\mathbf{\Phi})   \notag  \\
&=& \sup_{\{ \Phi_i \}  \in \mathbb{R} ^{|V| }    }    \Bigg \{ 
 \sum_{i\in V} \Phi_i^{N}  \rho_i   - \sum_{n=1}^N \Delta t \left( F( \{\rho\}_i^n ) -  \sum_i [\nabla_\rho F ( \{\rho\}_i^n ) ]_i \rho^n_i \right)     -  G^*( \{\Phi\}_i^0 )   :  \notag \\
& & \qquad \qquad
 \quad \quad  \begin{matrix}
  \rho^{n}_i - \rho^{n-1}_i +  \Delta t \sum_{v=1}^d D_p H \left(\frac{1}{\Delta x} (\Phi_i^n-\Phi_{i+e_v}^n) \right)    \theta^n_{i+\frac{1}{2}e_v}  = 0  \\
    \Phi^{n+1}_i - \Phi^{n}_i +  \frac{ \Delta t }{4} \sum_{v=1}^d  H  \left(\frac{1}{\Delta x}(\Phi_i^n - \Phi_{i+e_v}^n) \right) +\Delta t[\nabla_\rho F ( \{\rho\}_i^n ) ]_i  = 0 \\
     \rho^{N}_i = \rho_i, \, \Phi_i^N   = \Phi_i 
  \end{matrix}  \Bigg\}\ .
%\label{Smart}
\end{eqnarray*}
\end{Lemma}

\begin{proof}

Writing $ ( \tilde{ \mathbf{\rho} },\tilde{\mathbf{\Phi}}) = (  \{ \Phi^n_i \}_{ n=1}^{N-1}  ,  \{ \rho^n_i \}_{ n=1}^{N-1}  )  $, then we have
\beqnx
 \sup_{ \{ \Phi^n_i \}_{ n=1}^{N}   \in \mathbb{R} ^{|V| N}  }   \inf_{\{ \rho^n_i \}_{ n=1}^{N-1}   \in  [0,1] ^{|V| (N-1)}  }\tilde{\mathcal{F}} ( \tilde{ \mathbf{\rho} },\mathbf{\Phi})  = \sup_{ \{ \Phi^N_i \}  \in \mathbb{R} ^{N}  }   \sup_{ \{ \Phi^n_i \}_{ n=1}^{N-1}   \in \mathbb{R} ^{|V| (N-1)}  }   \inf_{\{ \rho^n_i \}_{ n=1}^{N-1}   \in  [0,1] ^{|V| (N-1)}  } \tilde{\mathcal{F}} ( \tilde{ \mathbf{\rho} }, \tilde{\mathbf{\Phi}}, \Phi^N)   \notag.
\eqnx
Given $(\tilde{\mathbf{\Phi}}, \Phi^N_i)$, from (A3) and (A7), we have that the infimum 
$$
 \inf_{\{ \rho^n_i \}_{ n=1}^{N-1}   \in  [0,1] ^{|V| (N-1)}  } \tilde{\mathcal{F}} ( \tilde{ \mathbf{\rho} }, \tilde{\mathbf{\Phi}}, \Phi^N)  
$$
is attained in the interior of the domain $[0,1] ^{|V| (N-1)} $. From the smoothness given by (A5), there exists $ \tilde{ \mathbf{\rho} }^* ( \tilde{\mathbf{\Phi}}, \Phi^N ) $ smoothly depending on $( \tilde{\mathbf{\Phi}}, \Phi^N_i )$ such that 
$$
  \frac{\Phi^{n+1}_i - \Phi^{n}_i}{\Delta t} +  \frac{1}{4} \sum_{v=1}^d  H  \left(\frac{1}{\Delta x}(\Phi_i^n - \Phi_{i+e_v}^n) \right) +[\nabla_\rho F ( \{\rho^*( \tilde{\mathbf{\Phi}}, \Phi^N )  \}_i^n ) ]_i  = 0 ,$$
holds.  Now let us fix $\Phi^N$. From the definition, we have 
\beqnx
& &  \inf_{\{ \rho^n_i \}_{ n=1}^{N-1}   \in  [0,1] ^{|V| (N-1)}  } \tilde{\mathcal{F}} ( \tilde{ \mathbf{\rho} }, \tilde{\mathbf{\Phi}}, \Phi^N)   \\
& = & - \sum_{n=1}^N  \Delta t F^* \left(  \left\{   \frac{ \Phi^{n+1}_i - \Phi^{n}_i}{\Delta t} + \frac{1}{4} \sum_{v=1}^d  H  \left(\frac{1}{\Delta x}(\Phi_i^n - \Phi_{i+e_v}^n) \right) \right\}_i^n \right)
-  G^*( \{\Phi\}_i^0 ) + \sum_{i \in V} \Phi_i^{N}  \rho_i.
\eqnx
Now for any given $\{ V_i^n \}_{n=0}^{N-1}$, by solving the difference equation, there exists $\{ \Phi_i^n \}_{n=1}^{N-1}$ such that
\beqnx
 \frac{ \Phi^{n+1}_i - \Phi^{n}_i}{\Delta t} + \frac{1}{4} \sum_{v=1}^d  H  \left(\frac{1}{\Delta x}(\Phi_i^n - \Phi_{i+e_v}^n) \right)  = V_i^{n-1}.
\eqnx
Therefore we have
\beqn
& &\sup_{ \{ \Phi^n_i \}_{ n=1}^{N-1}   \in \mathbb{R} ^{|V| (N-1)}  }   \inf_{\{ \rho^n_i \}_{ n=1}^{N-1}   \in  [0,1] ^{|V| (N-1)}  } \tilde{\mathcal{F}} ( \tilde{ \mathbf{\rho} }, \tilde{\mathbf{\Phi}}, \Phi^N) \notag \\
&= & \sup_{ \{ V^n_i \}_{ n=1}^{N-1}   \in \mathbb{R} ^{|V| (N-1)}  }   \left \{ - \sum_{n=0}^{N-1}  \Delta t F^* \left(  V_i^n \right) 
-  G^*( \{\Phi\}_i^0 ) + \sum_{i \in V} \Phi_i^{N}  \rho_i    \right \} \label{wow_hopf} \,.
\eqn
Note that by (A6), the supremum in the above is attained, and hence there exists a maximum point $ \{ {V^*}_i^n  \}_{n=0}^{N-1} $ for the functional $ - \sum_{n=0}^{N-1}  F^* \left(  V_i^n \right)$.  Now go back to find $\{ {\Phi^*}_i^n \}_{n=1}^{N-1}$ such that 
\beqnx
 \frac{ {\Phi^*}^{n+1}_i - {\Phi^*}^{n}_i}{\Delta t} + \frac{1}{4} \sum_{v=1}^d  H  \left(\frac{1}{\Delta x}({\Phi^*}_i^n - {\Phi^*}_{i+e_v}^n) \right)  = {V^*}_i^{n-1}.
\eqnx
Now, fixing $\Phi^N$ and noticing $\{ {\Phi^*}_i^n \}_{n=1}^{N-1}$ is a maximum value of the function
\beqnx
\tilde{\mathbf{\Phi}} = \{ {\Phi}_i^n \}_{n=1}^{N-1}\,  \mapsto \inf_{\{ \rho^n_i \}_{ n=1}^{N-1}   \in  [0,1] ^{|V| (N-1)}  } \tilde{\mathcal{F}} ( \tilde{ \mathbf{\rho} }, \tilde{\mathbf{\Phi}}, \Phi^N)  =  \tilde{\mathcal{F}} (  \tilde{ \mathbf{\rho} }^* ( \tilde{\mathbf{\Phi}}, \Phi^N )  , \tilde{\mathbf{\Phi}}, \Phi^N),
\eqnx
we see that the maximum is attained and by (A5), $\{ {\Phi^*}_i^n \}_{n=1}^{N-1}$ can be characterized by
\beqnx
{\rho^*( \tilde{\mathbf{\Phi}}, \Phi^N ) }^{n}_i - {\rho^*( \tilde{\mathbf{\Phi}}, \Phi^N )  }^{n-1}_i +  \Delta t \sum_{v=1}^d D_p H \left(\frac{1}{\Delta x} ({\Phi^*}_i^n-{\Phi^*}_{i+e_v}^n) \right)   \theta^n_{i+\frac{1}{2}e_v  }( \tilde{\mathbf{\Phi}}, \Phi^N )     = 0.
\eqnx
Concluding the above argument, for each $\Phi^N$, we have $$
 \sup_{ \{ \Phi^n_i \}_{ n=1}^{N-1}   \in \mathbb{R} ^{|V| (N-1)}  }   \inf_{\{ \rho^n_i \}_{ n=1}^{N-1}   \in  [0,1] ^{|V| (N-1)}  } \tilde{\mathcal{F}} ( \tilde{ \mathbf{\rho} }, \tilde{\mathbf{\Phi}}, \Phi^N)
=   \tilde{\mathcal{F}} ( \tilde{ \mathbf{\rho} }^*(\Phi^N), \tilde{\mathbf{\Phi}}^*(\Phi^N), \Phi^N),$$ which satisfies the following pair of equations
\beqn
 \begin{cases}
  \rho^{n}_i - \rho^{n-1}_i +  \Delta t \sum_{v=1}^d D_p H \left(\frac{1}{\Delta x} (\Phi_i^n-\Phi_{i+e_v}^n) \right)    \theta^n_{i+\frac{1}{2}e_v}  = 0  \\
    \Phi^{n+1}_i - \Phi^{n}_i +  \frac{ \Delta t }{4} \sum_{v=1}^d  H  \left(\frac{1}{\Delta x}(\Phi_i^n - \Phi_{i+e_v}^n) \right) +\Delta t[\nabla_\rho F ( \{\rho\}_i^n ) ]_i  = 0 \\
     \rho^{N}_i = \rho_i, \, \Phi_i^N   = \Phi_i 
\end{cases}
\eqn
for all $n = 0,1,..,N-1$. 
The conclusion of the lemma thus follows.
\end{proof}

\begin{remark}  We note that taking supremum of \eqref{wow_hopf}  over $\Phi^N$ yields the discrete version of the Hopf formula given in (57)-(59) of \cite{Hopf_MFG}. However, we are not trying to get back to (57)-(59) in \cite{Hopf_MFG} since the formulation, though very elegant mathematically, contains too many variables for numerical optimization of a low memory requirement, and is thus not our first choice. % considering the our purpose to develop efficient numerical schemes.
\end{remark}
Combining all the three lemmas above, we prove Theorem \ref{thm}. In next section, we  apply the discrete Hopf formula \eqref{Smart} to design numerical methods for HJD.
\section{Algorithm}\label{section4}
In this section, we compute the optimizer in the Hopf formula in \eqref{Smart}. We shall perform the following multi-level block stochastic gradient descent method.

We first consider a sequence of step-size $h_i = 2^{-i}$ where $i = 0,\dots,N $ and a nested sequence of finite dimensional subspaces $  V_{h_0} \subset V_{h_2}  \subset \dots \subset V_{h_{N-1}}  \subset V_{h_N} $ of a function space over $[-1,1]^d \subset X$.
Now, we also define a family of restriction and extension operators:
\beqnx
R_{ij} : V_{h_i} \rightarrow V_{h_j} \quad  \text{ and } \quad E_{ji} : V_{h_j} \rightarrow V_{h_i}
\eqnx
Now, let us define the following approximation of the functional $G(\cdot)$ as
\beqnx
G_j~:~V_{h_j} &\rightarrow &\mathbb{R} \\
G_j(\tilde{\Phi}) &:=&   \int_X R_{Nj} [ \rho](x)  \tilde{\Phi}(x) dx - G^*(  \Phi(E_{jN} [\tilde{\Phi}], 0,\cdot) )
\\
& & -\int_0^t  \left( F(  \rho(E_{jN} [\tilde{\Phi}],s,\cdot)  ) -  \int_X [\nabla_\rho F ( \rho(E_{jN} [\tilde{\Phi}],s,\cdot)  ) ] \rho(E_{jN} [\tilde{\Phi}],s,x) dx \right) dt \\
\eqnx
where $(\rho(E_{jN} [\tilde{\Phi}], s,x) , \Phi(E_{jN} [\tilde{\Phi}], s,x) )$ numerically solves the following terminal value problem:
\beqnx
\begin{cases}
 \partial_s  \rho(E_{jN} [\tilde{\Phi}], s,x)  + \nabla \cdot (  \rho(E_{jN} [\tilde{\Phi}], s,x) D_pH(x, \nabla \Phi(E_{jN} [\tilde{\Phi}], s,x)   )  )=0\\%+ f(\rho(E_{jN} [\tilde{\Phi}], s,x)) = 0 \notag \\
 \partial_s  \Phi(E_{jN} [\tilde{\Phi}], s,\cdot)  +  H(x, \nabla \Phi(E_{jN} [\tilde{\Phi}], s,x)  )+\nabla_\rho F ( \rho(E_{jN} [\tilde{\Phi}],s,\cdot))= 0 \\
\rho(E_{jN} [\tilde{\Phi}], t,x) = \rho(x), \quad \Phi(E_{jN} [\tilde{\Phi}], t,x) = E_{jN} [\tilde{\Phi}](x)  \,.
\end{cases}
\eqnx
The numerical method to compute this Cauchy problem will be discussed after we present the main algorithm in Remark \ref{rm10}.

With the above notation, we are ready to present our variant of stochastic gradient descent to optimize $G_N(\cdot)$. We utilize the following coordinate descent algorithm:
\begin{algorithm}
Take an initial guess $[\Phi_0]^1 \in V_{h_0}$, for $i = 0,...,N$, do:
\begin{itemize}
\item
Take an initial guess of the Lipschitz constant $L_i$, set ${count} := 0$ and $v_i := 1/L_i$. 
\item
For $k=1,....,M$, do:
\begin{itemize}
\item[{1}:]
Randomly select $I_k \in 2^{\{1,..., 2^{di} \}}$,
\item[{2}:]
Compute the following unit vector $v_{I_k}$ where
\beqnx
[v_{I_k}]_l = \begin{cases}
1/\sqrt{|I|} & \text{ if } l \in I,\\
0 & \text{ otherwise. }
\end{cases}
\eqnx
\item[{3}:]
Compute
\beqnx
\begin{cases}
[\Phi_i]_I^{k+1}  =  [\Phi_i]_I^{k} - v_i \partial_{v_{I_k}} G_i ( [\Phi_i]_I^{k} ) &  \text{ if } l \in I\\
[\Phi_i]_l^{k+1} =  [\Phi_i]_l^{k+1}  & \text{ otherwise. }
\end{cases}
\eqnx
\item[{4}:]
If $| [\Phi_i]_I^{k+1}  -  [\Phi_i]_I^{k}| > \varepsilon$, then set $\text{count} := 0$. \
If $k =M$, then reset $k:=0$ and set $v_i := v_i/2$,  (i.e. let $L_i := 2L_i$.)
\item[{5}:]
If $| [\Phi_i]_I^{k+1}  -  [\Phi_i]_I^{k}| < \varepsilon$, set $\text{count} := \text{count} +1$.
\item[{5}:]
If $\text{count} = 2^{di}$, define $[\Phi_i]_\text{final} :=   [\Phi_i]_I^{k+1} $ , stop.
\end{itemize}
\item
If $i <N$, set $[\Phi_{i+1}]^1  = E_{i \, (i+1) } \left( [\Phi_i]_\text{final} \right) $ \,.
\end{itemize}
Output $ [\Phi_N]_\text{final}$.
\end{algorithm}

\begin{remark}\label{rm10}
For computation of $\big(\rho(E_{jN} [\tilde{\Phi}], s,x) , \Phi(E_{jN} [\tilde{\Phi}], s,x) \big)$ with numerical PDE techniques, we notice that the primal-dual system in \eqref{Smart}, i.e.  the conservation law and the HJD, may not provide a stable PDE algorithm.  One way to address this pathology is to modify the numerical Hamiltonian that we have implicitly chosen when we derive our algorithm.
We notice indeed that the system \eqref{Smart} is a symplectic scheme that conserves the following numerical Hamiltonian (writing $\rho^{n} := \{ \rho^{n}_i\}$ and $\Phi^n := \{ \Phi^n_i\}$):
\beqnx
\mathbb{H} (\rho^{n}, \Phi^n) =  \sum_{i+\frac{e_v}{2}\in E}   H  \left(\frac{1}{\Delta x}(\Phi_i^n - \Phi_{i+e_v}^n) \right)   \theta^n_{i+\frac{1}{2}e_v}+  F( \{\rho\}_i^n )  \,.
\eqnx
On the other hand, we notice that the choice of such Hamiltonian is not unique: we can choose another numerical Hamiltonian that corresponds to an upwind (monotone) scheme for the primal system and monotone Hamiltonian for the dual system as follows (see also \cite{MFG1, MFG2, MFG}):
\beqnx
\mathbb{H} (\rho^{n}, \Phi^n) := \sum_i  \sum_{v=1}^d H  \left( \left[ \frac{1}{\Delta x}(\Phi_i^n - \Phi_{i+e_v}^n) \right]^+ \right)    \rho^n_i+F( \{\rho\}_i^n ), 
\eqnx
where $ \left[ \cdot \right]^+ := \max(0, \cdot )$ and  $ \left[ \cdot \right]^- := \min(0, \cdot )$. With this, the primal dual system will instead be as follows:
\beqnx
\frac{\rho^{n}_i - \rho^{n-1}_i}{\Delta t} +   \sum_{v=1}^d D_p H \left(  \left[ \frac{1}{\Delta x}(\Phi_i^n - \Phi_{i+e_v}^n) \right]^+ \right)    \rho^n_i   + \sum_{v=1}^d D_p H \left(  \left[ \frac{1}{\Delta x}(\Phi_i^n - \Phi_{i+e_v}^n) \right]^-\right)    \rho^n_{i+e_v}     &=& 0  \\
 \frac{ \Phi^{n+1}_i - \Phi^{n}_i}{\Delta t} +  \frac{1}{2} \sum_{v=1}^d  H  \left( \left[ \frac{1}{\Delta x}(\Phi_i^n - \Phi_{i+e_v}^n) \right]^+ \right) +[\nabla_\rho F ( \{\rho\}_i^n ) ]_i  &=& 0. 
\eqnx
To further enhance  stability, we can add, given a regularization parameter $\beta$, a Lax-Friedrichs scheme numerical diffusion term:
\beqnx
   & &\frac{\rho^{n}_i - \rho^{n-1}_i}{\Delta t} +  \Big\{\sum_{v=1}^d D_p H \left(  \left[ \frac{1}{\Delta x}(\Phi_i^n - \Phi_{i+e_v}^n) \right]^+ \right)    \rho^n_i  \\
&&\hspace{1.8cm} +  \sum_{v=1}^d D_p H \left(  \left[ \frac{1}{\Delta x}(\Phi_i^n - \Phi_{i+e_v}^n) \right]^-\right)    \rho^n_{i+e_v} \Big\} +\frac{\beta \Delta x}{2 (\Delta x)^2} \sum_{v=1}^d ( \rho_i^n - \rho_{i+e_v}^n ) = 0  \\
   & &\frac{\Phi^{n+1}_i - \Phi^{n}_i}{\Delta t} +  \frac{1}{2} \sum_{v=1}^d  H  \left( \left[ \frac{1}{\Delta x}(\Phi_i^n - \Phi_{i+e_v}^n) \right]^+ \right) + [\nabla_\rho F ( \{\rho\}_i^n ) ]_i + \frac{\beta \Delta x }{2 (\Delta x)^2} \sum_{v=1}^d ( \Phi_i^n - \Phi_{i+e_v}^n )  = 0 \,.
\eqnx
This adds a magnitude of $\beta \Delta x$ numerical diffusion in the system.
We notice in our numerical examples that stability improves after imposing $v >0 $ and considering an upwind monotone scheme.
\end{remark}
\section{Numerical results} \label{sec5}
In this section, we present numerical results for solving HJD by Algorithm 1. We tested several cases with the different Hamiltonians, including the convex
\begin{equation*}
H_1(x,p) =  \frac{1}{2} ( |p_1|^2 + |p_2|^2 ) \,,
\end{equation*}
 the {non-convex
\begin{equation*}
H_2(x,p) =  \frac{1}{2} ( |p_1|^2 - |p_2|^2 ) \,,
\end{equation*}
and the convex 1-homogeneous Hamiltonian}
\begin{equation*}
H_3(x,p) =  |p_1| + |p_2| .
\end{equation*}
For a given center $x_0$ and radius $R$, we consider
\beqnx
G(\rho) = \inf_{\tilde{\rho} \in \mathcal{P}(X) } \left \{ \iota_{\mathcal{P}(B_R(x_0))} (\tilde{\rho}) + \frac{1}{2 v} \int_X [ \tilde{\rho} - \rho(x)]^2 dx \right \},
\eqnx
where $v$ is a regularization parameter, and we recall that, for a given convex subset $\mathcal{B} \subset \mathcal{P}(X)$, the indicator function $\iota_{\mathcal{B}}(\rho) = 0 $ if $\rho \in \mathcal{B}$ and $\iota_{\mathcal{B}}(\rho) = \infty $ otherwise.
A direct computation shows
\beqnx
\iota_{\mathcal{P}(B_R(x_0))}^*(\Phi) = \sup_{\rho\in\mathcal{P}(B_R(x_0))}~\int_X\rho(x)\Phi(x) dx.
=  \sup_{x \in B_R(x_0)} \Phi(x) \,.
\eqnx
With the correspondence of summation and infimum convolution via Legendre transform, we arrive at
\beqnx
G^*(\Phi) = \sup_{x \in B_R(x_0)} \Phi(x) + \frac{v}{2} \int_X [ \Phi(x)]^2 dx.
\eqnx
In numerical examples, we set $v = 10^{-3}$.
This helps us compute a regularized projection of a given $\rho$ to the set of all the measures supported at an unit ball.
For simplicity, we set $F(\rho) = 0$ in all our examples. 

We utilize {Algorithm 1} for numerical computations. The number of levels $M=3$ is always chosen.  The Lipschitz constant is always chosen as $L_i = L =2$. For numerical approximation of PDE, we choose the upwind numerical Hamiltonian, together with an addition of Lax-Friedrichs numerical diffusion where $\beta=2$ is chosen.  The discretization parameters are chosen as $\Delta x = 0.04 $ and $\Delta t = 0.008$. In all experiments, we consider $X=\mathbb{T}^2$.

\begin{example}  In this example, we consider the Hamiltonian $H_1$ and the input distribution $\rho(x)$ as follows:
%is the same as in Example 3, again see Figure \ref{exp_4}.

\begin{figurehere}
     \begin{center}
     \vskip -0.3truecm
       \scalebox{0.4}{\includegraphics{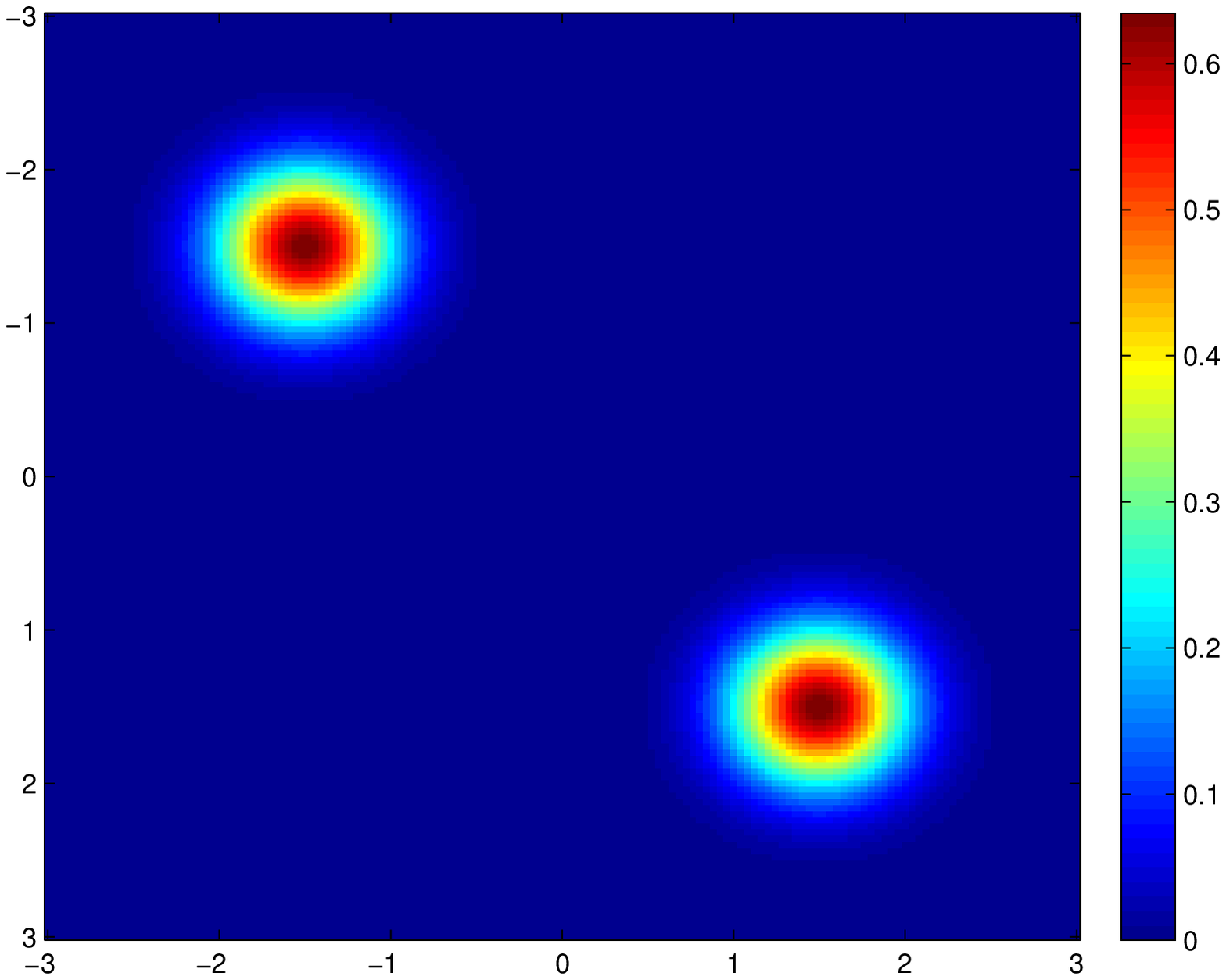}}
\caption{\small The distribution of the input $\rho$.}\label{exp_4}
     \end{center}
 \end{figurehere}

\noindent We choose the center and radius $(x_0 ,R)$ which helps to define $G(\rho)$ as $x_0 = (0,0), R = 1$. Figure \ref{exp_13} gives the optimizer $\tilde{\Phi}$ (left) in \eqref{Smart} and its gradient $\nabla_x \tilde{\Phi}$ (right) computed using {Algorithm 1} when $t = 1$ in the Hamiltonian.  The gradient $\nabla_x \tilde{\Phi}$ generates the final kick of the drift for the masses to be flown accordingly.

\begin{figurehere}
     \begin{center}
     \vskip -0.3truecm
        \scalebox{0.4}{\includegraphics{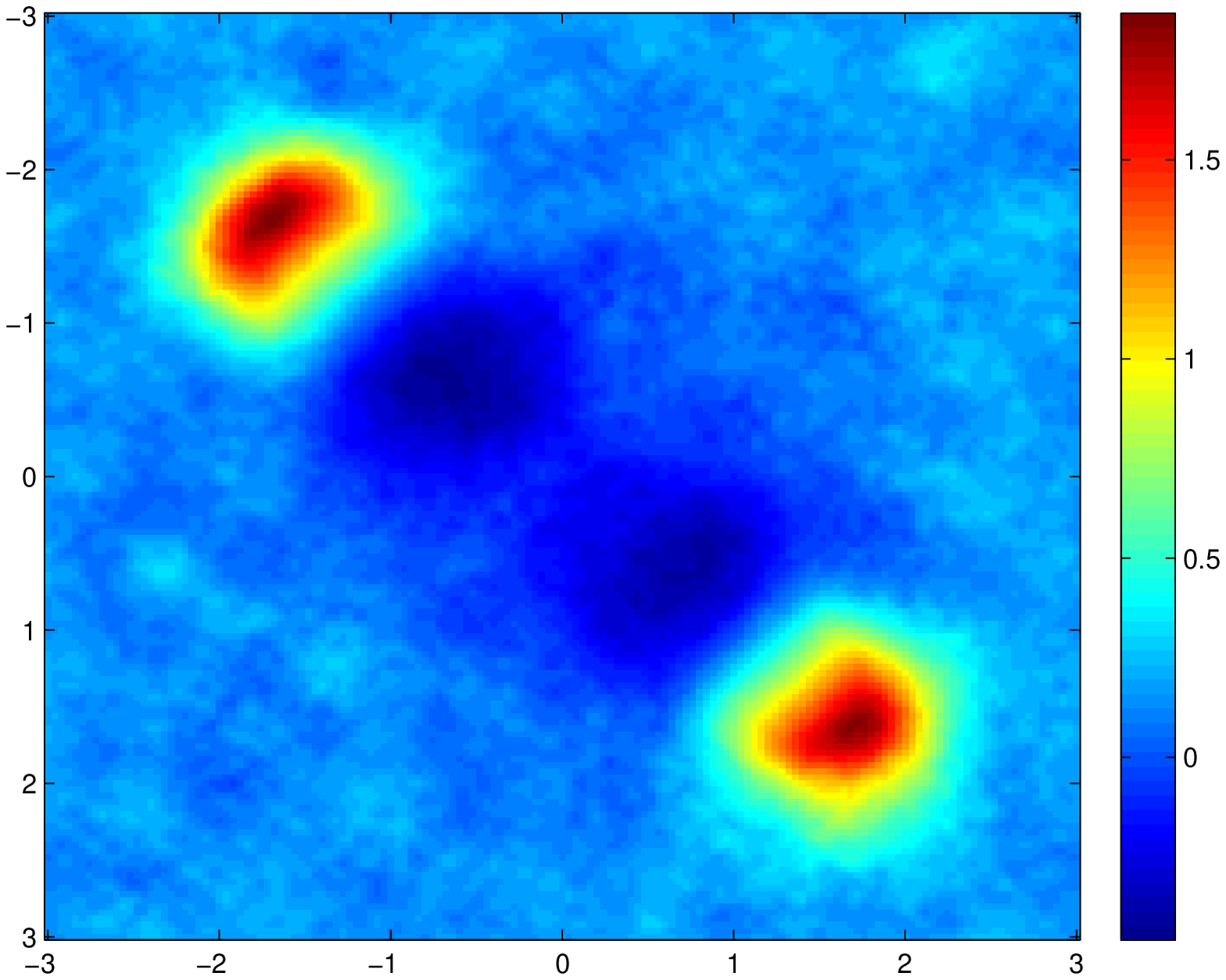}}
\scalebox{0.4}{\includegraphics{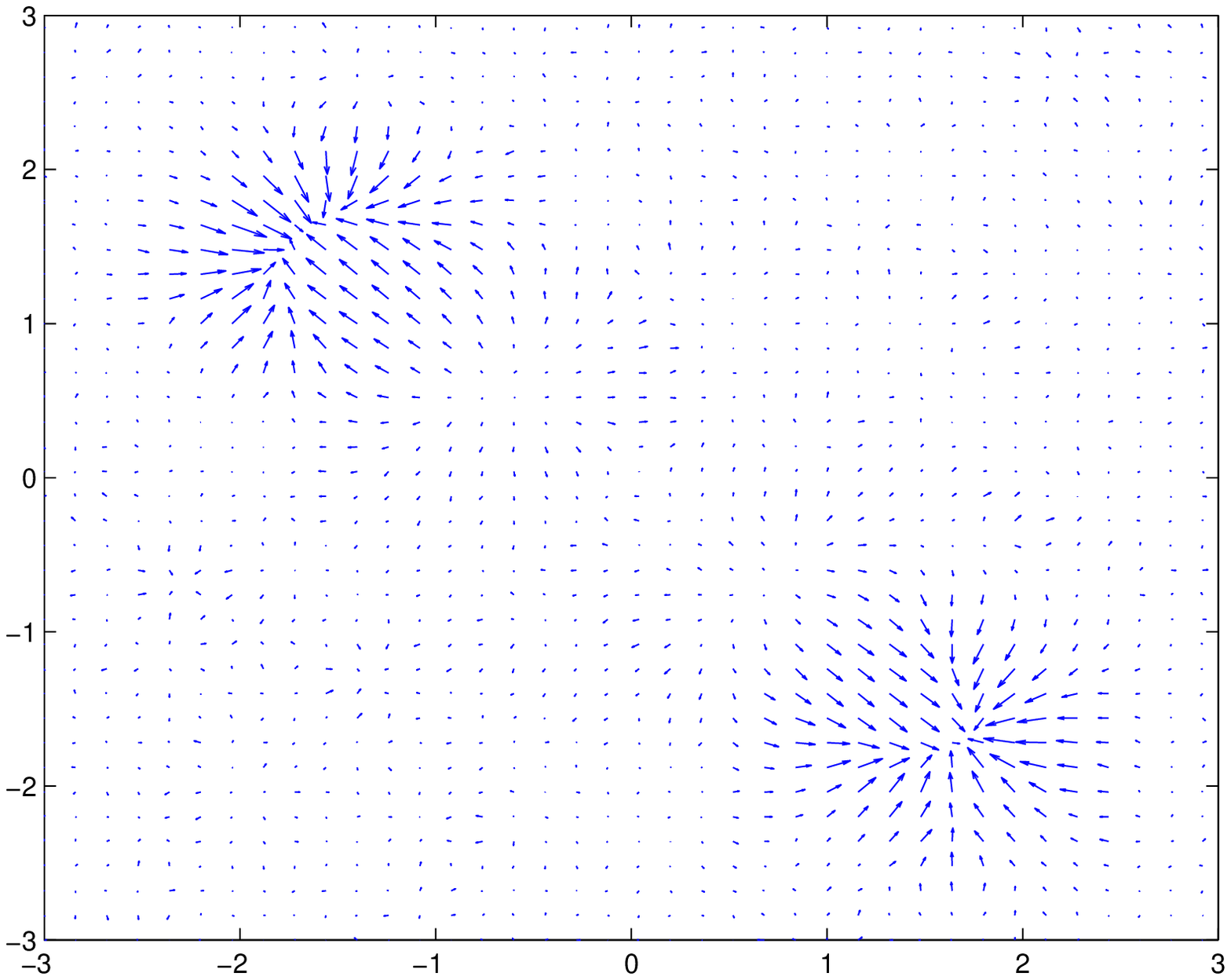}} \\
\caption{\small Left: optimizer $\tilde{\Phi}$ in $U(t,\rho)$ in \eqref{lala_hopf}, right: vector field $\nabla_x \tilde{\Phi}$.}\label{exp_13}
     \end{center}
 \end{figurehere}
In Figure \ref{exp_14}, we plot the distributions $ \rho(t,x) $ for different $t = 0,0.2,0.4,0.6,0.8,1.0$. It describes the transportation of the masses according to the flow generated by the gradient of $\Phi(t,x)$ at different times $t$.

\begin{figurehere}
     \begin{center}
     \vskip -0.3truecm
       \scalebox{0.35}{\includegraphics{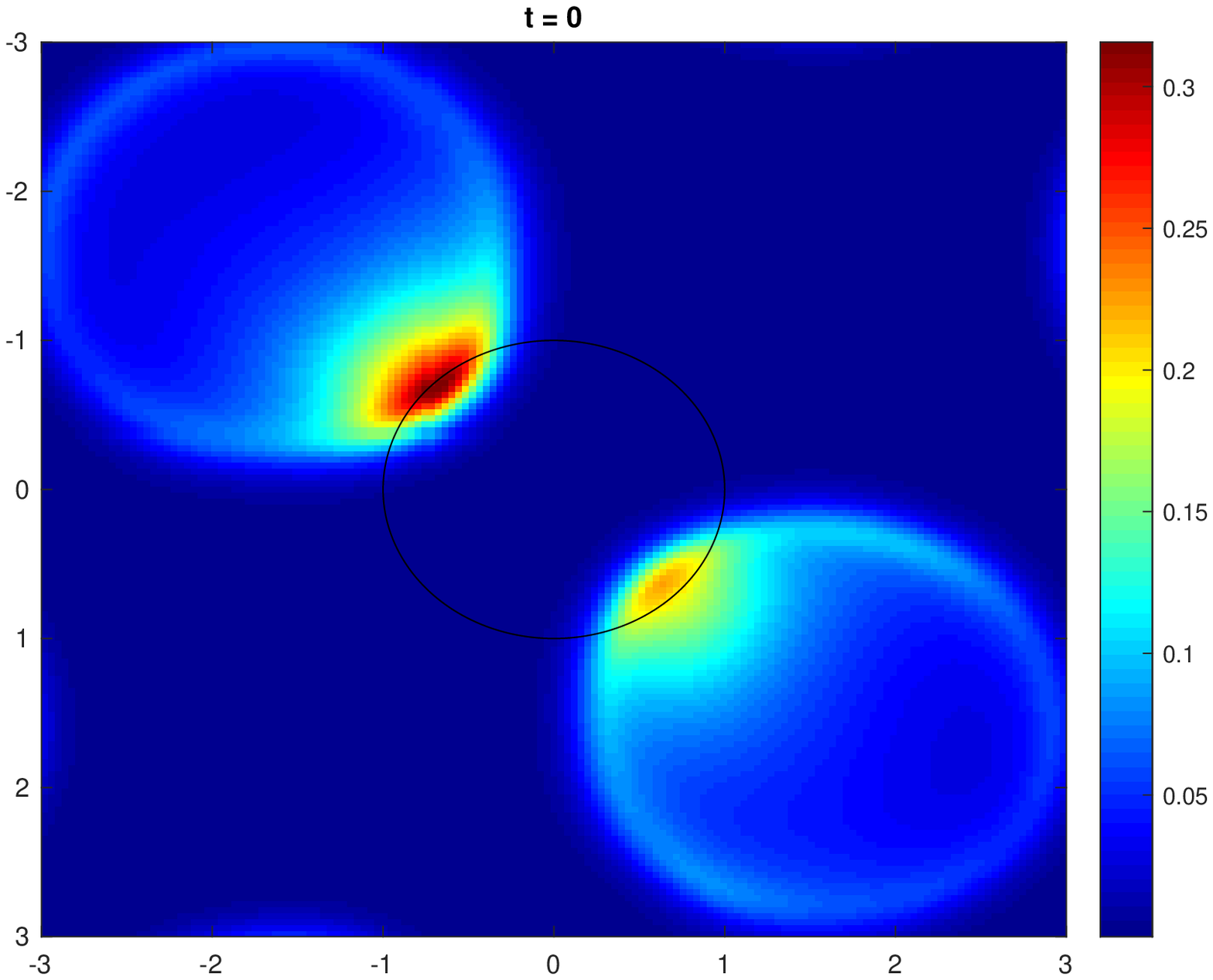}}
    \scalebox{0.35}{\includegraphics{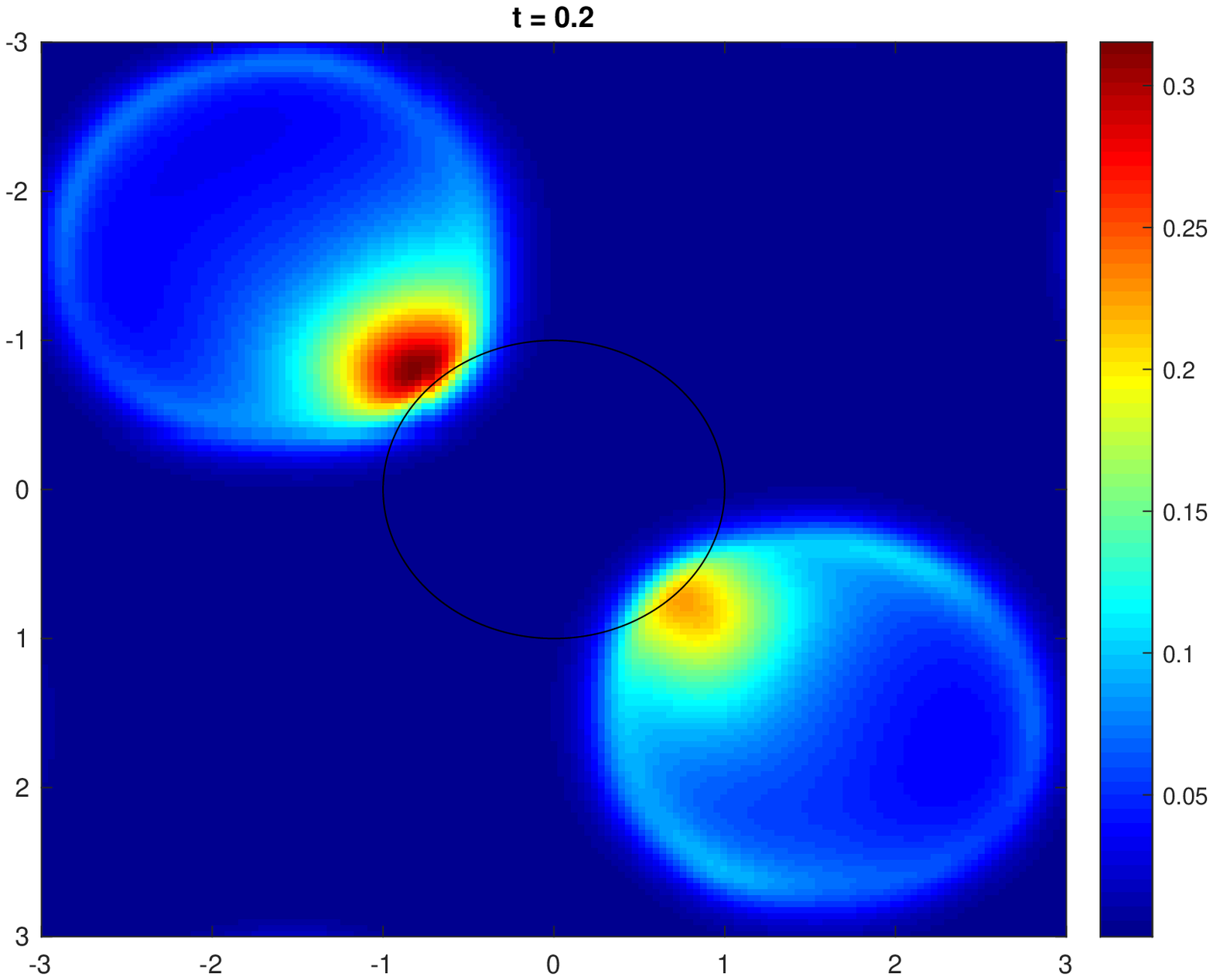}} \\
   \vskip -0.3truecm
       \scalebox{0.35}{\includegraphics{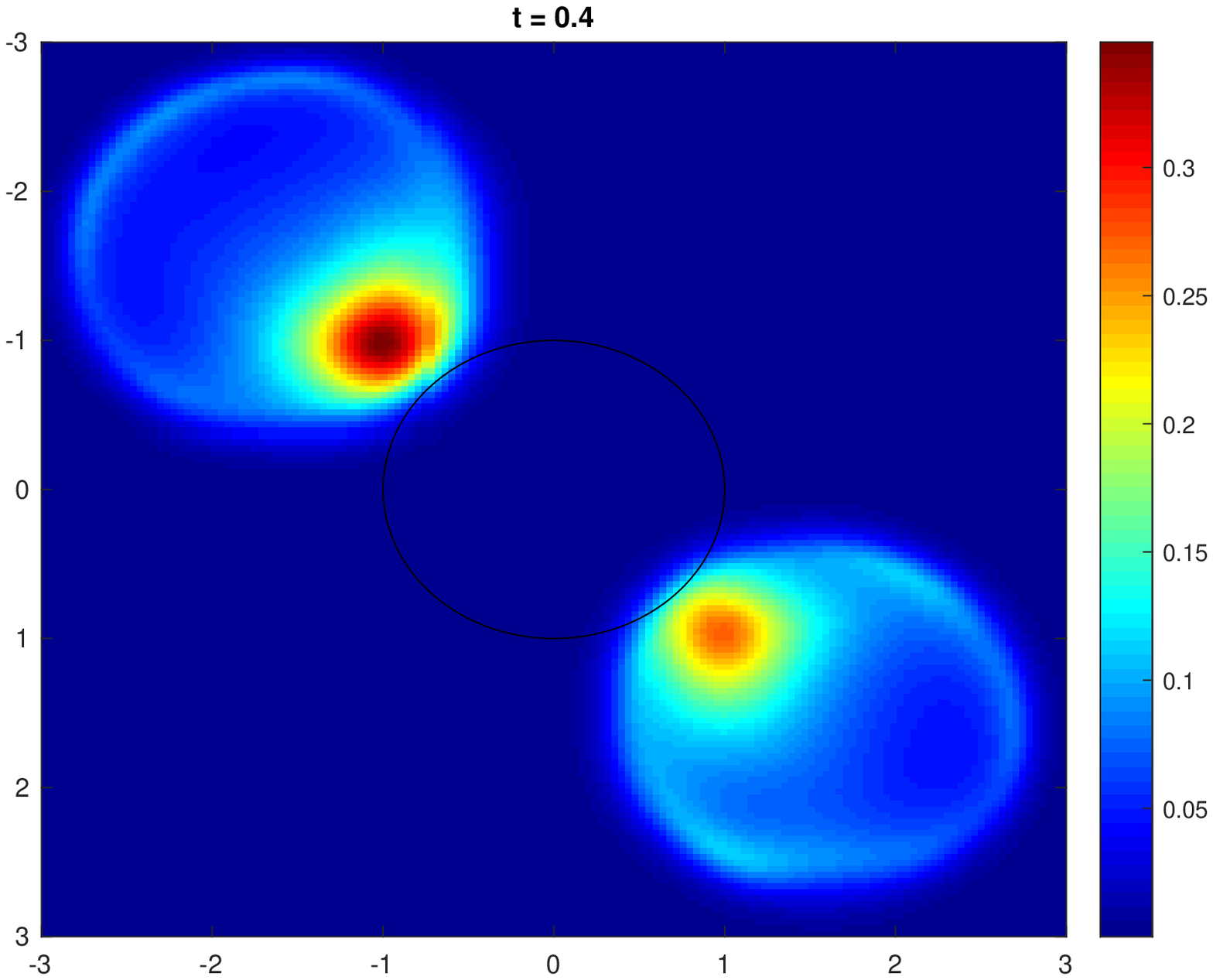}}
    \scalebox{0.35}{\includegraphics{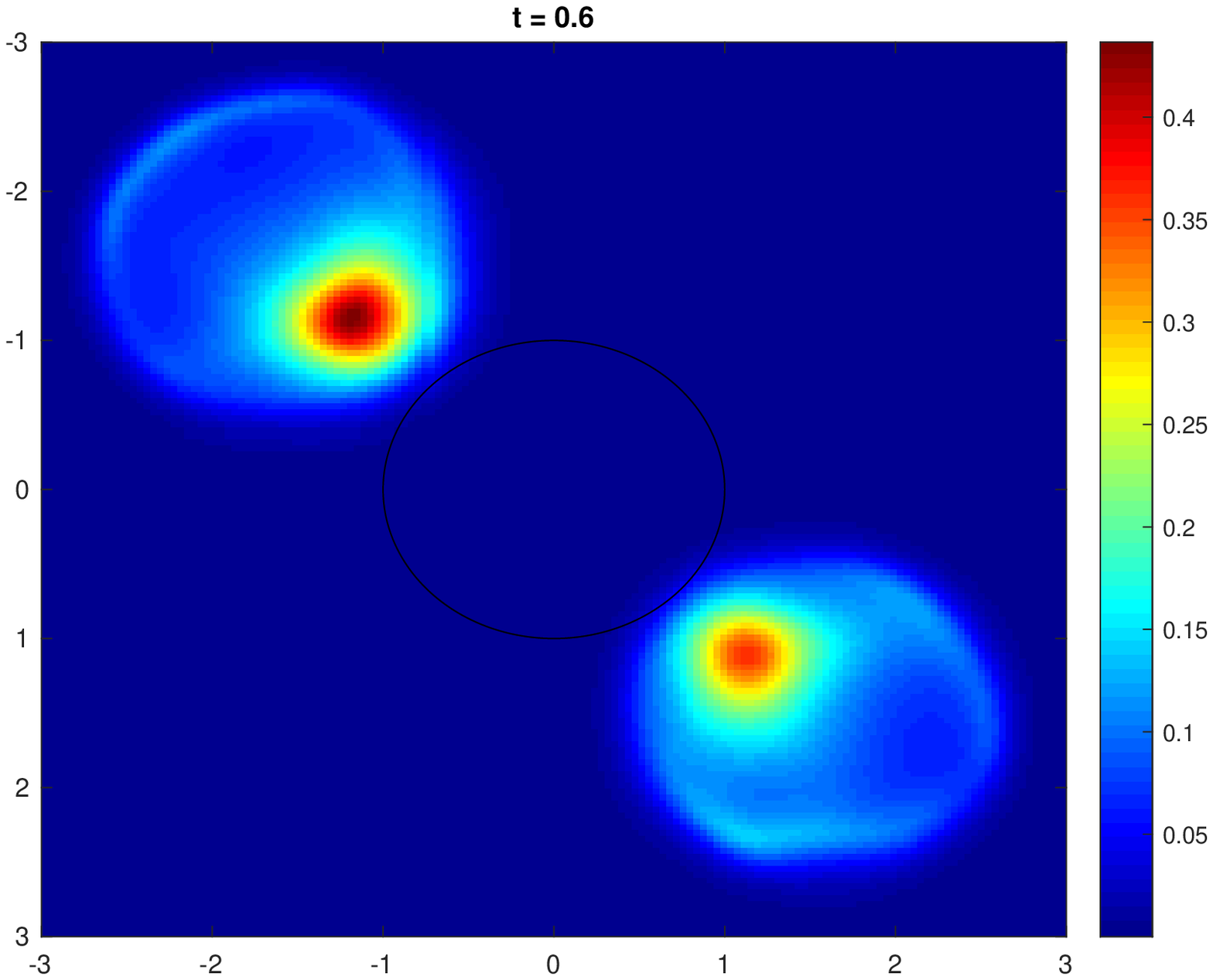}} \\
   \vskip -0.3truecm
       \scalebox{0.35}{\includegraphics{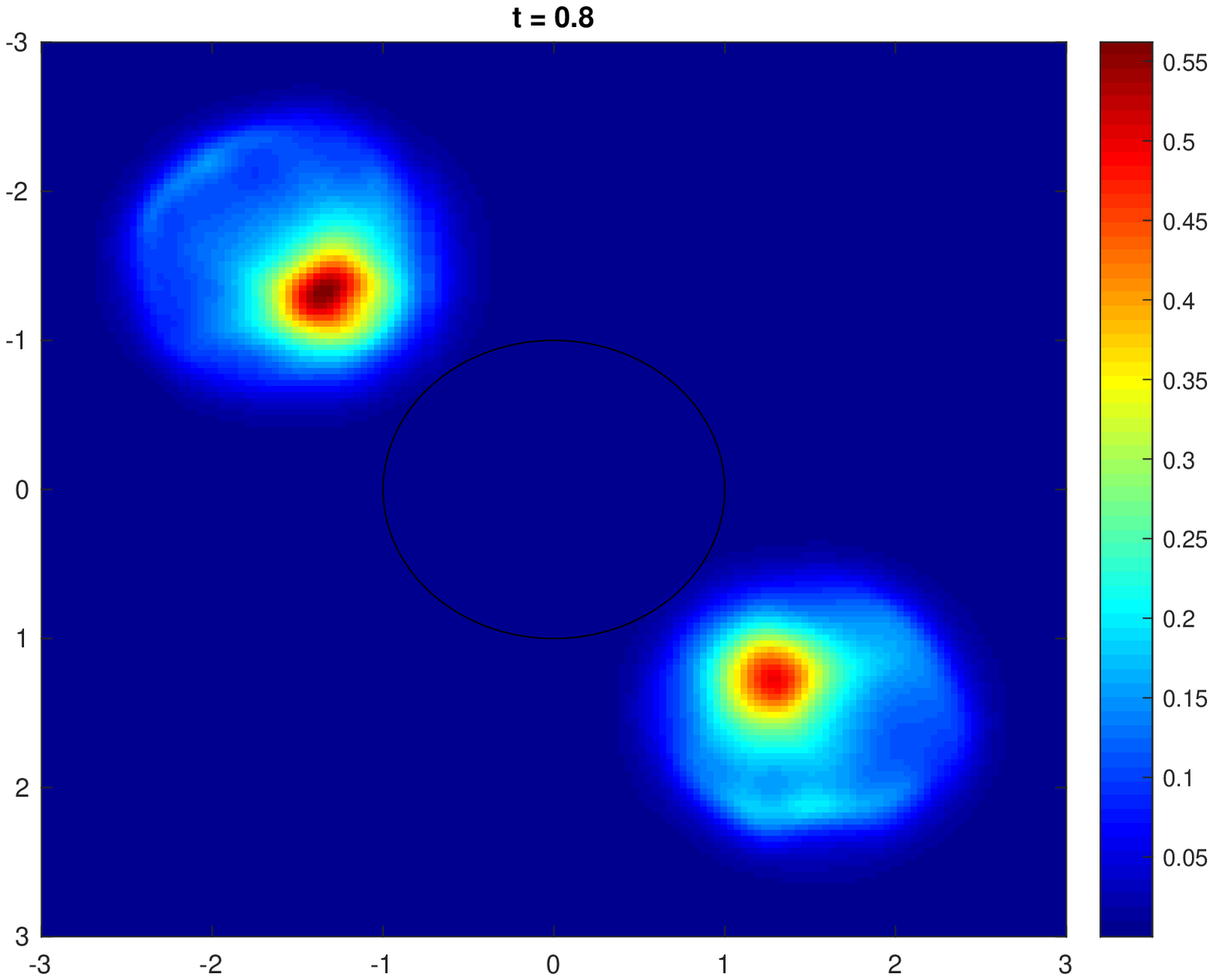}} 
         \scalebox{0.35}{\includegraphics{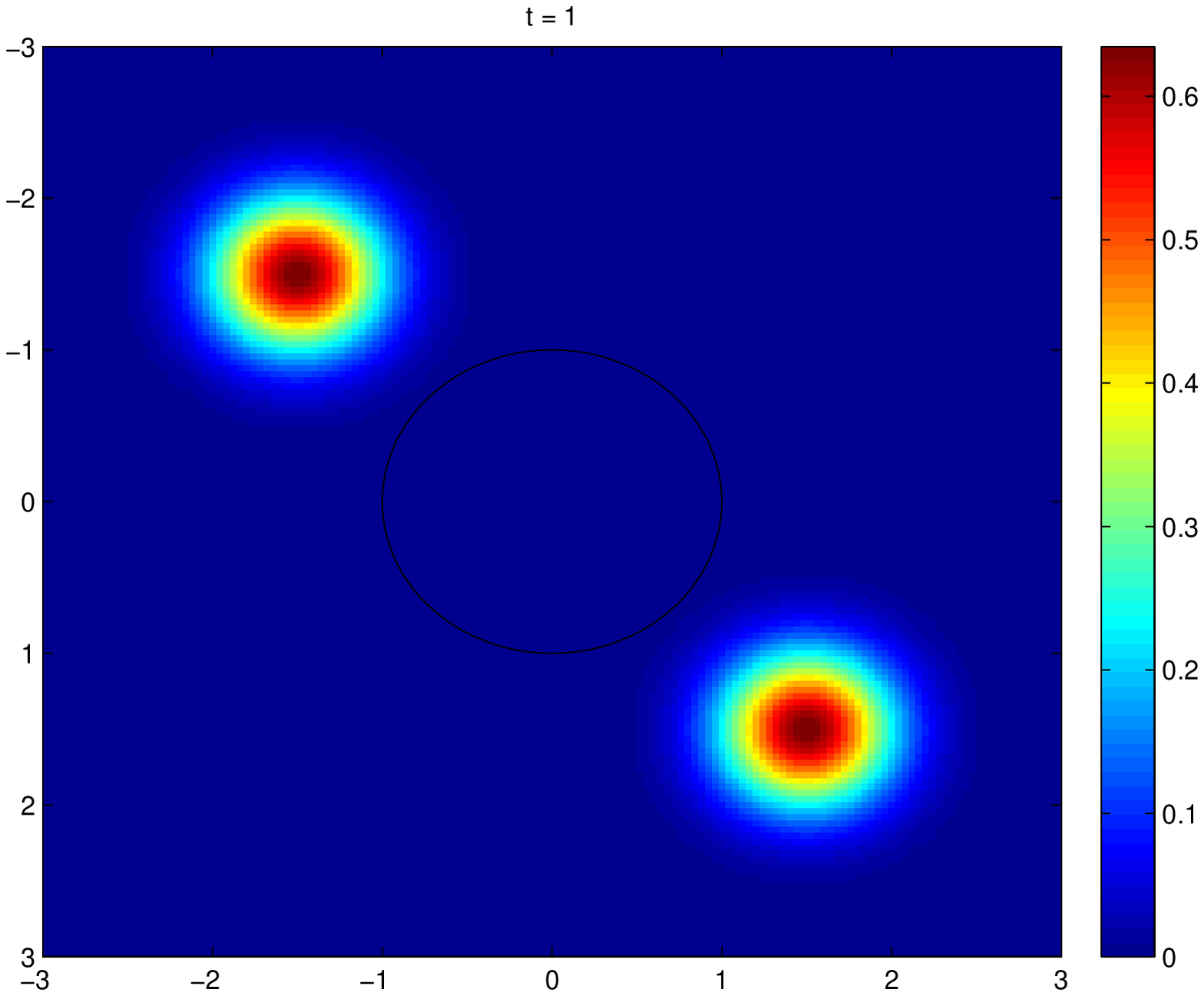}} \\
\caption{\small  The distribution $\rho(t,x) $ generating the convection on the mass for different $t = 0,0.2,0.4,0.6,0.8,1.0$. The black circle is the boundary of $B_1(0)$.}\label{exp_14}
     \end{center}
 \end{figurehere}

From the figures, we can see that our proposed method identifies simultaneously the two non-unique points closets from two mass lumps at antipodal positions to the ball in the center.  In particular, the projected measure is the average of the two Dirac masses at the boundary of the circle, where each of them is the closest point of the mass lumps.   The algorithm uses reversed time, and the reconstruction moves from the points on the balls to the two respective masses.
\end{example}

\begin{example} In this example, we consider the Hamiltonian $H_1$ above and the input distribution $\rho(x)$ as follows:
%is the same as in Example 3, again see Figure \ref{exp_5}.

\begin{figurehere}
     \begin{center}
     \vskip -0.3truecm
       \scalebox{0.4}{\includegraphics{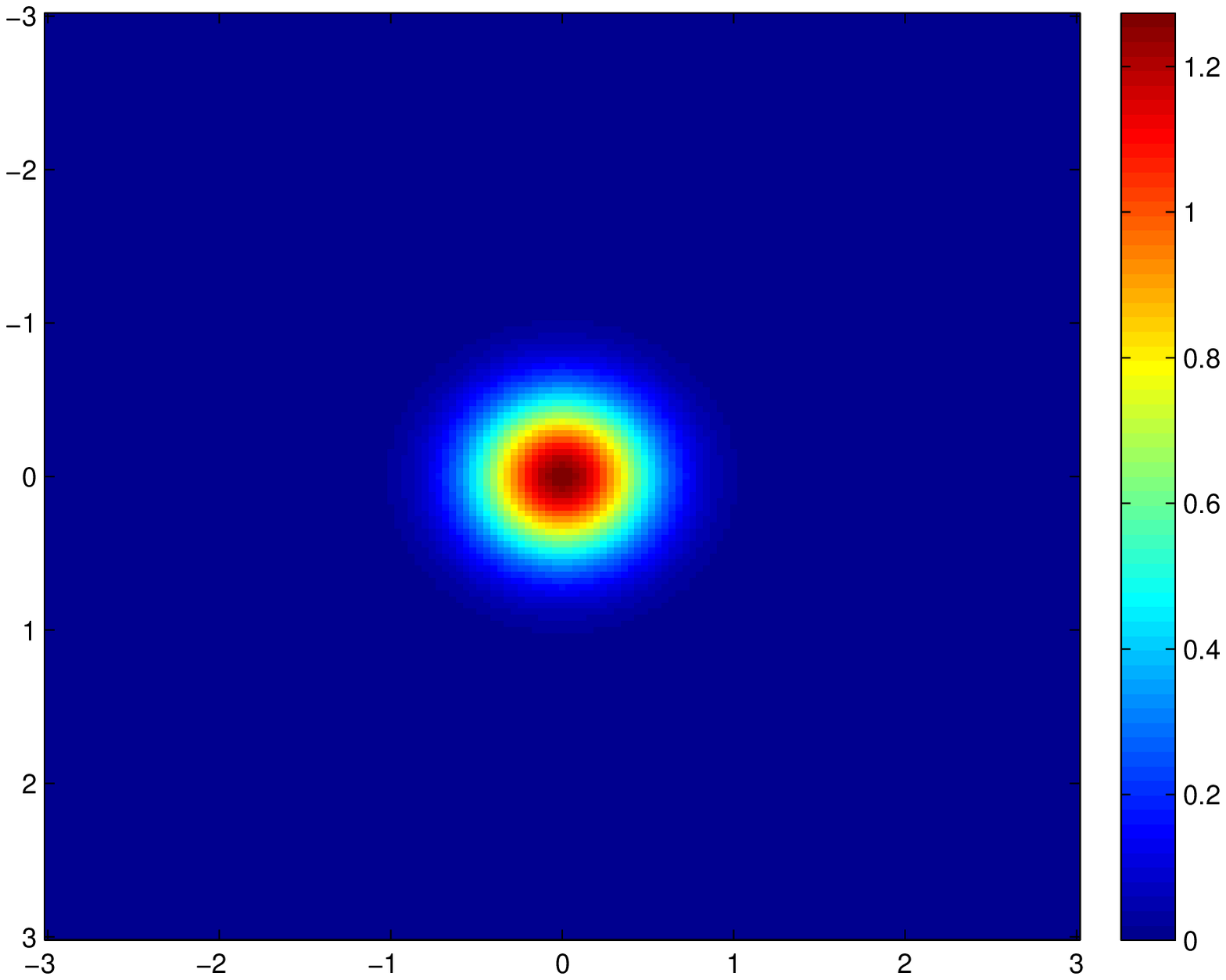}}
\caption{\small The distribution of the input $\rho$.}\label{exp_5}
     \end{center}
 \end{figurehere}

\noindent We choose the center and radius $(x_0 ,R)$, which helps to define $G(\rho)$ as $x_0 = (3,0), R = 2$. Since this is a torus, the mass sees a ``non-convex" object from both sides from afar.
Figure \ref{exp_15} gives the optimizer $\tilde{\Phi}$ (left) in \eqref{lala_hopf} and its gradient $\nabla_x \tilde{\Phi}$ (right) computed using {Algorithm 1} when $t = 1$ in the Hamiltonian.
We fix $L_i = 2$ in {Algorithm 1}.

\begin{figurehere}
     \begin{center}
     \vskip -0.3truecm
        \scalebox{0.4}{\includegraphics{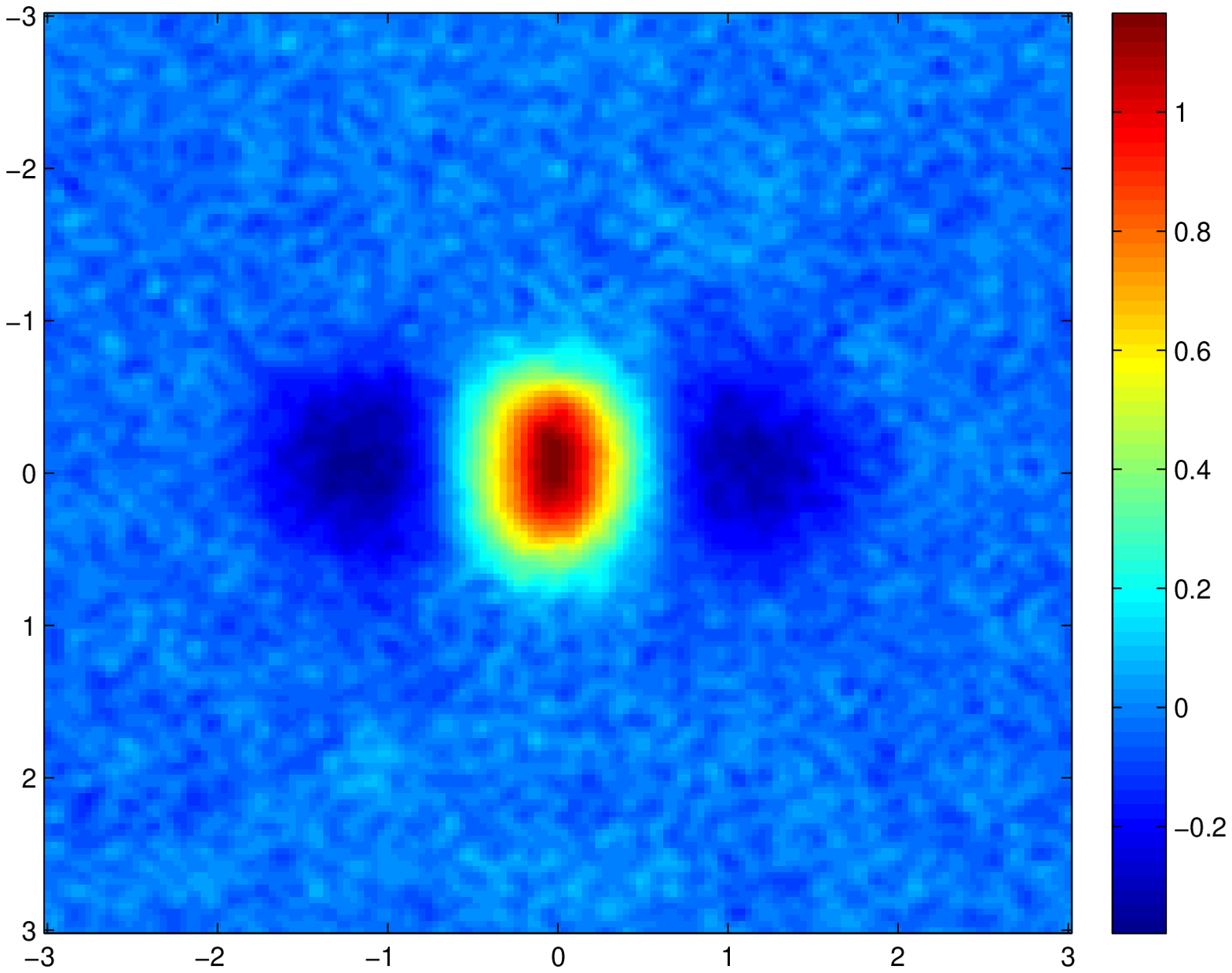}}
\scalebox{0.4}{\includegraphics{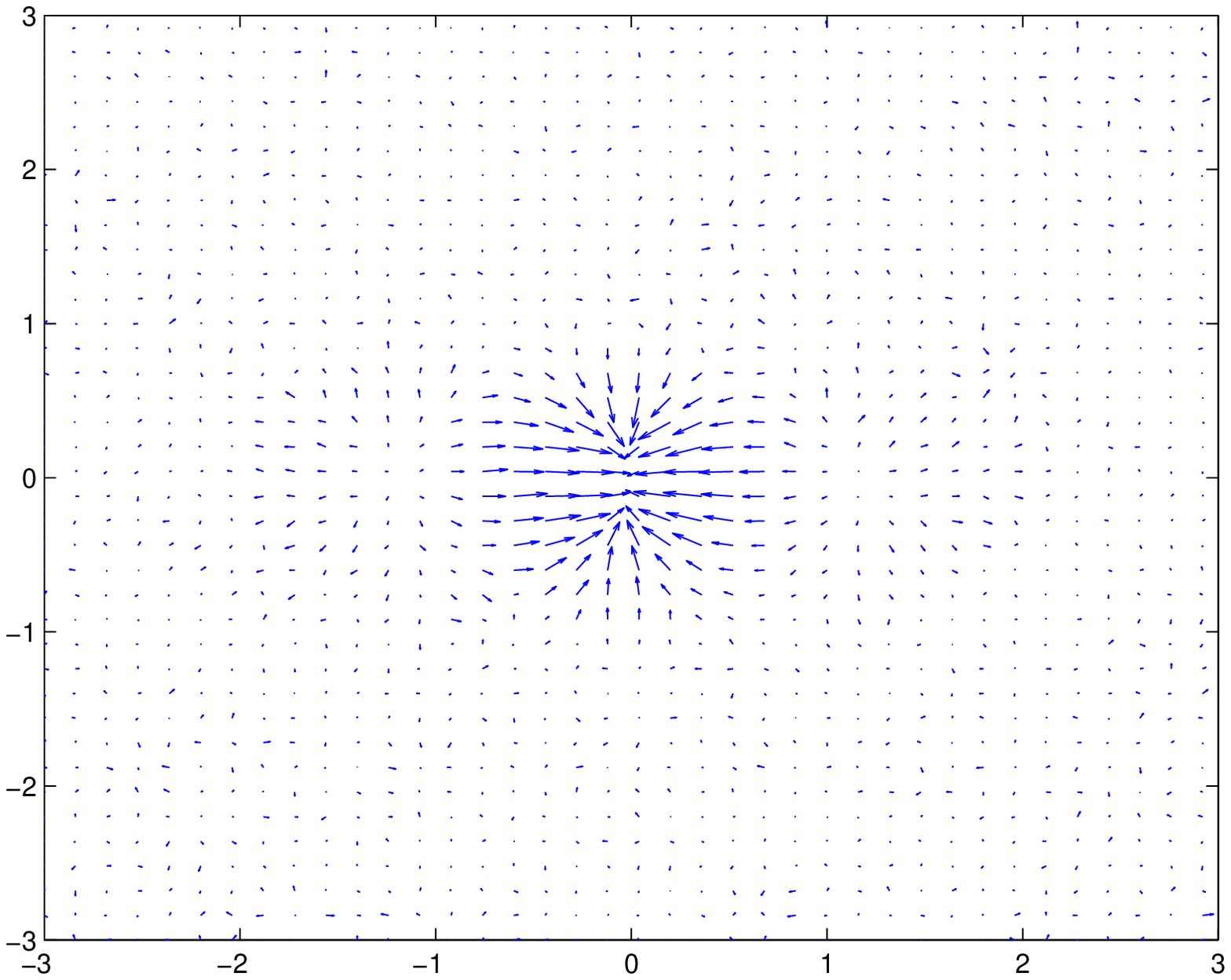}} \\
\caption{\small Left: optimizer $\tilde{\Phi}$ in $U(t,\rho)$ in \eqref{lala_hopf}, right: vector field $\nabla_x \tilde{\Phi}$.}\label{exp_15}
     \end{center}
 \end{figurehere}

In Figure \ref{exp_16}, we plot the distributions $ \rho(t,x) $ for different $t = 0,0.2,0.4,0.6,0.8,1.0$.

\begin{figurehere}
     \begin{center}
     \vskip -0.3truecm
       \scalebox{0.35}{\includegraphics{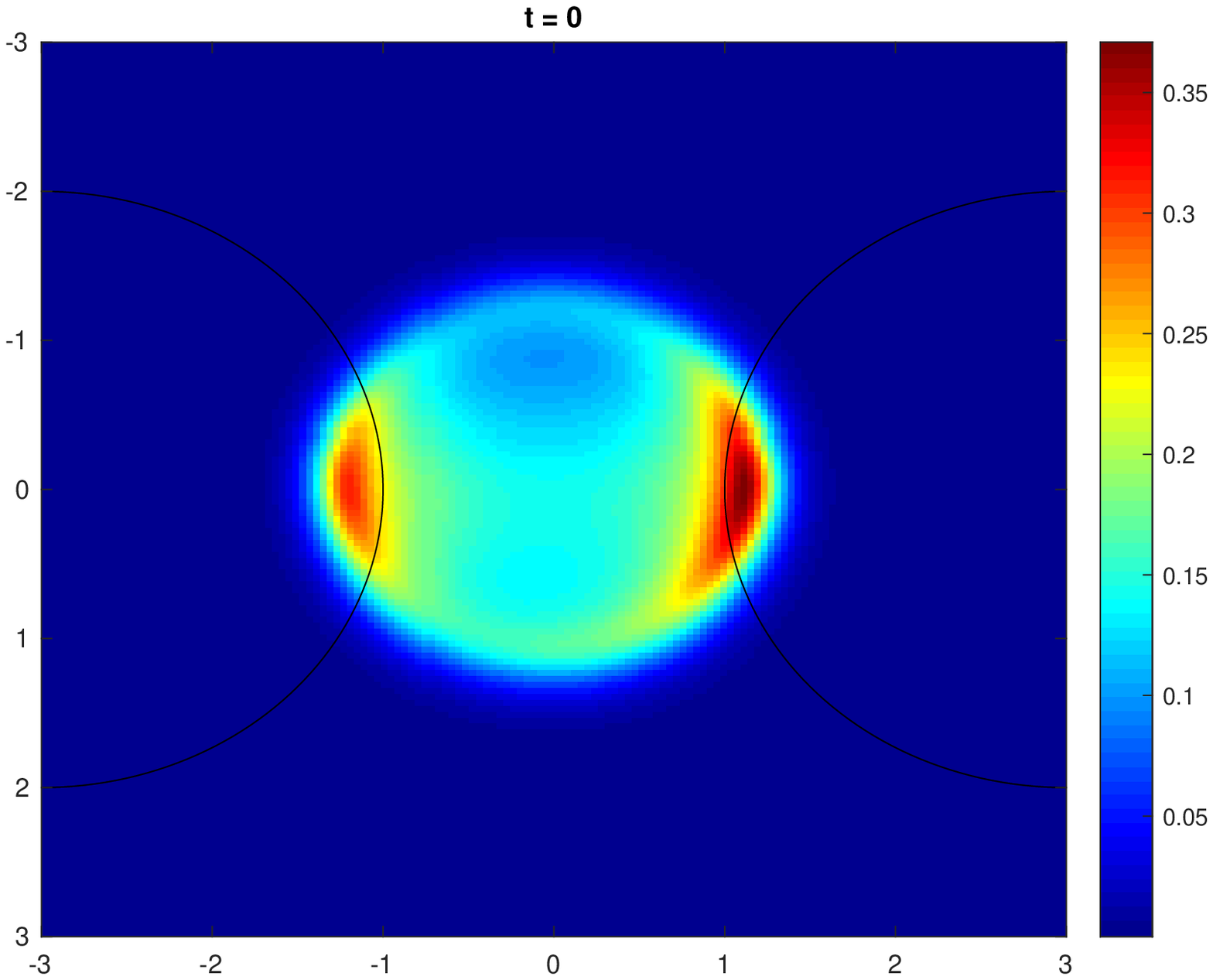}}
    \scalebox{0.35}{\includegraphics{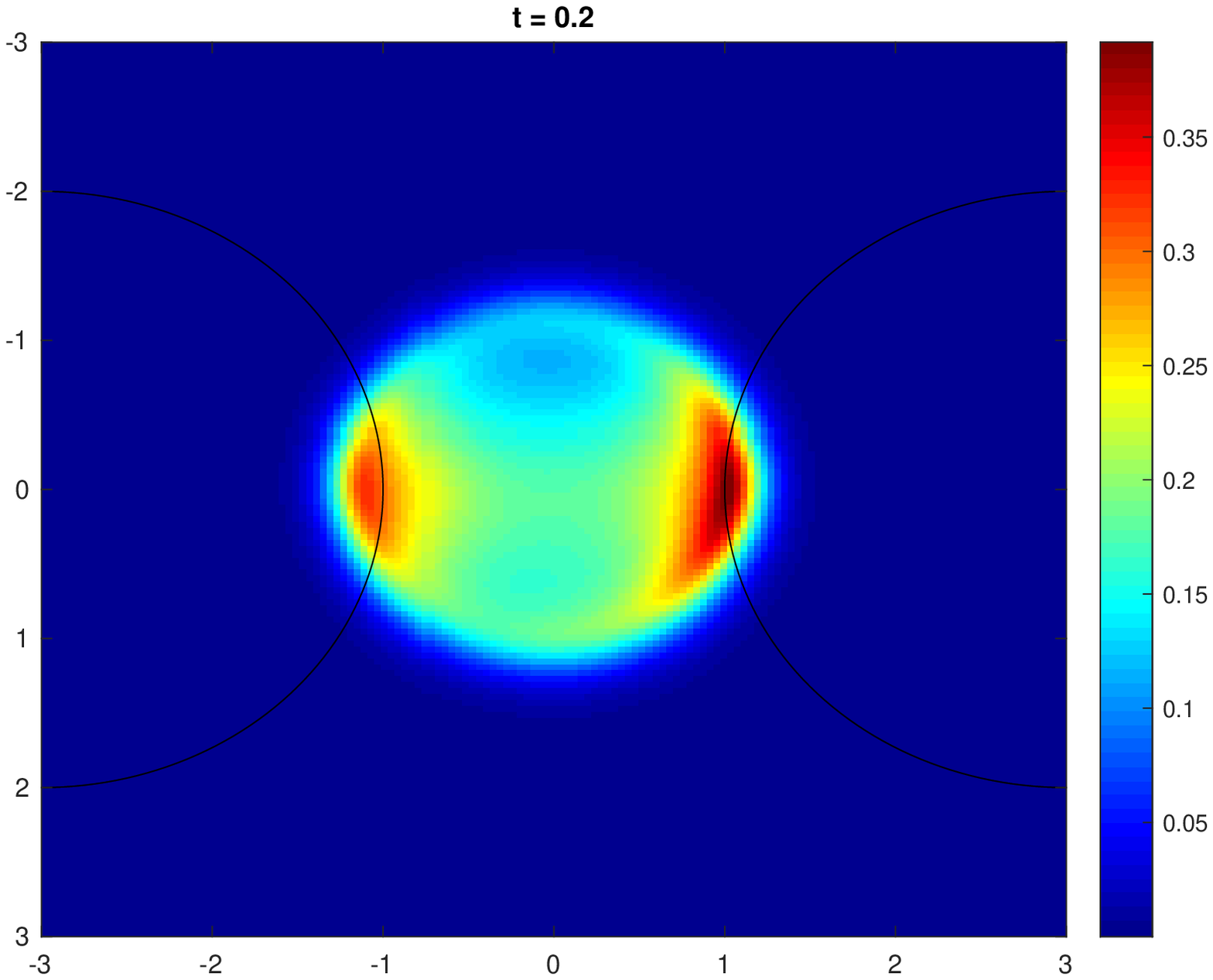}} \\
   \vskip -0.3truecm
       \scalebox{0.35}{\includegraphics{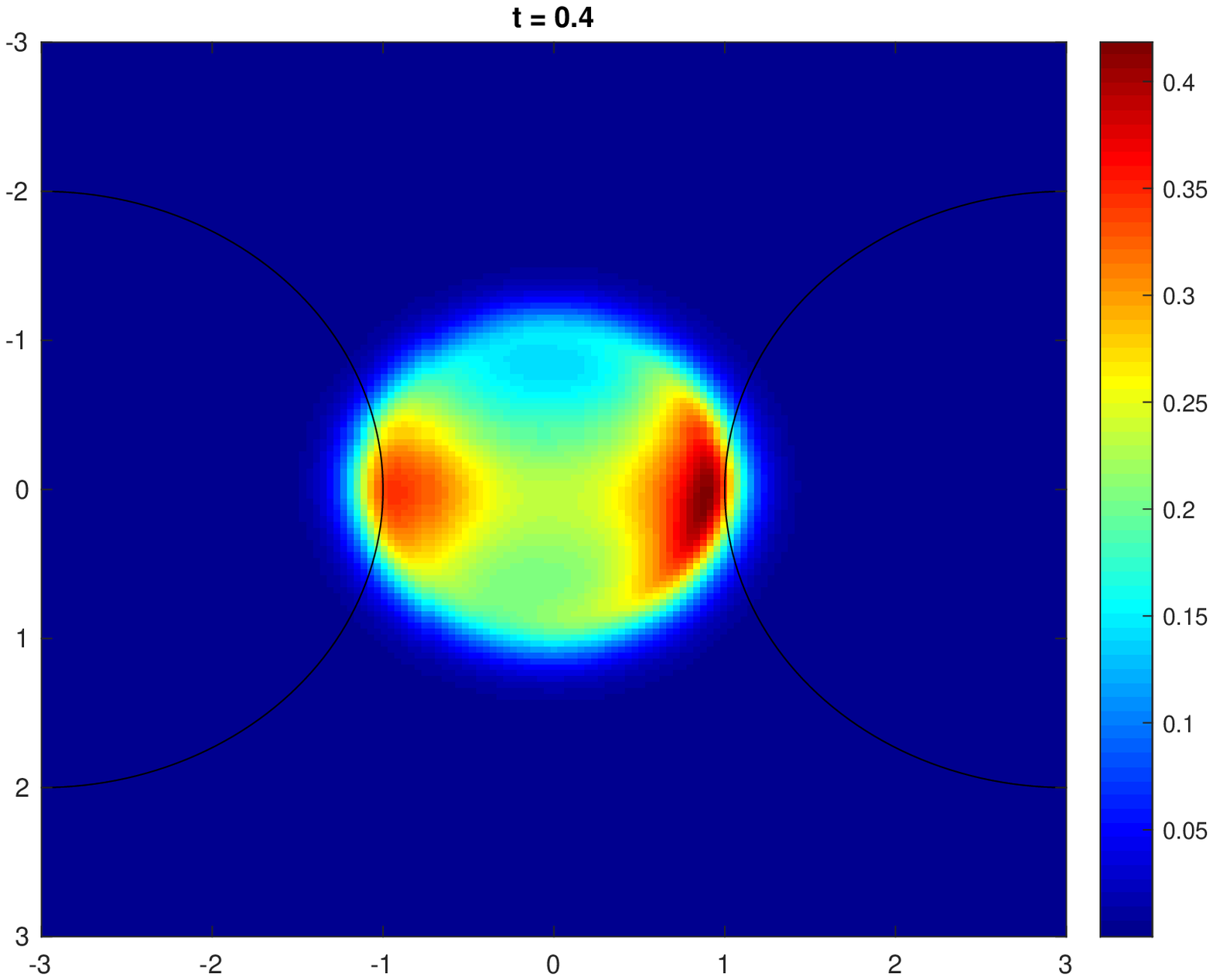}}
    \scalebox{0.35}{\includegraphics{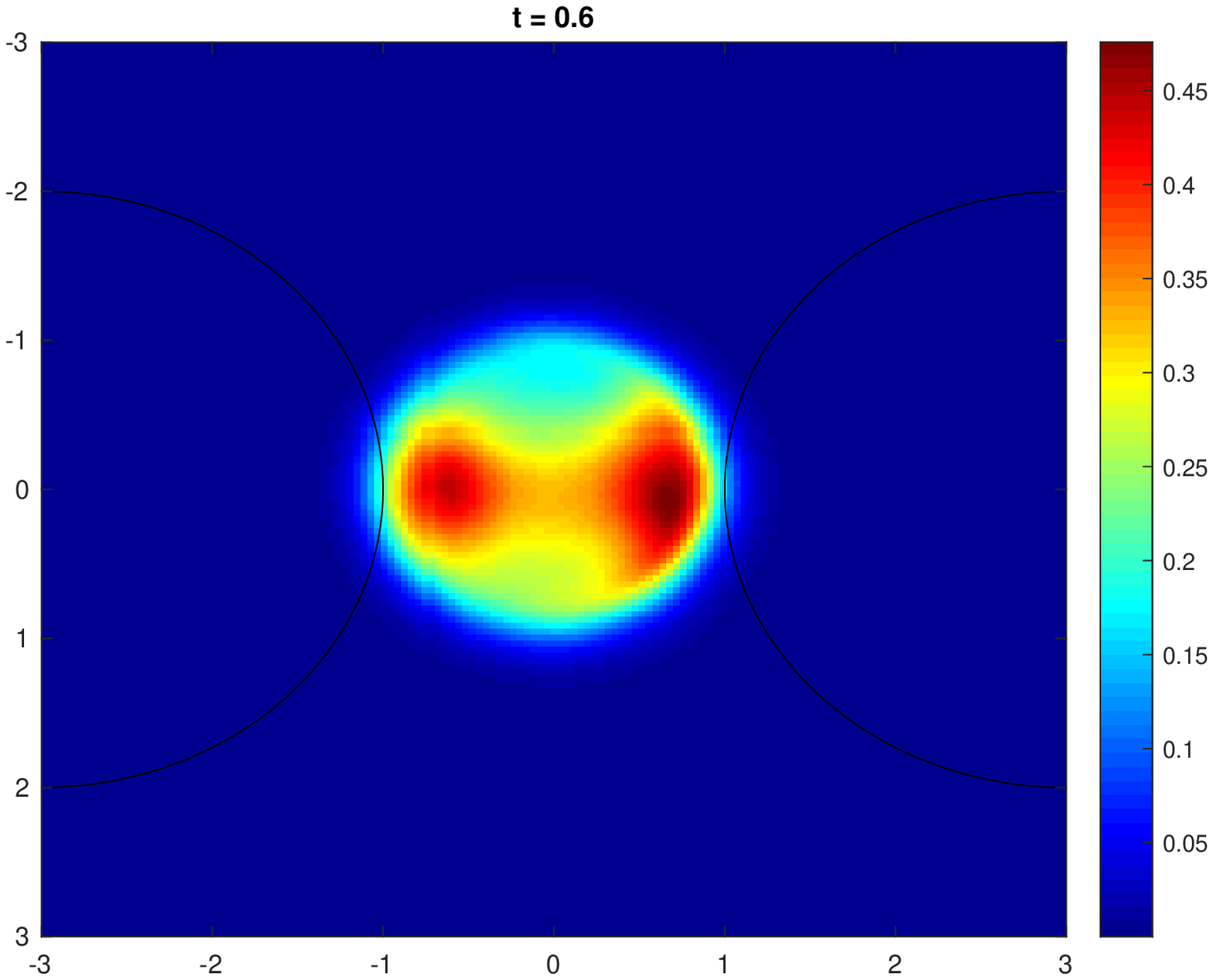}} \\
   \vskip -0.3truecm
       \scalebox{0.35}{\includegraphics{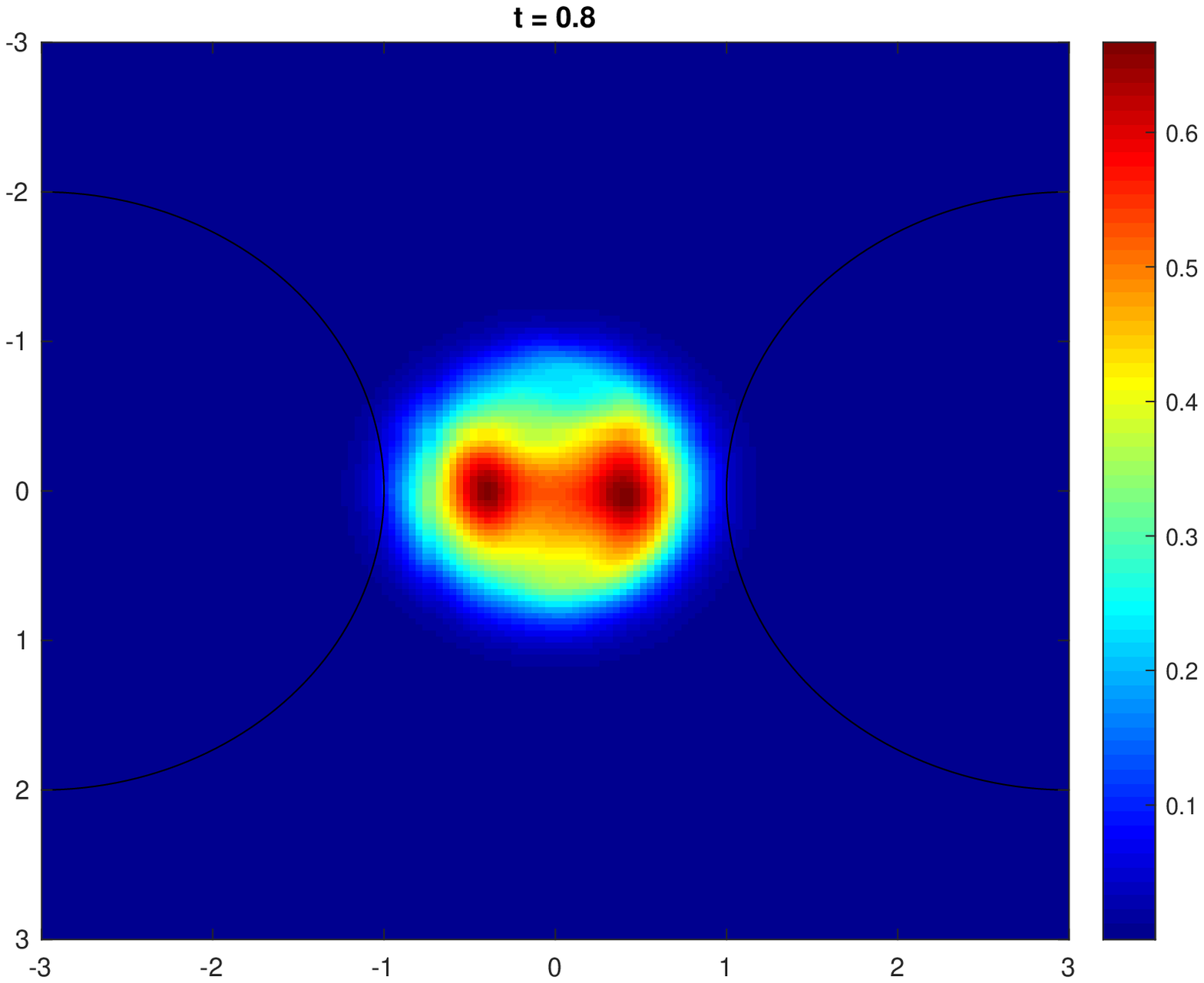}} 
         \scalebox{0.35}{\includegraphics{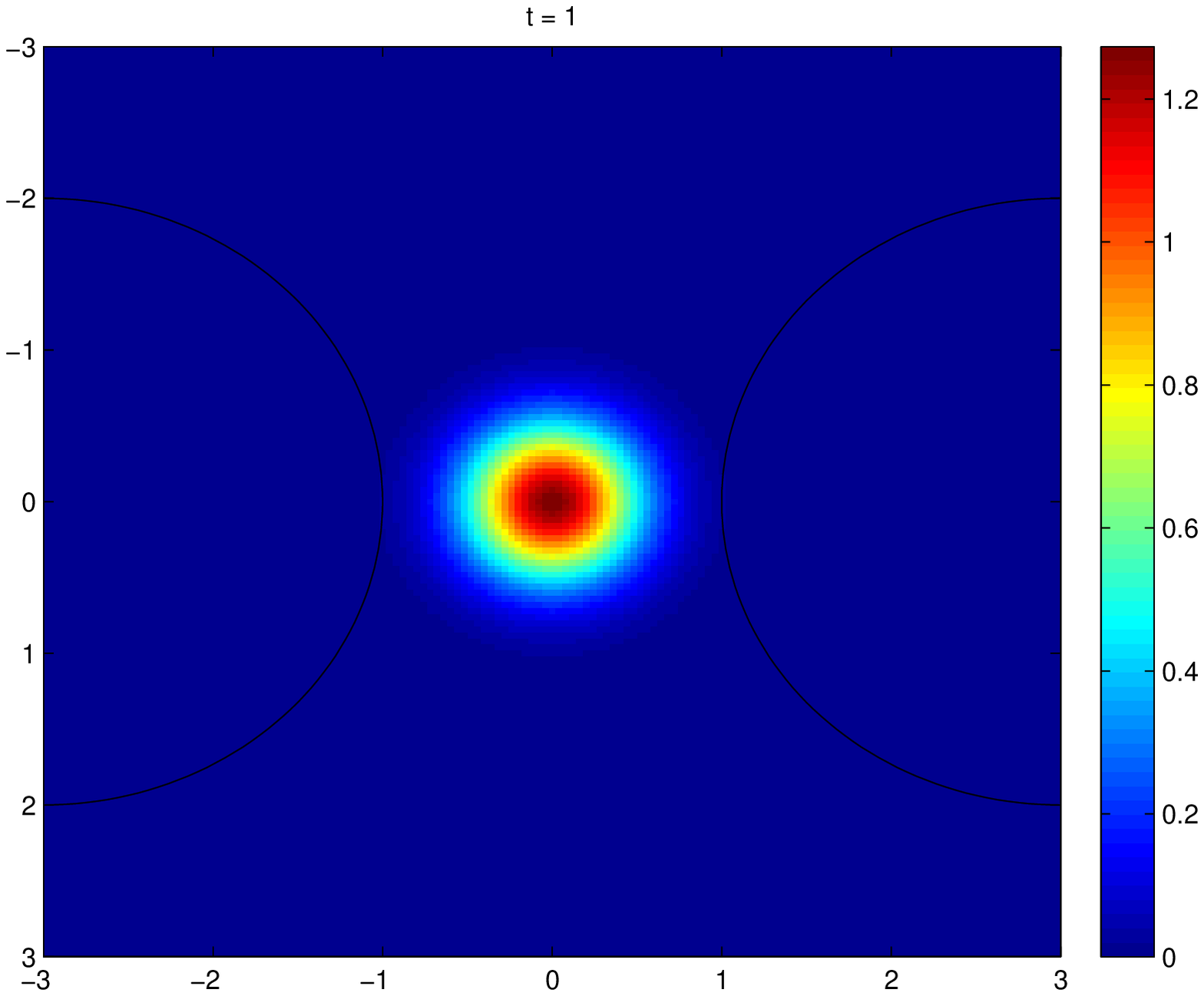}} \\
\caption{\small  The distribution $\rho(t,x) $ generating the convection on the mass for different $t = 0,0.2,0.4,0.6,0.8,1.0$. The black circle is the boundary of $B_1(0)$.}\label{exp_16}
     \end{center}
 \end{figurehere}

In this example, our method accurately finds the projections of a mass to the two non-unique closest points on the ``non-convex" body (which is in fact a ball "split" in two).  In particular, the flow of the mass splits into two opposite directions; each  brings half of the densities to the boundary of the target body.
\end{example}

\begin{example}
 In this example, we consider the non-convex Hamiltonian $H_3$ above and the input distribution $\rho(x)$ in Figure \ref{exp_0}.
\begin{figurehere}
     \begin{center}
     \vskip -0.3truecm
       \scalebox{0.4}{\includegraphics{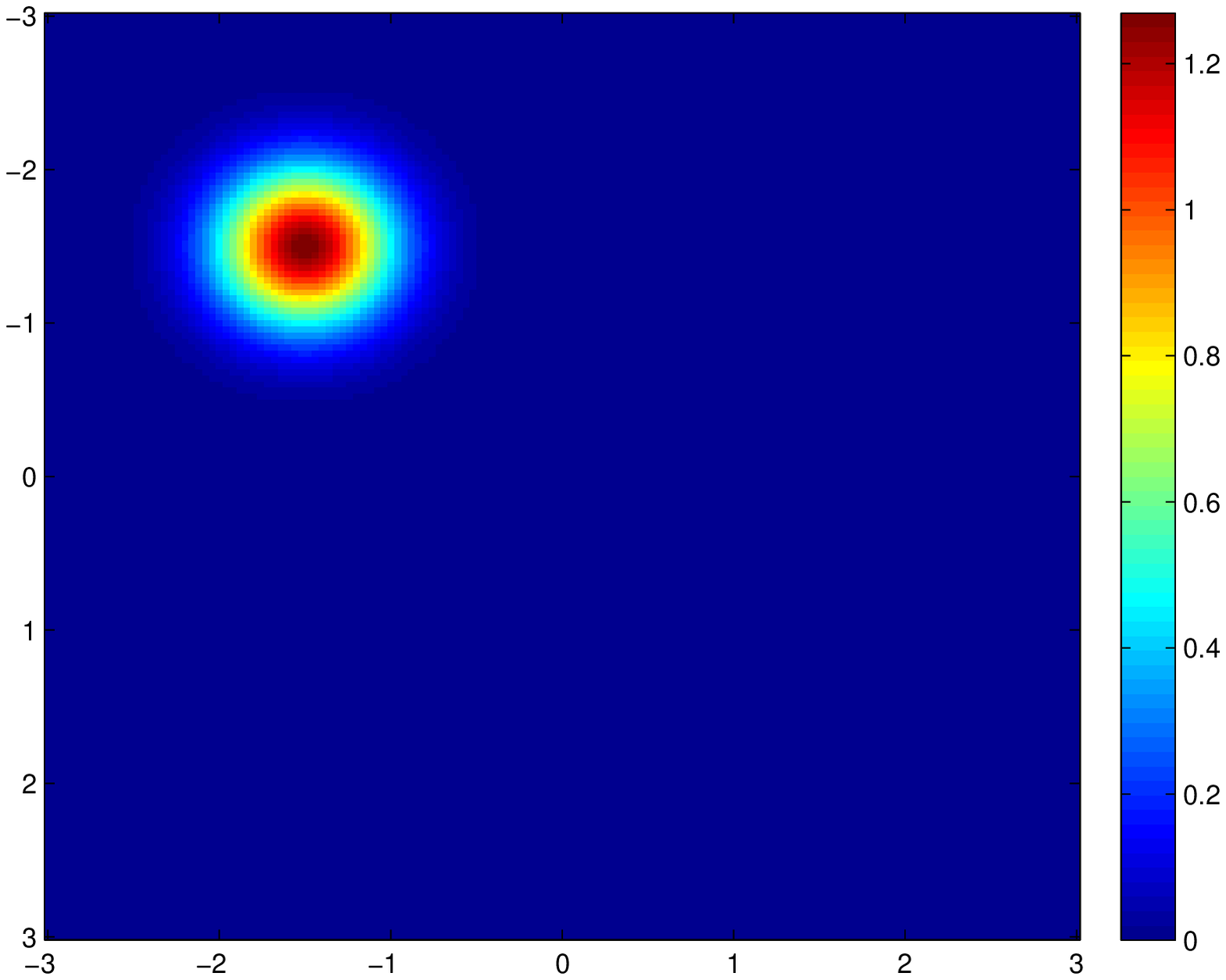}}
\caption{\small The distribution of the input $\rho$. }\label{exp_0}
     \end{center}
 \end{figurehere}

In order to compute the absolute value in the conservation law in a stable way, we replace the function with a soft absolute value as follows
\beqnx
\text{soft-abs-value}(x) := \frac{2}{\pi} \arctan (c x) \,,
\eqnx
where we choose $c=20$.  With this regularization, the Hamiltonian under consideration is actually
\beqnx
H_{3,c} (p):= H_c(p_1) + H_c(p_2)
\eqnx
where 
\beqnx
H_c(s) =   \frac{2}{\pi} s \arctan(s) - \frac{1}{c \pi} \log(1+ c^2 s^2) \,,
\eqnx
and thus the Lagrangian cost functional is
\beqnx
L_{3,c} (v):= L_c(v_1) + L_c(v_2),
\eqnx
where 
\beqnx
L_c(t) =    
\begin{cases}
\frac{1}{c \pi} \log\left( \sec^2\left(  \frac{\pi t}{2} \right)\right)  & \text{ if } |t| <1\,, \\
\infty & \text{ if } |t| \geq 1\,. \\
\end{cases}
\eqnx

\noindent  As in {Example 1}, we choose the center and radius $(x_0 ,R)$ which helps to define $G(\rho)$ as $x_0 = (0,0), R = 1$.  Figure \ref{exp_17} gives the optimizer $\tilde{\Phi}$ (left) in  \eqref{Smart} and its gradient $\nabla_x \tilde{\Phi}$ (right) computed using {Algorithm 1} when $t = 1$ in the Hamiltonian.

\begin{figurehere}
     \begin{center}
     \vskip -0.3truecm
        \scalebox{0.4}{\includegraphics{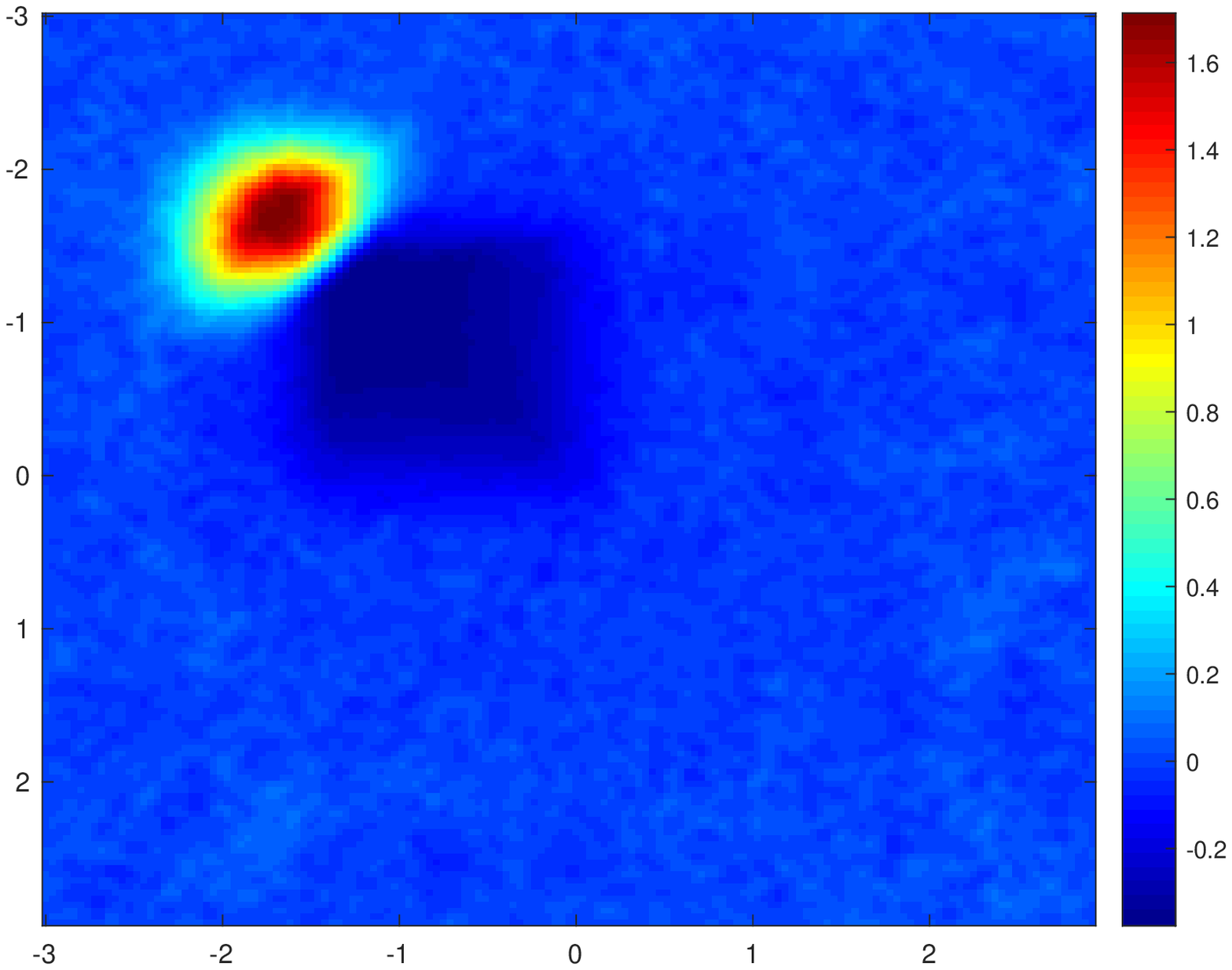}}
\scalebox{0.4}{\includegraphics{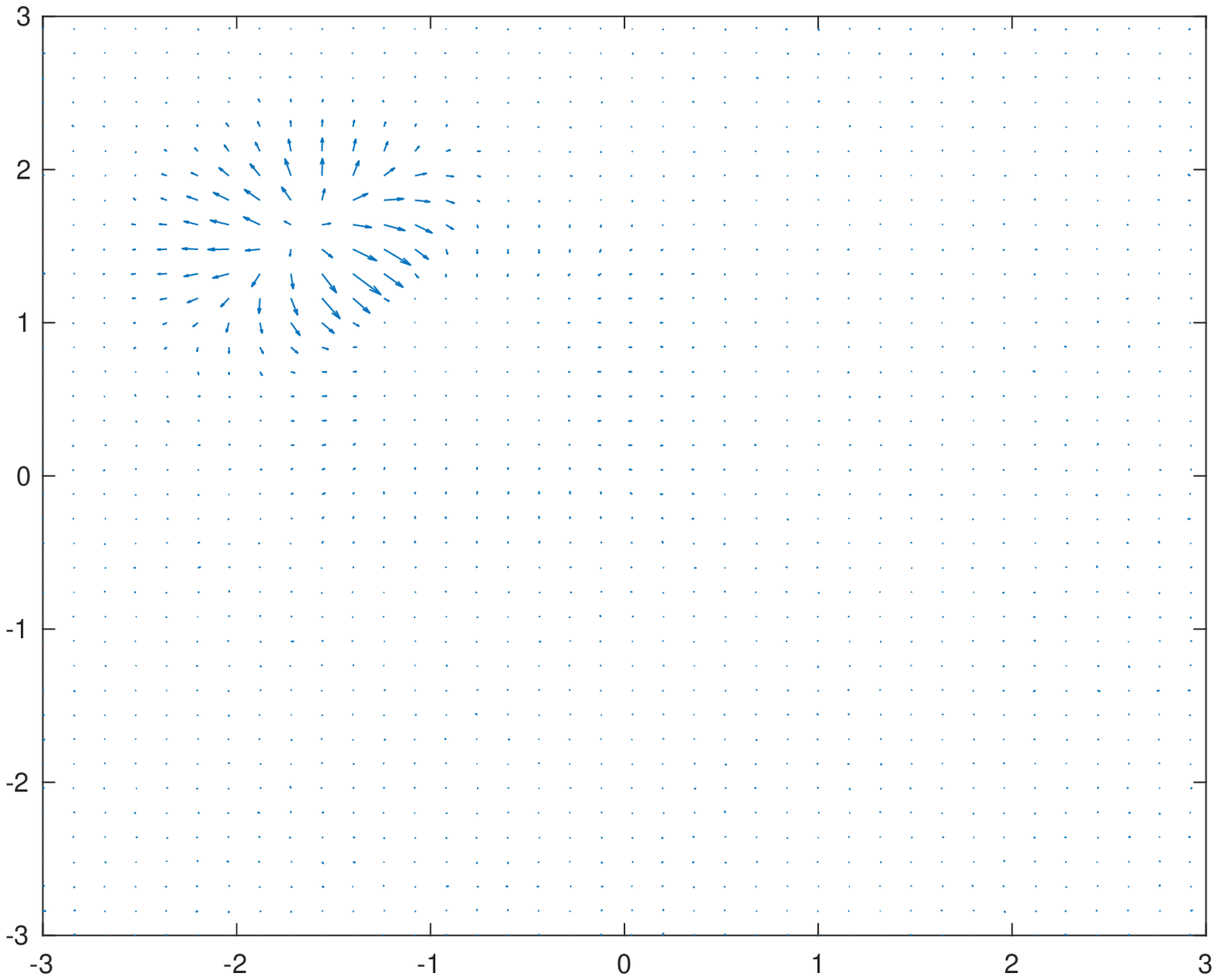}} \\
\caption{\small Left: optimizer $\tilde{\Phi}$ in $U(t,\rho)$ in \eqref{lala_hopf}, right: vector field $\nabla_x \tilde{\Phi}$.}\label{exp_17}
     \end{center}
 \end{figurehere}

\noindent Figure \ref{exp_18} plots the distributions $ \rho(t,x) $ for different $t = 0,0.2,0.4,0.6,0.8,1.0$.

\begin{figurehere}
     \begin{center}
     \vskip -0.3truecm
       \scalebox{0.35}{\includegraphics{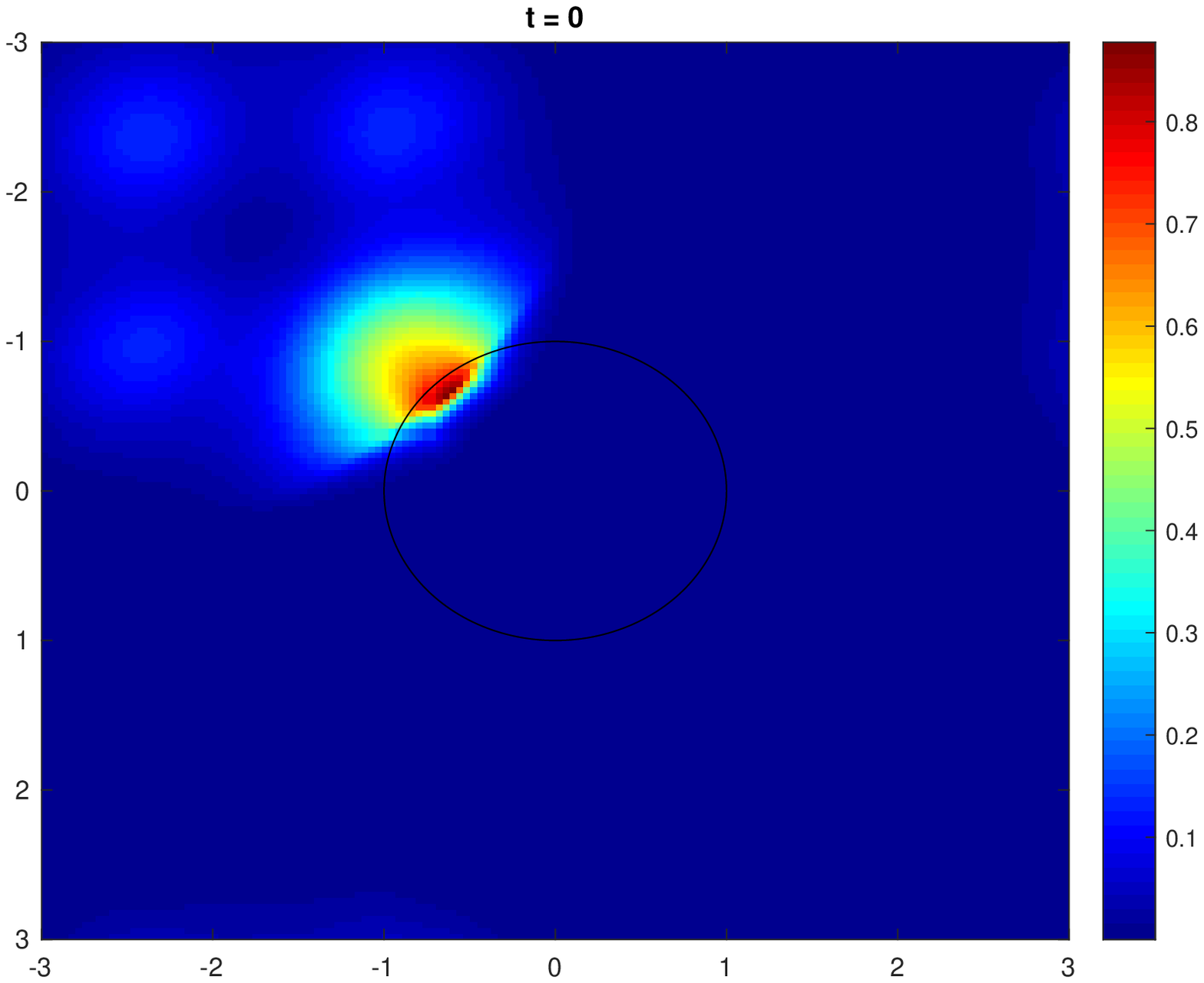}}
    \scalebox{0.35}{\includegraphics{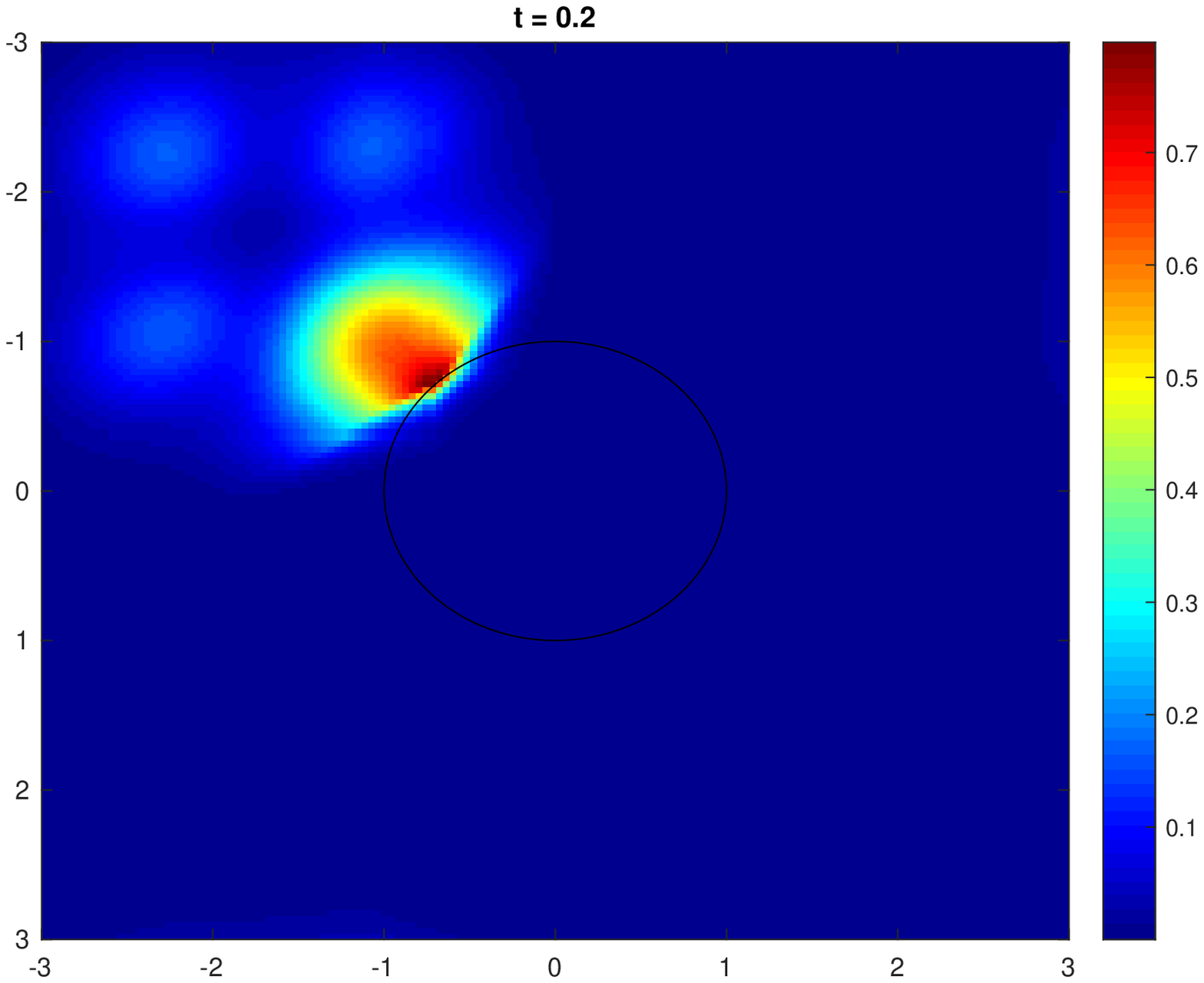}} \\
   \vskip -0.3truecm
       \scalebox{0.35}{\includegraphics{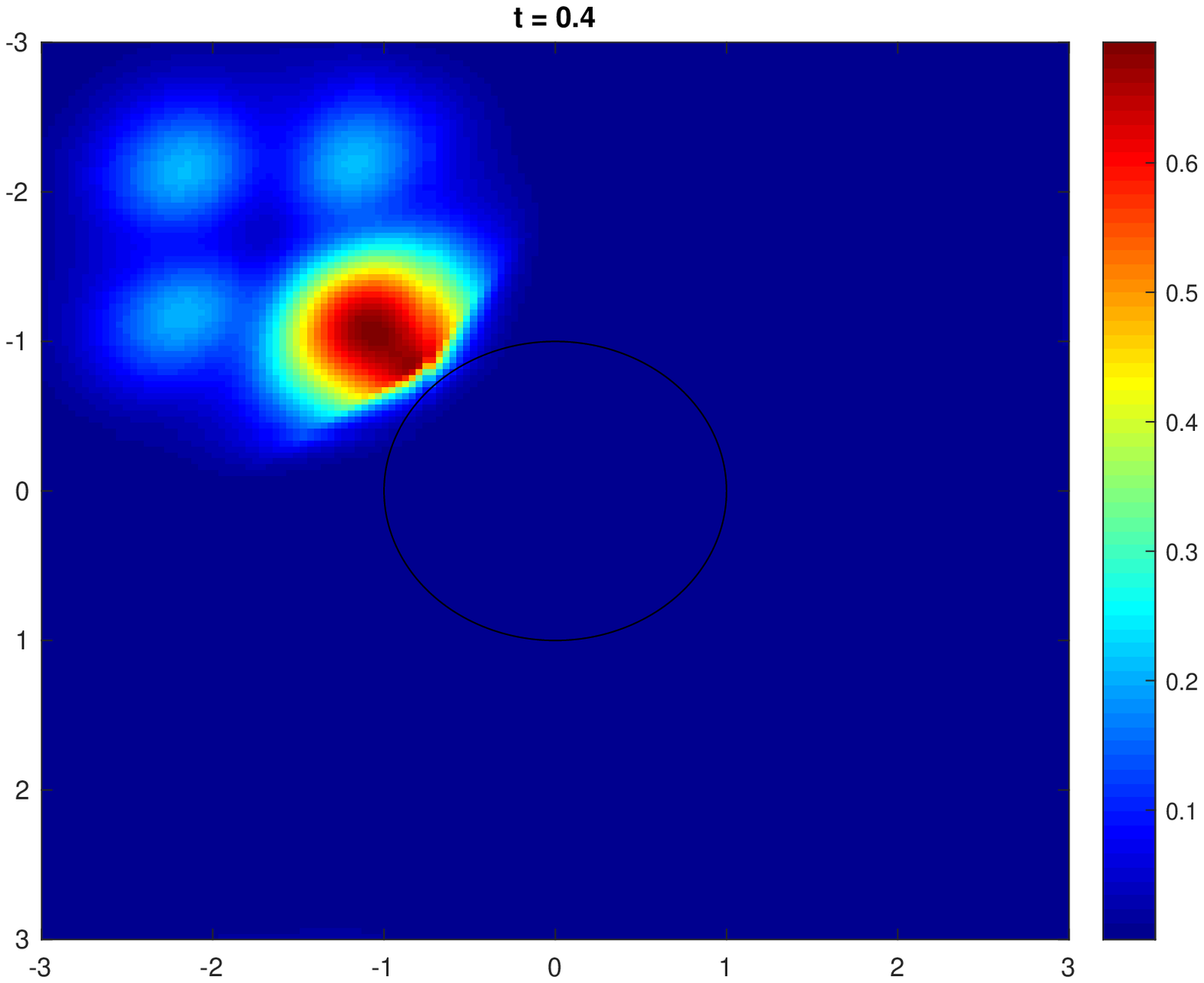}}
    \scalebox{0.35}{\includegraphics{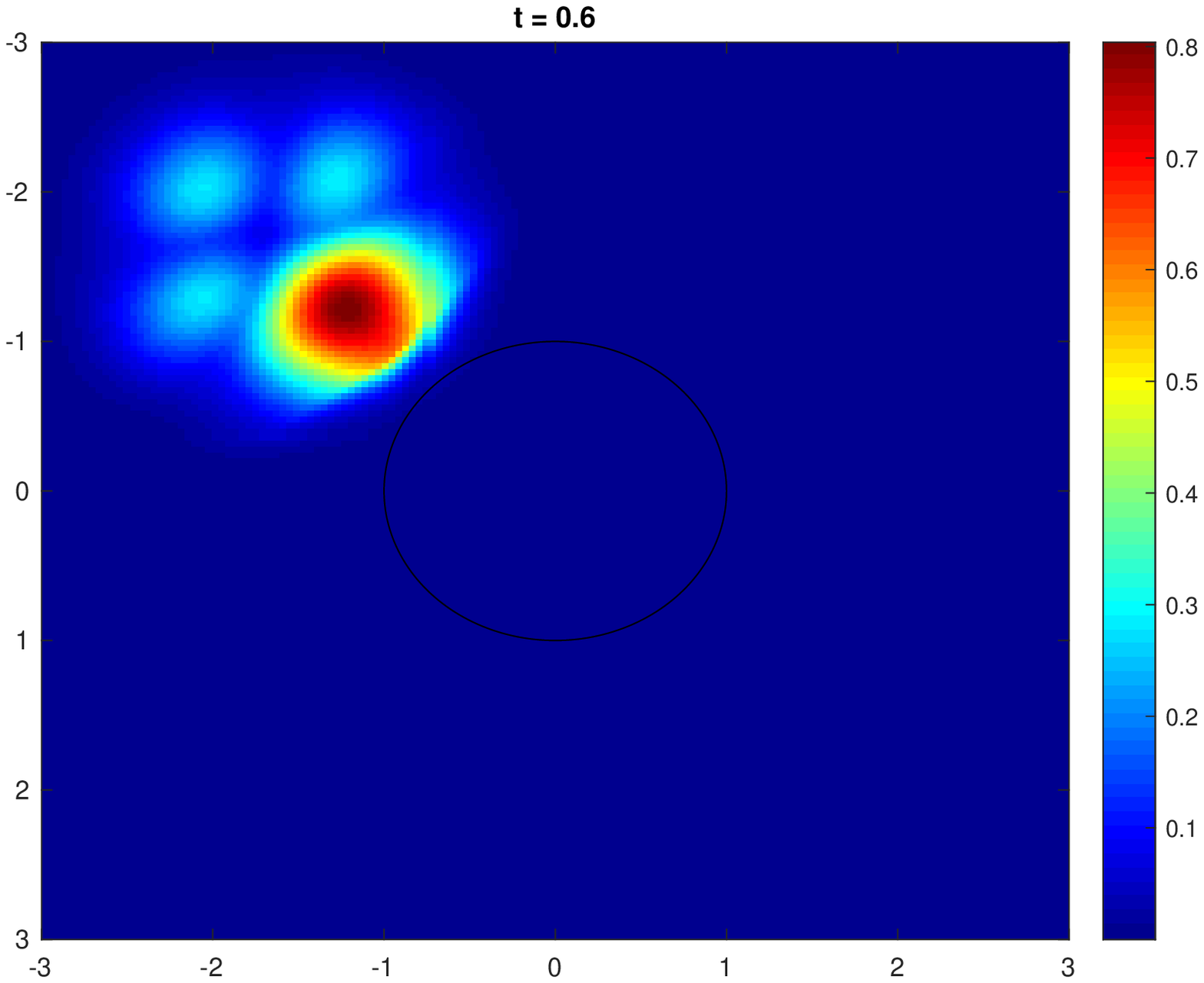}} \\
   \vskip -0.3truecm
       \scalebox{0.35}{\includegraphics{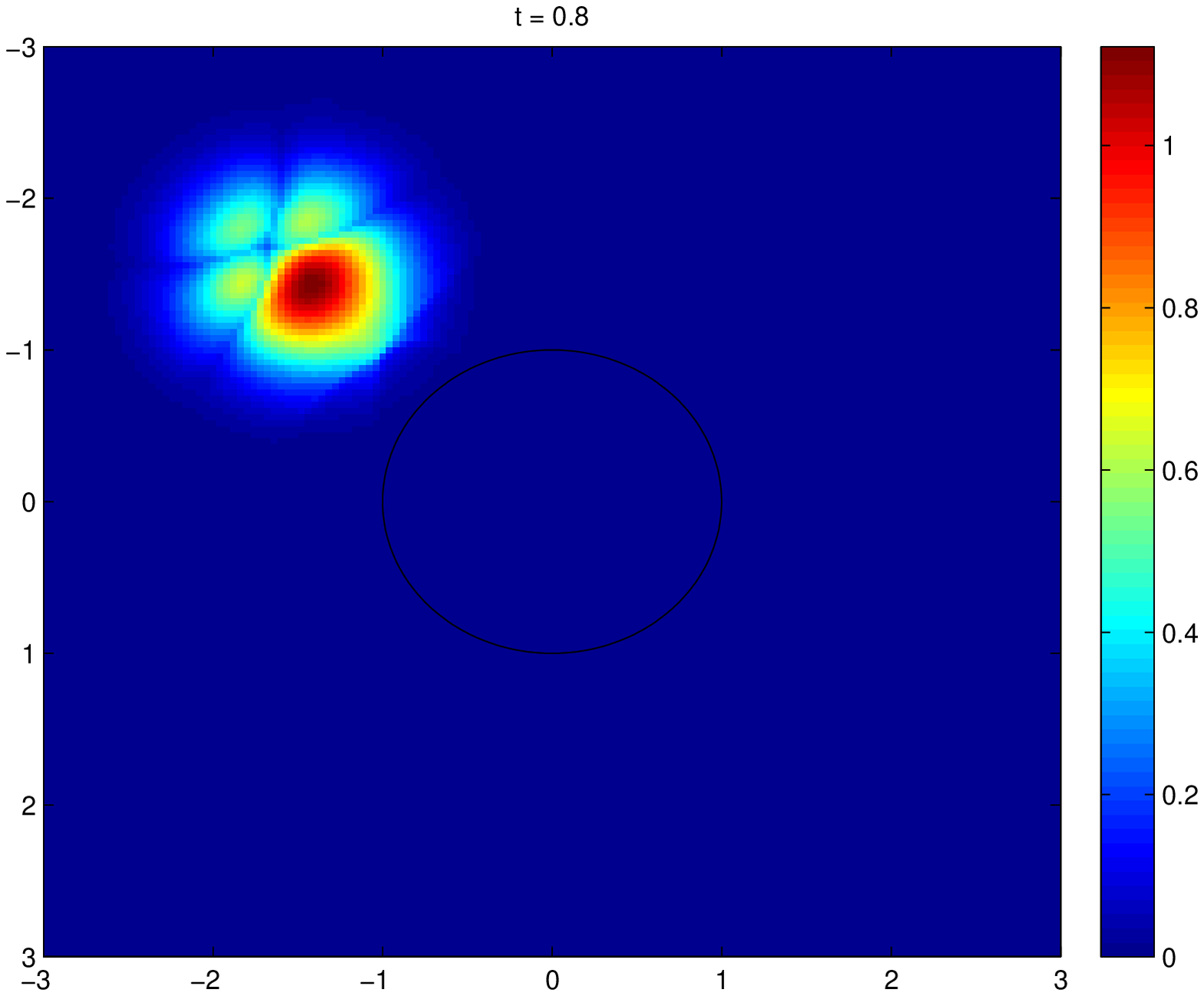}} 
         \scalebox{0.35}{\includegraphics{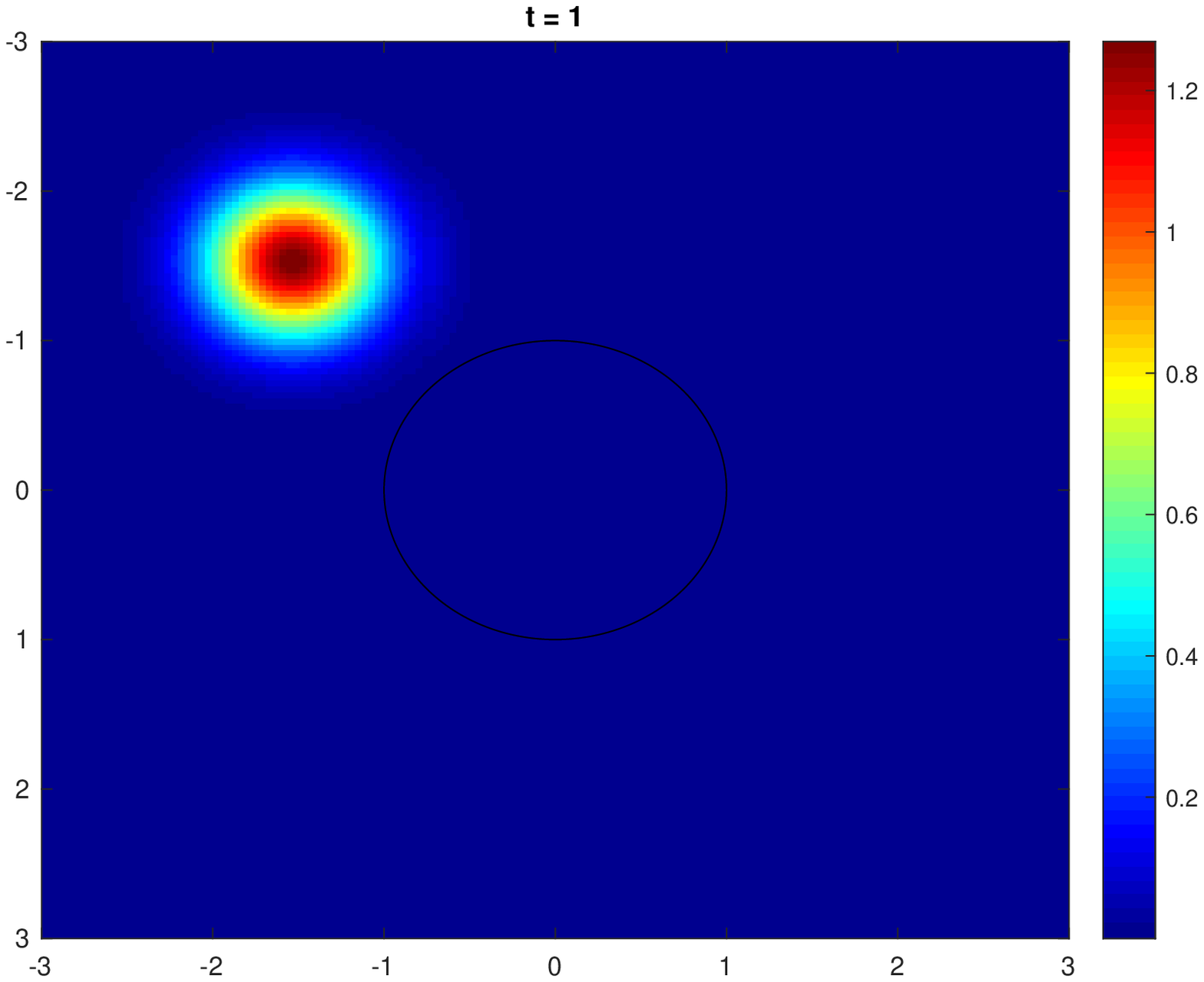}} \\
\caption{\small  The distribution $\rho(t,x) $ generating the convection on the mass for different $t = 0,0.2,0.4,0.6,0.8,1.0$. The black circle is the boundary of $B_1(0)$.}\label{exp_18}
     \end{center}
 \end{figurehere}

We identify reachability of the measure to the boundary of the ball w.r.t. to the $l_1$ Hamiltonian.
With our choice of regularization, we, however, see a defect in our numerical computation: there are three small tails that are left behind in the conservation law as the mass is moving since the exact cutoff of the absolute value is regularized.   Nonetheless, the solution makes perfect sense in term of identifying reachability.
\end{example}

\begin{example}
 In this example, we consider the non-convex Hamiltonian $H_2$ above and the input distribution $\rho(x)$ the same as in Example 1 in Figure \ref{exp_4}.

\noindent  Again, we choose the center and radius $(x_0 ,R)$, which helps to define $G(\rho)$ as $x_0 = (0,0), R = 1$. Figure \ref{exp_17} gives the optimizer $\tilde{\Phi}$ (left) in \eqref{Smart} and its gradient $\nabla_x \tilde{\Phi}$ (right)computed using {Algorithm 1} when $t = 1$ in the Hamiltonian.

\begin{figurehere}
     \begin{center}
     \vskip -0.3truecm
        \scalebox{0.4}{\includegraphics{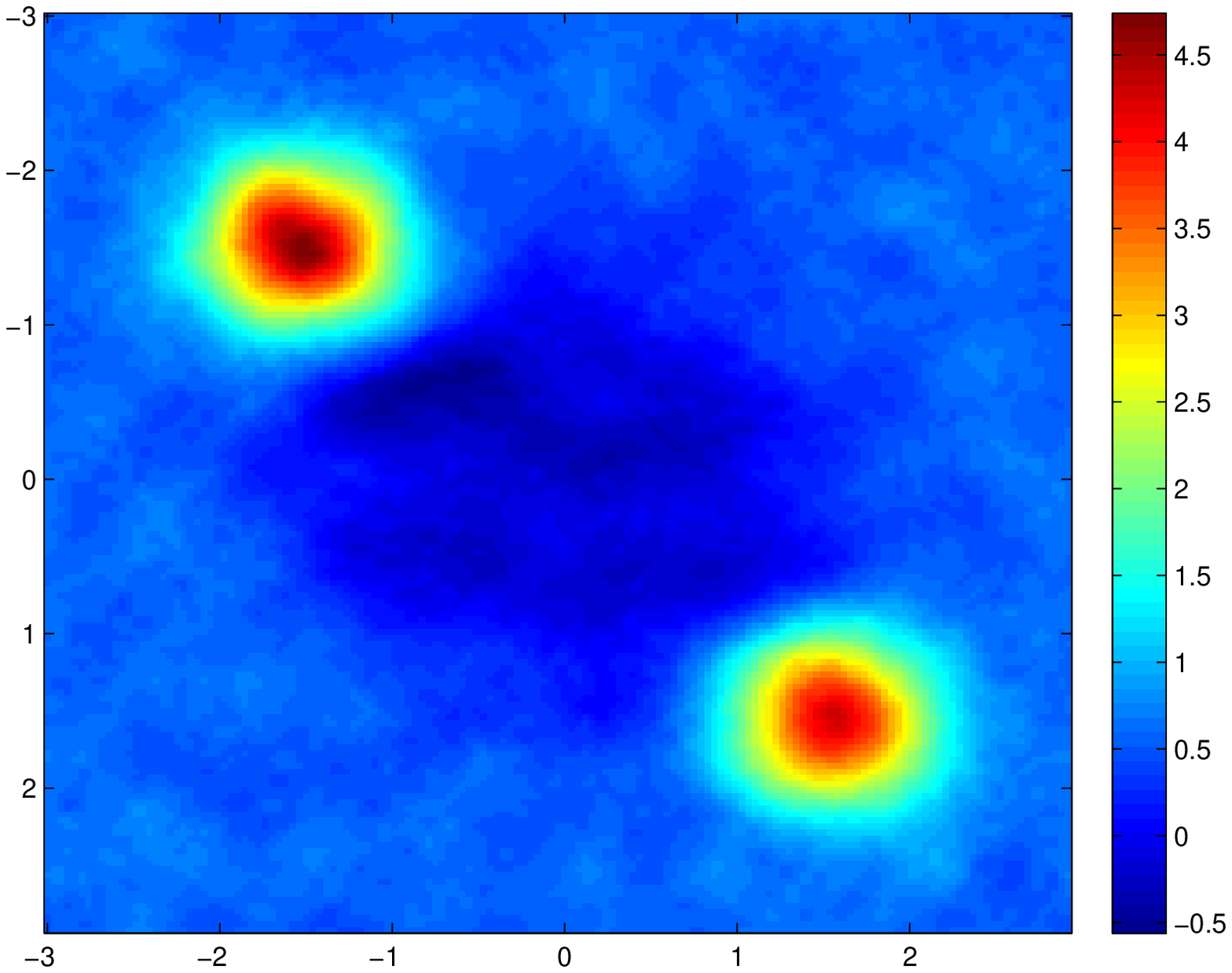}}
\scalebox{0.4}{\includegraphics{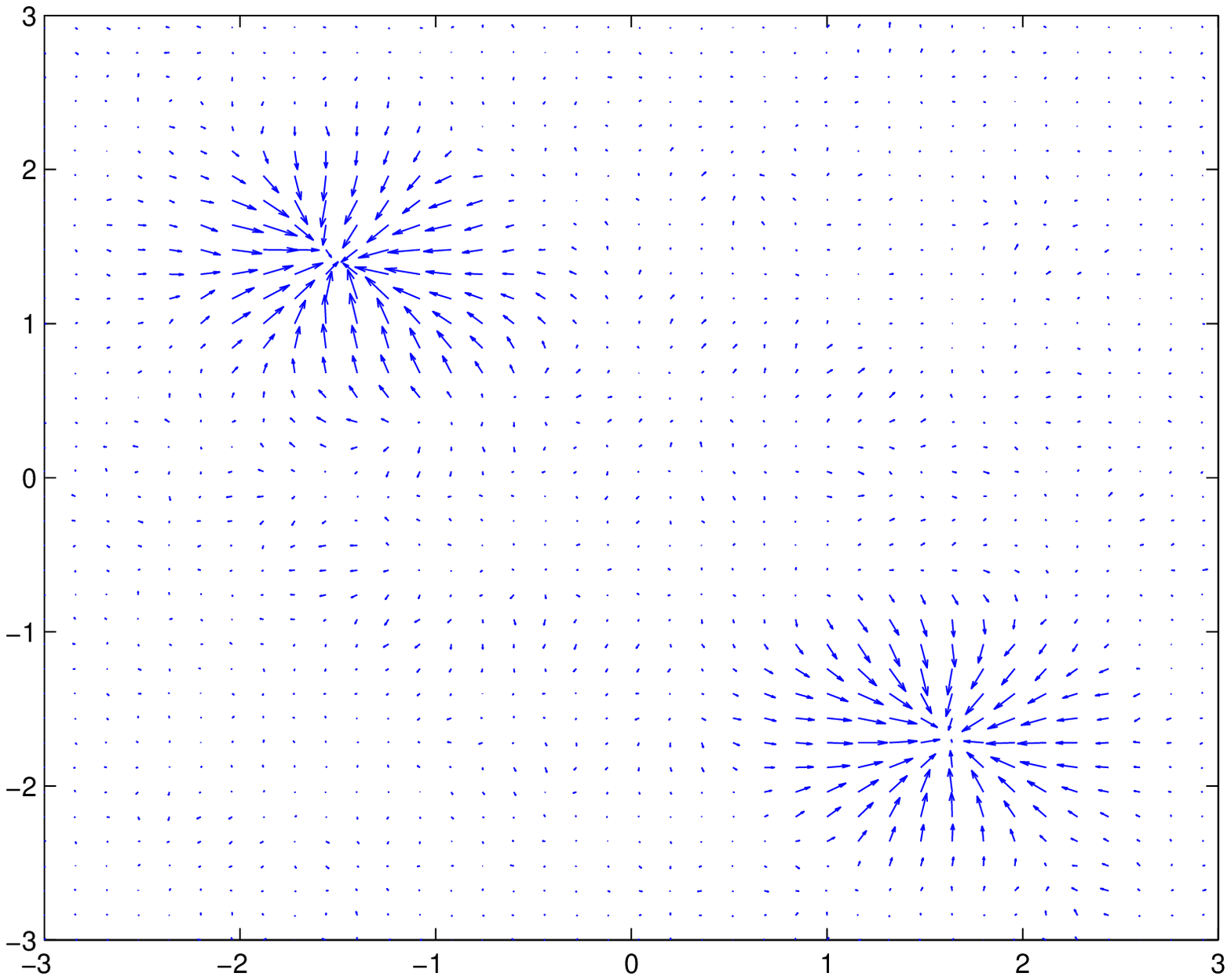}} \\
\caption{\small Left: optimizer $\tilde{\Phi}$ in $U(t,\rho)$ in \eqref{lala_hopf}, right: vector field $\nabla_x \tilde{\Phi}$.}\label{exp_17}
     \end{center}
 \end{figurehere}

\noindent Figure \ref{exp_18} plots the distributions $ \rho(t,x) $ for different $t = 0,0.2,0.4,0.6,0.8,1.0$.

\begin{figurehere}
     \begin{center}
     \vskip -0.3truecm
       \scalebox{0.35}{\includegraphics{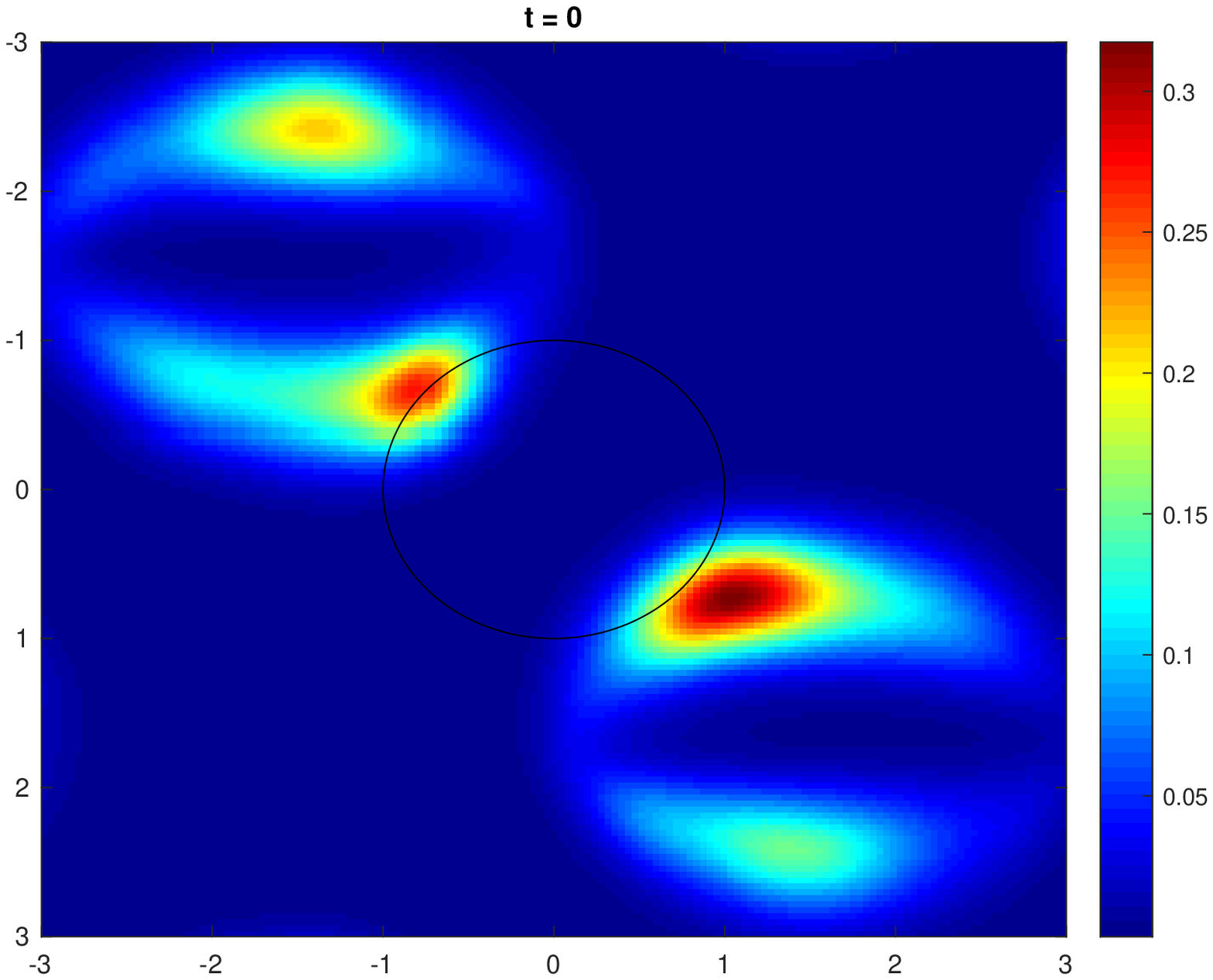}}
    \scalebox{0.35}{\includegraphics{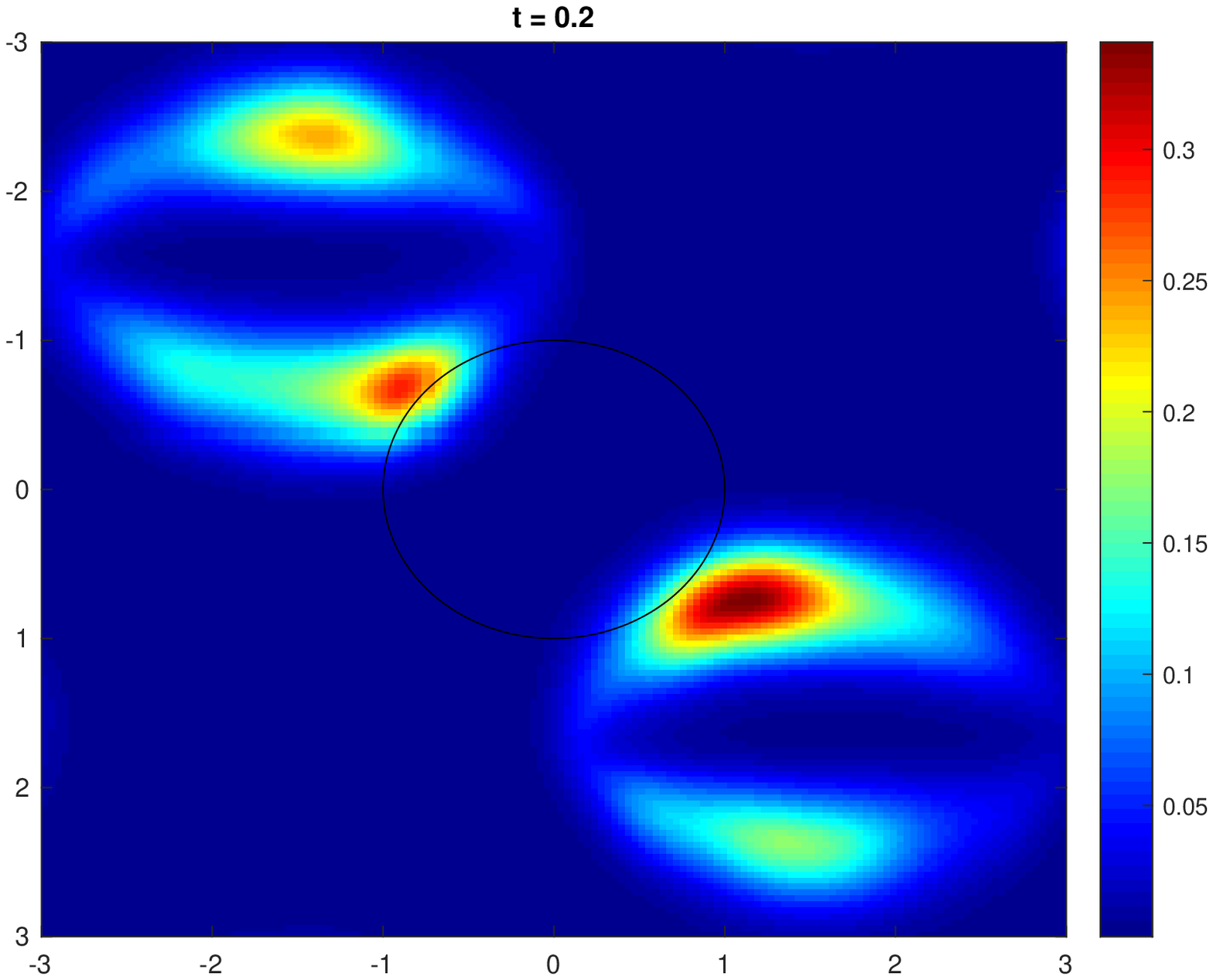}} \\
   \vskip -0.3truecm
       \scalebox{0.35}{\includegraphics{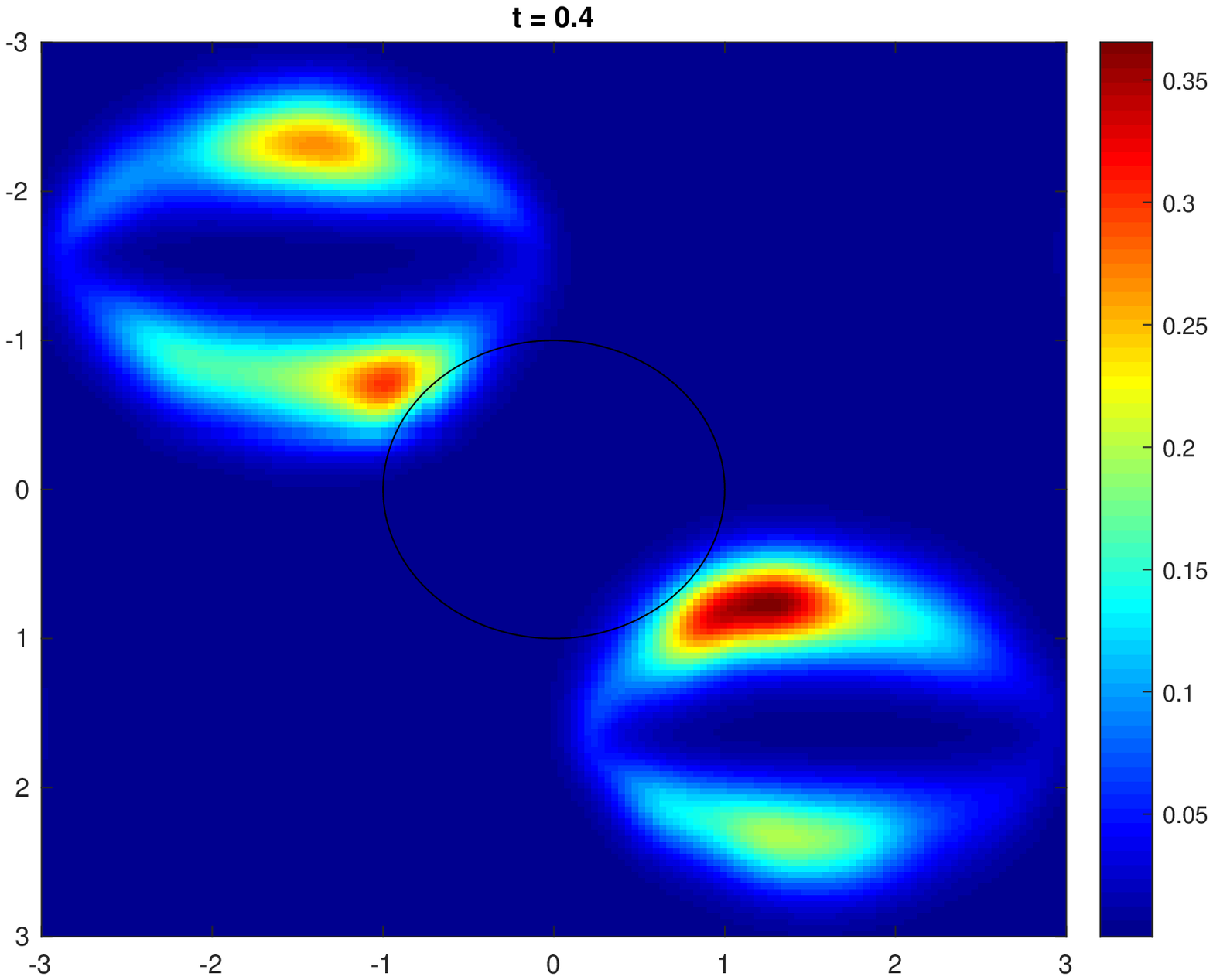}}
    \scalebox{0.35}{\includegraphics{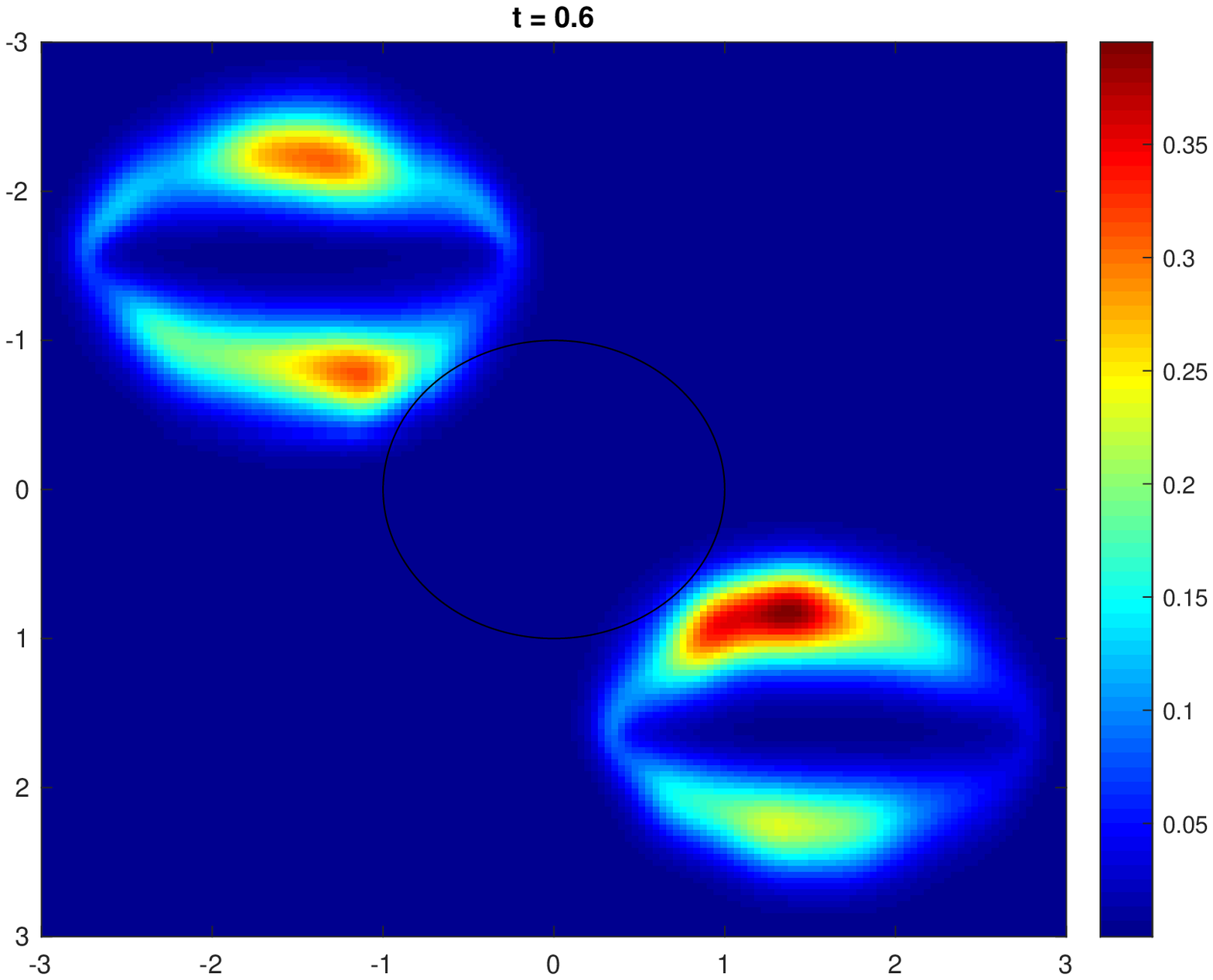}} \\
   \vskip -0.3truecm
       \scalebox{0.35}{\includegraphics{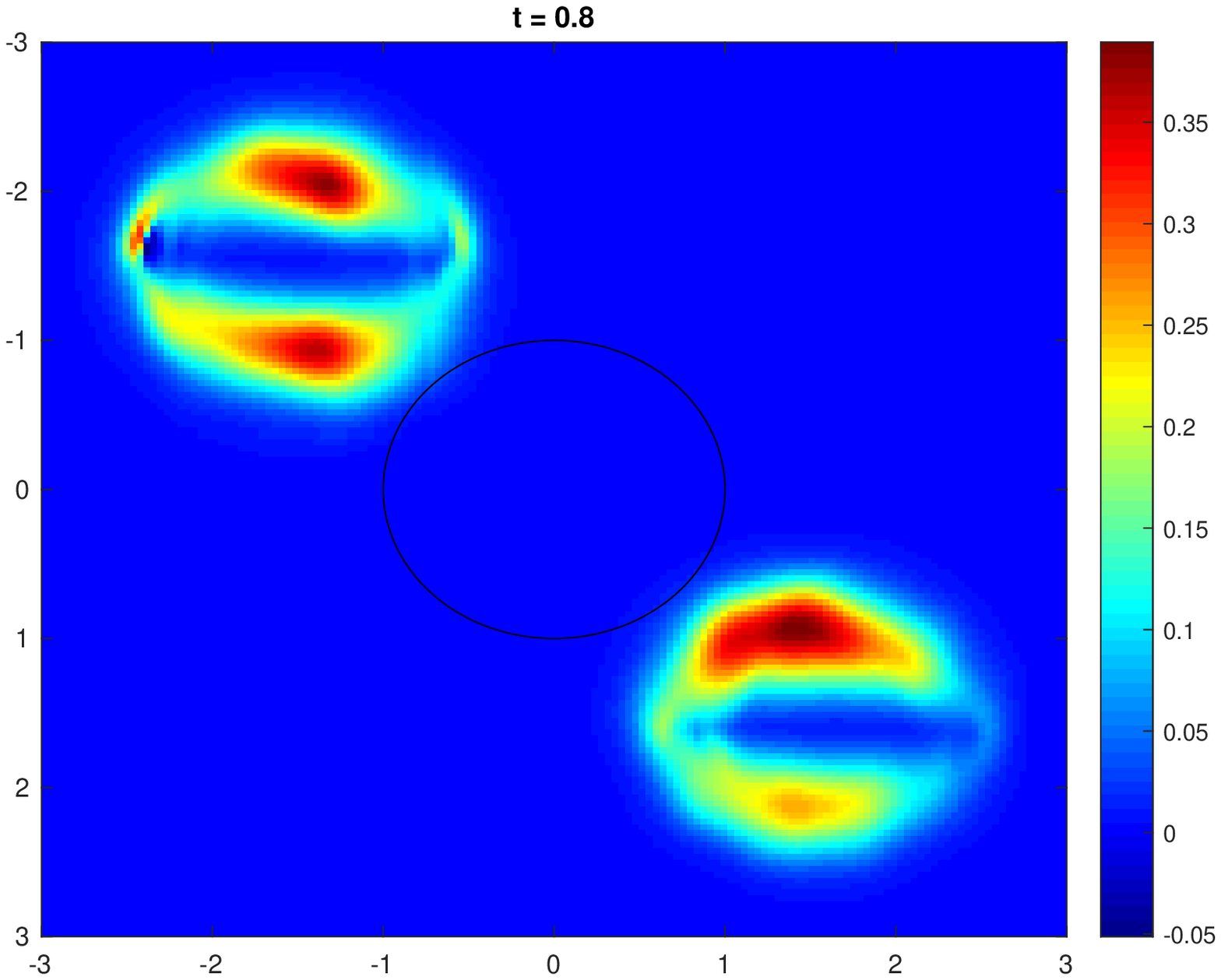}} 
         \scalebox{0.35}{\includegraphics{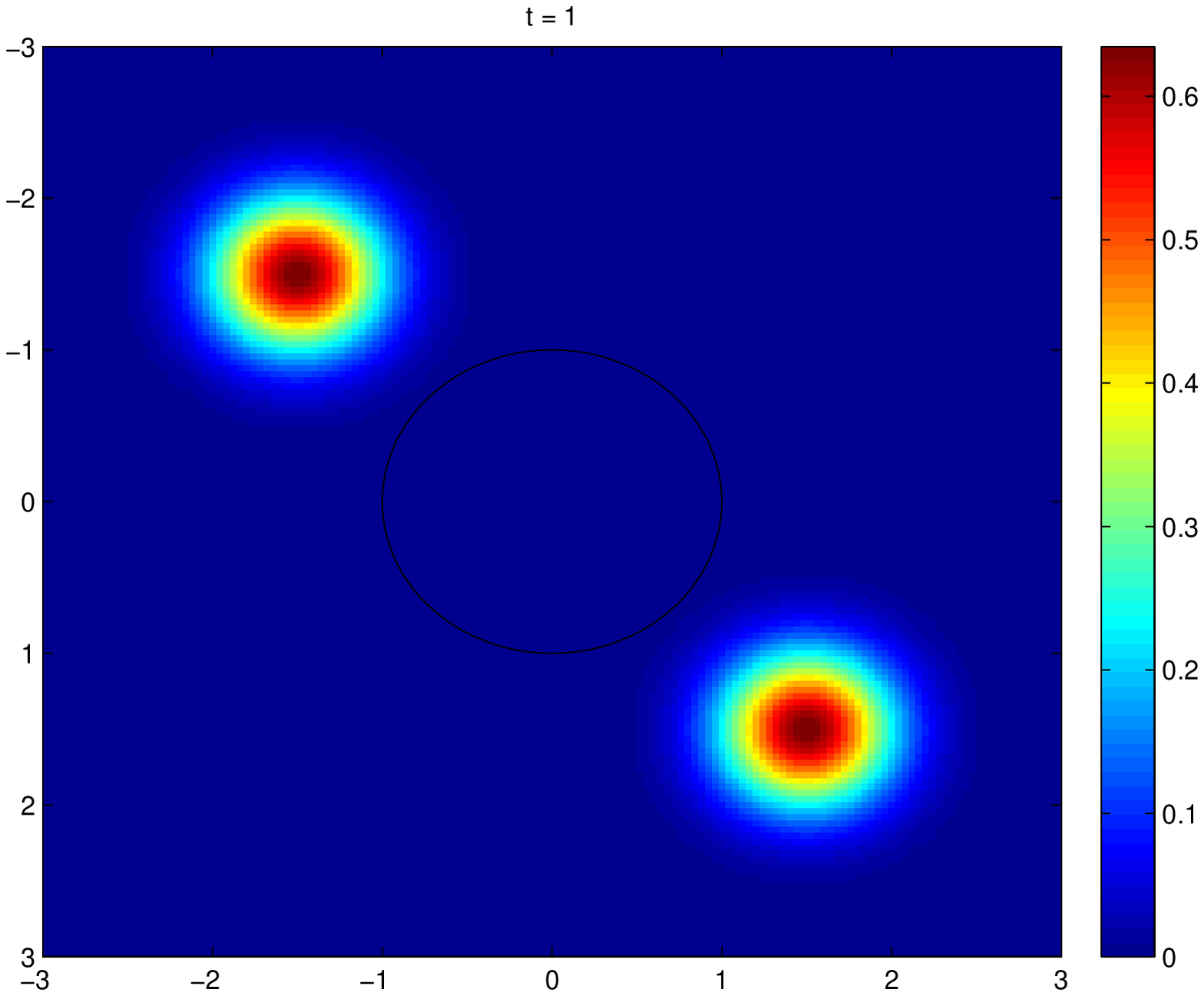}} \\
\caption{\small  The distribution $\rho(t,x) $ generating the convection on the mass for different $t = 0,0.2,0.4,0.6,0.8,1.0$. The black circle is the boundary of $B_1(0)$.}\label{exp_18}
     \end{center}
 \end{figurehere}

We can see the competing nature of the Hamiltonian, where one part of the Hamiltonian tries to drift the mass to the ball inward along one direction, while the other part of the Hamiltonian tries to drift the mass away along another direction.  This tears each mass lump apart into two lumps. Although the problem has not been fully understood mathematically, the numerical behavior of the solution shows the competing nature of a differential game problem in the mean field setting. 
\end{example}

 \section{Discussions}
 To summarize, we propose a generalized Hopf formula for potential mean field games. Our algorithm inherits main ideas in optimal transport on graphs and the Hopf formula for state-dependent optimal control problems. 
 
 Compared to the existing methods, the advantage of the proposed algorithm is three fold. First, the Hopf formula in density space introduces a minimization with variables depending on solely spatial grids. It has a lower complexity than the original optimal control problem. Second, the Hopf formula gives a simple parameterization for boundary problems in NE. This parameterization helps us design a simple first-order gradient descent method. This property allows us to compute the case of nonconvex Hamiltonians efficiently. Finally, our spatial discretization follows the dual of optimal transport on graphs. Hence, it is approximately discrete time reversible. This property conserves the primal-dual structure of potential mean field games.   
 
We notice that the Hopf formula in density space appears to {go beyond} monotonicity conditions and give legitimate numerical results, as shown in Section \ref{sec5}. 
Although it is beyond the scope of this paper, it is interesting to search for the precise conditions for the validity of the Hopf formula. Also, our current study only considers potential games without noise perturbations in players' decision processes. We will extend it to compute NEs for general non-potential games in future work. 
\bibliographystyle{abbrv}
\bibliography{HF}

\end{document}